\documentclass[a4paper,12pt]{article}

\usepackage[active]{srcltx}
\usepackage{hyperref}

\usepackage{amsmath,amsthm,amssymb,mathrsfs,latexsym,amsfonts,empheq}
\usepackage{graphicx,psfrag,epsfig}
\usepackage{bbm}
\usepackage[english]{babel}
\usepackage[latin1]{inputenc}
\usepackage{a4wide}

\newtheorem{theorem}{Theorem}
\newtheorem{lemma}[theorem]{Lemma}
\newtheorem{proposition}[theorem]{Proposition}

\newtheorem{corollary}[theorem]{Corollary}

\theoremstyle{definition}

\newtheorem{remark}{\it Remark}
\newtheorem{example}{Example}

\newcounter{paraga}[section]

\newtheorem{Main}{Theorem}

\def\MP{\,{<\hspace{-.5em}\cdot}\,}
\def\SP{\,{>\hspace{-.3em}\cdot}\,}
\def\PM{\,{\cdot\hspace{-.3em}<}\,}
\def\PS{\,{\cdot\hspace{-.3em}>}\,}
\def\EP{\,{=\hspace{-.2em}\cdot}\,}

\def\PE{\,{\cdot\hspace{-.2em}=}\,}
\def\N{\mathbb N}
\def\C{\mathbb C}
\def\Q{\mathbb Q}
\def\R{\mathbb R}
\def\T{\mathbb T}
\def\A{\mathbb A}
\def\Z{\mathbb Z}
\def\demi{\frac{1}{2}}

\begin{document}

\begin{titlepage}
  \title{\LARGE{\textbf{Hamiltonian perturbation theory for
        ultra-differentiable functions}}}%
  \author{Abed Bounemoura and Jacques F{\'e}joz\\
    PSL Research University\\
    (Universit{\'e} Paris-Dauphine and Observatoire de Paris)}
\end{titlepage}

\maketitle

\begin{abstract}
  Some scales of spaces of ultra-differentiable functions are
  introduced, having good stability properties with respect to
  infinitely many derivatives and compositions. They are well-suited
  for solving non-linear functional equations by means of hard
  implicit function theorems. They comprise Gevrey functions and thus,
  as a limiting case, analytic functions. Using majorizing series, we
  manage to characterize them in terms of a real sequence $M$ bounding
  the growth of derivatives.

  In this functional setting, we prove two fundamental results of
  Hamiltonian perturbation theory: the invariant torus theorem, where
  the invariant torus remains ultra-differentiable under the
  assumption that its frequency satisfies some arithmetic condition
  which we call BR$_M$, and which generalizes the Bruno-R{\"u}ssmann
  condition; and Nekhoroshev's theorem, where the stability time
  depends on the ultra-differentiable class of the pertubation,
  through the same sequence $M$. Our proof uses periodic averaging,
  while a substitute of the analyticity width allows us to bypass
  analytic smoothing.

We also prove converse statements on the destruction of invariant tori and on the existence of diffusing orbits with
  ultra-differentiable perturbations, by respectively mimicking a construction of Bessi (in the analytic category) and Marco-Sauzin (in the Gevrey non-analytic category). When the perturbation space
  satisfies some additional condition (we then call it
  \emph{matching}), we manage to narrow the gap between stability
  hypotheses (e.g. the BR$_M$ condition) and instability hypotheses,
  thus circumbscribing the stability threshold.

  The formulas relating the growth $M$ of derivatives of the
  perturbation on the one hand, and the arithmetics of robust
  frequencies or the stability time on the other hand,
  bring light to the competition between stability properties of
  nearly integrable systems and the distance to integrability. Due to
  our method of proof using width of regularity as a regularizing
  parameter, these formulas are closer to optimal as the the
  regularity tends to analyticity.
\end{abstract}

\clearpage
\tableofcontents

\section{Introduction}\label{sec:intro}

\subsection{High regularity in perturbation theory}
\label{sec:gintro}

The purpose of this article is two-fold. First, we introduce and study
the general properties of a class of ultra-differentiable functions
well-suited for solving non-linear functional equations with iterative
methods. Second, within this framework, we generalize the fundamental
results of Hamiltonian perturbation theory such as the KAM theorem on
the preservation of invariant tori or the Nekhoroshev theorem on the
long-time stability of solutions; we also prove ``converse" statements
such as the destruction of tori or the construction of unstable orbits
(the so-called Arnold diffusion).

\bigskip

First recall that real-analytic functions in $m$ variables are
characterized by a growth of their derivatives of order $s^{-|k|}|k|!$
for $k\in \N^m$ and some analyticity width $s>0$; in the periodic
case, this is equivalent to a decay of Fourier coefficients of order
$e^{-s|j|}$ for $j \in \Z^m$.

In his foundational paper~\cite{Kol54}, Kolmogorov proved that an
invariant torus in a real-analytic non-degenerate integrable
Hamiltonian system can be preserved, as a real-analytic embedded
torus, under an arbitrary small real-analytic perturbation, provided
the frequency vector satisfies a Diophantine condition. This theorem
of Kolmogorov has generated a tremendous amount of work, and new
proofs of more general statement appeared. Among them, still in the
analytic case, by using very accurate approximation results of
analytic functions by polynomials, R{\"u}ssmann (\cite{Rus01}) was able to
improve the arithmetic condition, and get what is now called the
Bruno-R{\"u}ssmann condition (BR-condition for short). This condition,
which was introduced earlier (in a different form) by Bruno
(\cite{Bru71},~\cite{Bru72}) in the Siegel linearization problem (a
different problem involving small divisors), is known to be optimal in
the Siegel problem in complex dimension one (this is a celebrated
result of Yoccoz~\cite{Yoc88},~\cite{Yoc95}). Whether BR is optimal in
higher dimension or in the Hamiltonian problem is open. A different
proof of the Kolmogorov theorem with the BR-condition was then given
in~\cite{BF13},~\cite{BF14}; instead of approximating the
real-analytic perturbation by a polynomial, the idea consists in
approximating the frequency vector by suitable periodic vectors. Let
us point out that all these proofs crucially use a complex extension
of the domain and of the Hamiltonian in order to use methods of
complex analysis. Also, even though it is not known if the
BR-condition is optimal, a construction of Bessi~\cite{Bes00},
inspired by the theory of Arnold diffusion, shows that BR cannot be
improved much, if at all.

A second fundamental result is due to Nekhoroshev, who showed
in~\cite{Nek77},~\cite{Nek79} that all solutions of a perturbed
``generic" real-analytic integrable Hamiltonian system are stable for
an exponentially long interval of time. This generic class of
integrable Hamiltonians, which are known as steep, includes as a
particular case the quasi-convex functions for which the proof of
Nekhoroshev theorem greatly simplifies. As a consequence, better
quantitative results were obtained
in~\cite{Loc92},~\cite{LN92},~\cite{LNN94} and in~\cite{Pos93} for the
quasi-convex case. Together with the improvements obtained
in~\cite{BM11} and in~\cite{ZZ17} and the examples of analytic Arnold
diffusion in~\cite{Bes96}, ~\cite{Bes97} and~\cite{Zha11}, all these
results determine precisely the time-scale of stability. The
quantitative result of the quasi-convex case was recently extended to
the general steep case in~\cite{GCB16}; however, no examples of Arnold
diffusion are known for steep non-convex integrable Hamiltonians. Let
us also remark that even though linear integrable Hamiltonians are
non-steep, they are stable for a long interval of time provided the
(constant) frequency is non-resonant; this is a well-known result
assuming the frequency to be Diophantine, but in the general case this
was observed in~\cite{Bou12} along with the fact that the stability
estimates thus obtained are optimal as examples of Arnold diffusion
with controlled speed are quite easily constructed in this context.

\bigskip

Next, given a real parameter $\alpha \geq 1$, allowing a growth of the
derivatives of order $s^{-|k|}|k|!^\alpha$ for $k \in \N^m$ or,
equivalently, a decay of Fourier coefficients of order
$e^{-(s|j|)^{1/\alpha}}$ for $j \in \Z^m$ in the periodic case, one is
led to consider $\alpha$-Gevrey functions of $m$ variables, which thus
reduce to real-analytic functions when $\alpha=1$. The ``width''
$s>0$ extends the significance of the analyticity width of analytic
functions. Since the introduction by Gevrey of the class of functions
now baring his name (\cite{Gev18}), there has been a huge amount of
works on Gevrey functions, mainly for PDEs, but also more recently in
other fields, including dynamical systems.

The theorem of Kolmogorov, under the classical Diophantine condition,
was extended to this Gevrey context by Popov~\cite{Pop04}. His proof,
which is based on approximation of Gevrey functions by real-analytic
ones, did not allow him to reach weaker arithmetic condition even
though one would expect the statement to be true under a condition
which generalizes the BR-condition. Recently, we introduced
in~\cite{BFa17} such a condition that we called the
BR$_\alpha$-condition and under which an invariant torus is preserved
as an $\alpha$-Gevrey torus for a non-degenerate close to integrable
$\alpha$-Gevrey system. The proof in~\cite{BFa17} was direct, in the
sense that in the real-analytic case it does not use any complex
extension, and in the Gevrey case it does not use analytic
approximation. Actually, one of the main difficulties was to obtain
functional estimates (and in particular product, derivatives and
composition estimates) which were not known and were clearly required
if one wants to implement an iterative scheme (typical of KAM) within
the $\alpha$-Gevrey category. Indeed, of particular importance was to
obtain a norm estimate for the composition of a function by a
diffeomorphism in which the loss of ``width" is arbitrarily small
provided the diffeomorphism is arbitrarily close to the identity;
such an estimate is needed if one insists on having some ``width''
parameter left after infinitely many compositions. These crucial
estimates were obtained in~\cite{BFa17}, and with further work using
the method introduced in~\cite{BF13},~\cite{BF14} the result was
achieved. Also we observed that Bessi's example could be modified to
give a necessary condition close to the sufficient
BR$_\alpha$-condition; in particular, invariant torus with a frequency
which satisfy the BR-condition can be destroyed by an arbitrarily
small perturbation which is $\alpha$-Gevrey for any $\alpha>1$, and
this is in sharp contrast with some other small divisors problem (as
for instance the formal Gevrey Siegel problem considered
in~\cite{CM00} and~\cite{Car03}).

As for the Nekhoroshev theorem, the extension to $\alpha$-Gevrey
quasi-convex Hamiltonians was achieved in~\cite{MS02}. Here precise
composition estimates were not needed (even though the authors were
aware that their composition result was not accurate). Indeed, it is
known since the work of Lochak that for quasi-convex Hamiltonians, one
only needs a normal form at periodic frequency and thus, from an
analytical point of view, one is dealing with the problem of
eliminating the time-dependence in a slow-fast periodic system. This
problem always has a formal solution which of course may be divergent:
in~\cite{RS96} it was proved that for an analytic system, the
associated formal series has a Gevrey character, and this was extended
(in a Hamiltonian setting) for an arbitrary $\alpha$-Gevrey
system. Once one knows the existence of a formal solution with such a
Gevrey Character, techniques of re-summation allow to find a Gevrey
(convergent) normal form with an exponentially small remainder, and
this is all what is needed to prove the Nekhoroshev theorem for
quasi-convex Hamiltonians. Examples of Arnold diffusion were also
obtained in~\cite{MS02} but in the Gevrey non-analytic case
$\alpha>1$, as the method uses the existence of bump
functions. However, even though the improvement of~\cite{BM11} does
extend to the more general Gevrey setting, for the moment this is not the case
of~\cite{ZZ17} and thus the optimal time-scale of stability in the
Gevrey case is not yet completely determined. Nekhoroshev estimates
have been obtained also in the general steep case for Gevrey
Hamiltonians (see~\cite{Bou11}), yet the quantitative result is very
poor compared to the analytic case studied in~\cite{GCB16}. Finally,
for Gevrey linear integrable Hamiltonians, precise and essentially
optimal stability and instability result are proved in~\cite{Bou13a}.

\bigskip

More generally, for a sequence of positive real numbers
$M=(M_l)_{l \in \N}$, we can consider the class of smooth functions
having a growth of derivatives of order $s^{-|k|}M_{|k|}$ for
$k \in \N^m$ and some $s>0$ which we still call a ``width". This
defines a class of ultra-differentiable functions associated to $M$,
or simply $M$-ultra-differentiable functions. The real-analytic and
$\alpha$-Gevrey regularity are nothing but the particular cases where
respectively $M_l=l!$ and $M_l=l!^\alpha$. In terms of Fourier
coefficients for periodic functions, this amounts to a decay of order
$e^{-\Omega(s|j|)}$ for $j \in \Z^m$, for some suitable positive
increasing function $\Omega$; again the real-analytic and
$\alpha$-Gevrey regularity correspond to respectively $\Omega(y)=y$
and $\Omega(y)=y^{1/\alpha}$. It is one of our goal here to introduce
a class of ultra-differentiable functions in which non-linear problems
can be solved by an iterative method within this class. To do so, we
will restrict the allowed sequence $M$ by requiring two conditions to
be satisfied (see \S~\ref{sec:udiffintro} for the precise definition
of $M$-ultra-differentiable functions and these two conditions). The
first of these conditions (which we will call \eqref{H1}) is a kind of
regularity condition and it is classical and very often satisfied. In
particular it implies stability by product and by composition; for the
composition a loss of ``width" occur so in order to be able to perform
infinitely many compositions, we will have to show that the loss can
be arbitrarily small when the composition involves a diffeormorphism
arbitrarily close to the identity. The second condition (which we will
call \eqref{H2}) is, however, more restrictive as it requires the sequence
$M$ not to grow too fast at infinity; it implies not only stability by
derivation but moreover that the loss of width can be taken
arbitrarily small, and thus one may say that this condition amounts to
stability under `infinitely many derivations". This growth restriction
is quantified by a function which we will call $C$, and in good cases
it will be related to the function $\Omega$ which controls the decay
of Fourier coefficients in the periodic case (see \S~\ref{sec:COmega}
for more information on these functions $C$ and $\Omega$). The
functional estimates we obtained in~\cite{BFa17} in the Gevrey setting
will be generalized to this $M$-ultra-differentiable setting where $M$
satisfies these two conditions \eqref{H1} and \eqref{H2}.

As a first application, we will give the proof of KAM type theorems within this class of $M$-ultra-differentiable functions (see \S~\ref{sec:KAMintro} for precise statements). We will introduce an arithmetic condition, the BR$_M$-condition, and under which an invariant torus is preserved as an $M$-ultra-differentiable torus for a non-degenerate close to integrable $M$-ultra-differentiable Hamiltonian system (see \S~\ref{sec:arithmetic} for the definition of this arithmetic condition). This BR$_M$-condition reduces to the BR$_\alpha$-condition when $M_l=l!^\alpha$ so this includes the main result of~\cite{BFa17} but it encompasses many more cases. Let us point out that this BR$_M$-condition does not depend directly on $M$ but rather on the associated function $C$ that was mentioned above. Again, Bessi's example could be modified to give a necessary condition for the preservation of an invariant torus, yet this necessary condition now involves the function $\Omega$ and not $C$ and thus only in the good cases (where these two functions are related; this will be called \emph{matching} in the sequel) this necessary condition shows that the sufficient condition cannot be improved too much, if any.

As a second application, we will extend Nekhoroshev theory to this
class of $M$-ultra-differentiable functions (see \S~\ref{sec:Nekintro}
for precise statements). As we already explained, in the analytic case
we have a good knowledge, the main issue being the Arnold
diffusion for steep non-convex systems. In the more general Gevrey
case, the time-scale of stability is not yet completely determined in
the quasi-convex case but also only poor quantitative statement are
known in the general steep case. In the ultra-differentiable setting,
basically we will be able to extend the knowledge we have for the
general Gevrey case but not the more precise knowledge of the analytic
case. First, for linear integrable Hamiltonians, we will give accurate
and essentially optimal stability and instability result, which
completely extends the analytic and Gevrey case; exactly like for the
KAM type results, the stability (that is, positive) result involves
the $C$ function whereas the diffusion (that is, negative) result
involves the $\Omega$ function and so one has to consider the result
to be accurate only for matching sequences. Then, for quasi-convex
integrable Hamiltonians, our stability result again will be a complete
generalization of the analytic and Gevrey case, however, exactly like
for the Gevrey case, the improvement of~\cite{ZZ17} cannot be reached
and the construction of Arnold diffusion, following~\cite{MS02},
requires bump function and thus a further assumption of
non-quasi-analyticity. Finally, in the steep case, a result
generalizing~\cite{Bou11} will be obtained but the quantitative result
will be again quite far from the more accurate analytic
case~\cite{GCB16}.
 
\subsection{Ultra-differentiable functions}\label{sec:udiffintro} 

Recall that $n\geq 1$ is an integer, $\T^n=\R^n / \Z^n$ and let
$D \subseteq \R^n$ be the open ball of radius one centered at the origin. For a small parameter $\varepsilon\geq 0$, we consider a Hamiltonian function
$H : \T^n \times D \rightarrow \R$ of the form
\begin{equation}\label{Ham1}
\begin{cases}\tag{$*$}
  H(\theta,I)= h(I) +\varepsilon f(\theta,I), \\
  \nabla h(0):=\omega_0 \in \R^n.
\end{cases}  
\end{equation}
We shall assume that the Hamiltonian $H$ is
\textit{ultra-differentiable} on $\T^n \times D$, in the following
sense. Let us fix a sequence $M:=(M_l)_{l \in \N}$ of positive real
numbers with $M_0=M_1=1$. The function $H$ is said to be
ultra-differentiable with respect to $M$ (or $M$-ultra-differentiable)
if $H$ is smooth and there exists $s>0$ such that, using multi-indices
notation (see \S~\ref{sec:udiff}),
\begin{equation}\label{defn}
  |H|_{M,s} := c\sup_{(\theta,I) \in \T^n \times
    D}\left(\sup_{k \in
      \N^{2n}}\frac{(|k|+1)^2{s}^{|k|}|\partial^k
      H(\theta,I)|}{M_{|k|}}\right)<\infty, \quad c:=4\pi^2/3.  
\end{equation}
The space of such Hamiltonians will be denoted by
$\mathcal{U}_{M,s}(\T^n \times D)$ or, more simply,
$\mathcal{U}_s(\T^n \times D)$. This definition can be extended to
vector-valued functions $X: \T^n \times D \rightarrow \R^p$ by setting
\begin{equation}\label{defn2}
  |X|_{M,s} := c\sup_{(\theta,I) \in \T^n \times
    D}\left(\sup_{k \in
      \N^{2n}}\frac{(|k|+1)^2{s}^{|k|}|\partial^k
      X(\theta,I)|_1}{M_{|k|}}\right)<\infty 
\end{equation}
where $|\,.\,|_1$ is the $l_1$-norm of vectors in $\R^p$, or the sum
of the absolute values of the components, thus defining the space
$\mathcal{U}_s(\T^n \times D,\R^p)$. As a rule, we will use the
$l_1$-norm for vectors, so for simplicity we shall write
$|\,.\,|_1=|\,.\,|$.

Condition~\eqref{defn} may be imaged as some modification of the
Taylor series of $H$ having radius of convergence at least $s$. Thus, the
parameter $s$ extends the concept of width of analyticity of an
analytic function (see example~\ref{exempleMalpha}), and will be
called the (ultra-differentiable) \emph{width}. To emphasize its role,
we shall also say that $H$ is $(M,s)$-ultra-differentiable
if~\eqref{defn} holds. These classes of ultra-differentiable functions
we are considering here are usually called Denjoy-Carleman classes (of
Roumieu type) in the literature; we refer to~\cite{Thi08} for a nice
survey.

General properties of these ultra-differentiable norms (under certain
assumptions we will introduce below) are described in
\S~\ref{sec:udiff}. In particular we explain there the (inessential)
role of the factor $(|k|+1)^2$ and the normalizing constant $c>0$
in~\eqref{defn}.

Let us now give the main example of ultra-differentiable class of
functions.

\begin{example}\label{exempleMalpha}
Given $\alpha \geq 1$, let us define the sequence 
\begin{equation}\label{Malpha}
M_\alpha:=(l!^\alpha)_{l \in \N}.
\end{equation}
Observe that a function is $M_\alpha$-ultra-differentiable if and only
if it is $\alpha$-Gevrey, and thus it is $M_1$-ultra-differentiable if
and only if it is real-analytic, in which case the parameter $s>0$ is
the width of analyticity. (The case when $0 \leq \alpha < 1$
corresponds to entire functions and will not be considered in this
work.)  More examples of sequences will be given in
\S~\ref{sec:arithmetic}.
\end{example}

Without further information on the sequence $M$, one cannot expect
much structure on the space of $M$-ultra-differentiable functions. For
our purposes here, we will have to make two assumptions. But first
associated to the sequence $M$ we define three other sequences
$\mu:=(\mu_l)_{l \in \N}$, $N:=(N_l)_{l \in \N}$ and
$\nu:=(\nu_l)_{l \in \N}$ by
\begin{equation}\label{sequences}
\mu_l:=\frac{M_{l+1}}{M_l}, \quad N_l:=\frac{M_l}{l!}, \quad \nu_l:=\frac{N_{l+1}}{N_l}=\frac{\mu_l}{l+1}. 
\end{equation} 
As $M_0=M_1=1$, we have $N_0=N_1=1$ and $\mu_0=\nu_0=1$. Clearly, the knowledge of one of these sequences determines the other three. The two assumptions we will require are the following:
\begin{align}
  &\mbox{The sequence $N$ is log-convex, i.e., the
    sequence $\nu$ is non-decreasing.} \tag{H1}\label{H1}\\
  &\mbox{The sequence $\mu$ is sub-exponential, i.e., }\lim_{l
    \rightarrow + \infty}l^{-1}\ln(\mu_l)=0. \tag{H2}\label{H2}
\end{align}
Let us briefly describe the role of these two assumptions.

The log-convexity condition~\eqref{H1} is a very classical assumption
in the literature, in particular it implies that the space of
$M$-ultra-differentiable function is closed under product and
composition; as a matter of fact, for this to be true slightly weaker
conditions are required (see Lemma~\ref{lemmeprodavant} and
Lemma~\ref{lemmecompavant} in respectively \S~\ref{sec:products} and
\S~\ref{sec:composition}). Clearly, if $\nu$ is increasing then so is
$\mu$ thus the log-convexity of $N$ implies the log-convexity of $M$
(but the reverse implication is not true in general). Let us also
point out that \eqref{H1} and the normalization $N_0=M_0=1$ imply
\[ N_l \geq 1, \quad M_l \geq l!, \quad l \in \N \]
with the limiting case where $N_l=1$, $M_l=l!$, corresponding to the space
of analytic functions. 

Up to our knowledge, the sub-exponential condition \eqref{H2} hasn't
been used so far in the literature. This condition implies that our
space of functions is closed under derivation, which is easily seen to
be equivalent to the condition
\[ \sup_{l \in \N} \mu_l^{\frac{1}{l}}=\sup_{l \in \N}
\left(\frac{M_{l+1}}{M_l}\right)^{\frac{1}{l}} < +\infty, \]
but it is definitely stronger: instead of just requiring that the
sequence $\left(\mu_l^{1/l}\right)_{l \in \N}$ is bounded, \eqref{H2}
demands that this sequence converges to $1$.

However, \eqref{H2} is not an inordinate property. Indeed, a condition
commonly assumed in the literature on ultra-differentiable functions
is the following, so-called \emph{moderate growth}:
\begin{equation}
  \label{eq:MG}%
  \sup_{l,j \in \N}
  \left(\frac{M_{l+j}}{M_lM_j}\right)^{\frac{1}{l+j}} < +\infty \tag{MG}
\end{equation}
As Lemma~\ref{lm:MG} shows (Appendix~\ref{sec:MG}), this condition is
stronger than our hypothesis \eqref{H2}.

Actually, one can image \eqref{H2} as the guarantee that our functional
space is ``stable by infinitely many derivations". Indeed, in view of
our choice of norms, the condition \eqref{H2} is exactly the one that
ensures we have an analogue of the Cauchy estimates for analytic
functions: if $f$ is $(M,s)$-ultra-differentiable, then its derivative
$f'$ is $(M,s')$-ultra-differentiable for any $s'<s$ and we can bound
the $(M,s')$-norm of $f'$ in terms of the $(M,s)$-norm of $f$ and a
function of $s-s'$, that we call $C(s-s')$ ($C$ stands for Cauchy). As
$s'$ converges to $s$, this function blows up and its asymptotics
plays an important role in the sequel. It will be defined and briefly
studied in the next section, along another function, which we call
$\Omega$ (this function was already mentioned in the introduction) and
which can actually be defined without any assumption on the sequence
$M$.

\subsection{Functions C and $\Omega$}\label{sec:COmega}

Given $0 < \sigma < 1$, we define the Cauchy function $C=C_M$ by
\begin{equation}\label{Cfunction}
  C(\sigma) := \sup_{l \in \N}\mu_l e^{-\sigma l} = \sup_{l \in
    \N}\frac{M_{l+1}}{M_l} e^{-\sigma l}
\end{equation}
where the sequence $\mu$ was defined in~\eqref{sequences}: that
$C(\sigma)$ is finite for all $\sigma>0$ follows from (and in fact is
equivalent to) \eqref{H2} which guarantees that $\mu$ has a sub-exponential
growth. Let us first remark that $\nu_l \geq \nu_0=1$ (this follows from \eqref{H1} and the normalization) and thus
\begin{equation}\label{Csigma}
  C(\sigma) = \sup_{l \in \N}(l+1)\nu_le^{-\sigma l} \geq \sup_{l \in
    \N}(l+1)e^{-\sigma l} \geq (e\sigma)^{-1}.  
\end{equation}
Clearly, the supremum in~\eqref{Cfunction} is in fact a maximum and
the function $C$ is thus continuous on $]0,1[$. Besides, there exists
some $0<\bar{\sigma}<1$ which depends only on the sequence $M$ such
that $C(\sigma)=1$ for any $\bar{\sigma} \leq \sigma <1$ (we will
assume that $\bar{\sigma}$ is minimal with this property); the
interesting limit is when $\sigma$ goes to zero (in the discussion
above, this corresponds to $s'=s(1-\sigma)$ and the limit
$s'\rightarrow s$ amounts to $\sigma \rightarrow 0$). For
$0 < \sigma \leq \bar{\sigma}$, the function $C$ is then decreasing,
and so the restriction
\[ C : \; ]0,\bar{\sigma}] \rightarrow [1,+\infty[ \]
is a homeomorphism, which has a continuous decreasing functionnal
inverse
\[ C^{-1} : [1,+\infty[ \rightarrow ]0,\bar{\sigma}]. \]
The behavior of $C(\sigma)$, as $\sigma$ tends to zero, translates
into the behavior of $C^{-1}(y)$, as $y$ tends to infinity.

Let us now introduce another function, this one classically attached
to the sequence $M$, which we will call $\Omega=\Omega_M$, and which
is defined by
\begin{equation}\label{Ofunction}
  \Omega(y):=\ln\left(\sup_{l \in \N}y^lM_l^{-1}\right) = \sup_{l \in
    \N}\left(l\ln y-\ln M_l\right). 
\end{equation}
One can show (see \cite{CC94} or~\cite{Thi03}) that this defines a
function
\[ \Omega : [0,+\infty[ \rightarrow [0,+\infty[ \]
which is continuous, constant equal to zero for $y \leq 1$ and
strictly increasing for $y \geq 1$. Clearly, $\Omega(y)$ grows faster
than $\ln\left(y^l\right)$ for any $l \in \N$ as $y$ goes to
infinity, but also
\begin{equation}\label{Omegalog}
  \lim_{y \rightarrow +\infty} \frac{\Omega(y)}{\ln(y)}=+\infty.
\end{equation}
For a periodic function $f \in \mathcal{U}_s(\T^n)$, letting
$f=\sum_{j \in \Z^n}f_je_j$ be its Fourier expansion, one easily
checks that $\Omega$ controls the decay of the Fourier coefficients in
the sense that
\begin{equation}\label{decay}
|f_j| \leq \exp(-\Omega(s|j|)), \quad j \in \Z^n.
\end{equation}
Here too, it is the asymptotic of the function $\Omega$ at infinity
which is
of interest to us, so to make things more precise, let us
introduce the following definition.

Let us say that two real-valued functions $f$ and $g$ of one variable,
defined in the neighborhood of $u$, $u \in \{0,+\infty\}$, are
\emph{equivalent up to scalings} (at $u$) if there exist positive
constants $a$ and $b$ such that
\[ f(a x) \leq g(x) \leq f(b x), \quad x \rightarrow
u.  \]
In the sequel we shall simply write $f(x) \asymp g(x)$ without referring
to the precise value $u$.

\begin{example}\label{exempleMalpha2}
Let us come back to the Example~\ref{exempleMalpha} where
\begin{equation*}
M_\alpha=(l!^\alpha)_{l \in \N}, \quad \mu_\alpha=((l+1)^\alpha)_{l \in \N}.   
\end{equation*}
Clearly, \eqref{H1} and \eqref{H2} are satisfied, and simple computations lead to
\[ C(\sigma) \asymp \sigma^{-\alpha}, \quad C^{-1}(y) \asymp y^{-{1/\alpha}}, \quad \Omega(y) \asymp y^{1/\alpha}. \]
\end{example} 

In the above example, one remarks that the inverse of $\Omega$ is
equivalent up to scalings (at infinity) to the functional inverse of
$C$. Let us give a name to sequences with such a property: a sequence
$M$ will be called \emph{matching} if
\begin{equation}\label{matching}
C^{-1} \asymp 1/\Omega
\end{equation}
where the functions $C$ and $\Omega$ have been defined respectively
in~\eqref{Cfunction} and~\eqref{Ofunction}. In the sequel, matching
sequences will play a special role: these will be the sequences for
which we will be able to compare in the sharpest way stability results
(KAM, Nekhoroshev) to instability results (destruction of tori, Arnold
diffusion).

\subsection{An arithmetic condition for ultra-differentiable functions}\label{sec:arithmetic}

Given $\omega_0 \in \R^n$, define the function
\begin{equation}\label{eqpsi}
\Psi_{\omega_0} : [1,+\infty) \rightarrow [\Psi_{\omega_0}(1),+\infty], \quad
  Q \mapsto \max \{|k\cdot\omega_0|^{-1}\; | \; k \in \Z^n, \; 0 <
  |k|\leq Q\}.
\end{equation}
This function $\Psi_{\omega_0}$ measures the size of the so-called small denominators. We will impose a condition that prevents $\Psi_{\omega_0}$ from growing too fast at infinity; the precise condition will actually depends on the sequence $M$ through $C^{-1}$, the inverse of the Cauchy function $C$ we introduced in the previous section.

But first let us recall that the function $\Psi_{\omega_0}$ is
non-decreasing (thus with a countable number of discontinuities) and
piecewise constant. In the sequel, it will be more convenient to work
with a continuous version of $\Psi_{\omega_0}$: it is not hard to
prove (see, for instance, Appendix A of~\cite{BF13}) that one can find
a continuous non-decreasing function
$\Psi : [1,\infty) \to [\Psi(1),+\infty)$ such that
$\Psi(1)=\Psi_{\omega_0}(1)$ and
\begin{equation}\label{foncpsi}
\Psi_{\omega_0}(Q) \leq \Psi(Q) \leq \Psi_{\omega_0}(Q+1), \quad Q \geq 1.
\end{equation}
Clearly, for all $k \in \Z^n\setminus\{0\}$, we still have
\[ |k\cdot \omega_0| \geq 1/\Psi(|k|) \]
and in the condition we will introduce below, one may use $\Psi$ instead of $\Psi_{\omega_0}$. 

Let us now define the function
$$\Delta : [1,+\infty) \to [\Psi(1),+\infty),
\quad Q \mapsto Q\Psi(Q)$$
which is continuous and increasing; $\Delta$ has therefore a functional inverse
$$\Delta^{-1} : [\Psi(1),+\infty) \to [1,+\infty), \quad \Delta^{-1} \circ
\Delta =\Delta \circ \Delta^{-1}= \mathrm{Id}$$
which is also continuous and increasing.

We are now ready to introduce our arithmetic condition adapted to
$M$-ultra-differentiable functions. We say that
$\omega_0 \in \mathrm{BR_M}$ if, for any $c>0$, 
\begin{equation}\label{BRM}
\int_{\Delta(1/c)}^{+\infty}C^{-1}(c\Delta^{-1}(x))\frac{dx}{x}<+\infty,
\tag{$\mathrm{BR_M}$} 
\end{equation}
where $C^{-1}$ is the functional inverse of the Cauchy function $C$
defined in~\eqref{Cfunction}. This definition is somewhat involved,
which calls for a few comments and examples.

First observe that the condition~\eqref{BRM} is unchanged if one
replaces $C^{-1}$ by a function which is equivalent to it up to
scalings. In particular, instead of using the function $\Psi$, we may
as well use $\Psi_{\omega_0}$ (the role of $\Delta^{-1}$ will then be
played by a generalized inverse of $Q \mapsto Q\Psi_{\omega_0}(Q)$) as
one can easily check from~\eqref{foncpsi} (see Appendix A
of~\cite{BF13} for more details). The set \ref{BRM} may be empty
(see Example~\ref{exempleexpracine} below) or not. Before giving
concrete examples, recall that the set $D_\tau$ of $\tau$-Diophantine
vectors ($\tau\geq n-1$) is the set of vectors for which there exists
$\gamma>0$ such that $\Psi_\omega(Q) \leq Q^\tau/\gamma$ for all
$Q\geq 1$. It is well-known and easy to prove that $D_\tau$ is
non-empty, and has full measure if $\tau>n-1$.

\begin{example}\label{exempleMalpha3}
Let us come back to the Examples~\ref{exempleMalpha} and~\ref{exempleMalpha2}, where
\begin{equation*}
M_\alpha=(l!^\alpha)_{l \in \N}, \quad C^{-1}(y) \asymp y^{-1/\alpha}.   
\end{equation*}
As we already pointed out, $M_\alpha$ is matching. In this case,
condition~\eqref{BRM} is satisfied if and only if
\begin{equation}
\int_{\Delta (1)}^{+\infty}\frac{1}{x(\Delta^{-1}(x))^{1/\alpha}}dx <
  \infty \tag{$\mathrm{A}_\alpha$}   
\end{equation}
which is a condition we introduced in~\cite{BFa17}, and where it is
shown to be equivalent to the $\alpha$-Bruno-R{\"u}ssmann condition
\begin{equation}\label{alphaBR}
\int_{1}^{+\infty}\frac{\ln(\Psi_{\omega_0}(Q))}{Q^{1+1/\alpha}}dQ <
\infty. \tag{$\mathrm{BR}_\alpha$}
\end{equation}
Let us denote by $\mathrm{BR}_\alpha$ the set of vectors satisfying this latter condition. It is not hard to check that $\mathrm{BR}_\alpha$ is strictly included in $\mathrm{BR}_{\alpha'}$ if $\alpha' < \alpha$ and 
\[\cap_{\alpha\geq 1} \mathrm{BR}_\alpha \setminus \cup_{\tau \geq n-1} D_\tau\]  
is non-empty.
\end{example} 

\begin{example}\label{exempleMalphabeta}
Given $\alpha \geq 1$ and $\beta \geq 0$, let us now consider more generally the sequence $M_{\alpha,\beta}$ defined by
\begin{equation*}
(M_{\alpha,\beta})_0=0, \quad (M_{\alpha,\beta})_l=l!^\alpha\prod_{i=0}^l(\ln(e-1+i))^{\beta}, \quad l \geq 1.   
\end{equation*}
Clearly, $M_{\alpha,0}=M_\alpha$ and \eqref{H1} and \eqref{H2} are satisfied. We have
\[ \mu_j=(j+1)^\alpha(\ln(e+j))^\beta \]
and elementary computations show that
\[ C(\sigma) \asymp \sigma^{-\alpha}(\ln(\sigma^{-1}))^\beta, \quad C^{-1}(y) \asymp y^{-1/\alpha}(\ln(y))^{\beta/\alpha}  \]
but also
\[ \Omega(y) \asymp y^{1/\alpha}(\ln(y))^{-\beta/\alpha} \]
and so $M_{\alpha,\beta}$ is matching. In this case,
condition~\eqref{BRM} is satisfied if and only if
\begin{equation}
  \int_{\Delta
    (1)}^{+\infty}\frac{(\ln(\Delta^{-1}(x)))^{\beta/\alpha}}{x
    (\Delta^{-1}(x))^{1/\alpha}}dx  <  \infty.
  \tag{$\mathrm{BR}_{\alpha,\beta}$} 
\end{equation}
By extension, let us denote by $\mathrm{BR}_{\alpha,\beta}$ the set of
vectors satisfying this latter condition, then
$\mathrm{BR}_{\alpha,\beta}$ is strictly included in
$\mathrm{BR}_{\alpha',\beta'}$ if $\alpha' < \alpha$ or
$\alpha'=\alpha$ and $\beta'<\beta$, and
\[\cap_{\alpha\geq 1,\; \beta\geq 0}
\mathrm{BR}_{\alpha,\beta}=\cap_{\alpha\geq 1} \mathrm{BR}_\alpha
\setminus \cup_{\tau \geq n-1} D_\tau.\]   
\end{example} 

\begin{example}\label{exempleexplog}
  Let us now give an example in which the growth of the sequence $M$
  is wilder. Equivalently, it is sufficient to prescribe the sequence
  $\mu$, and let us choose
  \[ \mu_l=\exp((\ln(l))^2), \quad M_l \sim \exp(l(\ln(l))^2).  \]
  As before, one easily check that \eqref{H1} and \eqref{H2} are satisfied. One
  can then computes
  \[ C(\sigma) \asymp
  \exp\left(\left(\ln\sigma^{-1}\right)^2 +
    \ln\sigma^{-1}\ln\ln\sigma^{-1}\right) \] 
  and
  \[ C^{-1}(y) \asymp \exp\left(-(\ln y -\ln\ln y)^{1/2}\right). \]
  Here, one can check that the sequence is not matching. But observing
  for instance that given any $d>1$, for $y$ large one has
  \[ C^{-1}(y) \leq (\ln y)^{-d} \]
  then one easily sees that the set of vectors satisfying
  condition~\eqref{BRM} contains all Diophantine vectors (actually the
  inclusion is strict, though it is cumbersome to make the
  condition~\eqref{BRM} explicit in this case).
\end{example} 

To conclude, let us give another example in which the growth is so
wild that no vectors satisfy the condition~\eqref{BRM}.
  
\begin{example}\label{exempleexpracine}
  Let 
  \[ \mu_l=\exp(l^{1/2}), \quad M_l \sim \exp(l^{3/2}).  \]
  Here too, \eqref{H1} and \eqref{H2} are satisfied but it is easy to compute
  \[ C(\sigma) \asymp \exp\left(\frac{1}{4\sigma}\right), \quad
  C^{-1}(y) \asymp \frac{1}{4\ln y}. \]
  Again, this sequence is not matching. Recalling that
  $\Delta^{-1}(x)$ grows at most like $x^{1/n}$ when $x$ goes to
  infinity (this is an easy consequence of Dirichlet's box principle),
  one verifies that the integral in~\eqref{BRM} is always divergent,
  and hence the set~\eqref{BRM} is empty.
\end{example}

\subsection{KAM type results}\label{sec:KAMintro}

Now let us come back to the Hamiltonian $H$ as in~\eqref{Ham1}, which we recall is given by
\begin{equation*}
\begin{cases}
  H(\theta,I)= h(I) +\varepsilon f(\theta,I), \quad (\theta,I) \in \T^n \times D \\
  \nabla h(0):=\omega_0 \in \R^n.
\end{cases}  
\end{equation*}
The integrable Hamiltonian $h$ is said to be \emph{non-degenerate} at $0 \in D$ if the Hessian matrix $\nabla^2 h(0) \in M_n(\R)$ has a non-zero determinant. The image of the map $\Theta_0 : \T^n \rightarrow \T^n \times B$, $\theta \mapsto (\theta,0)$, is an embedded torus invariant by the flow of $h$ carrying a quasi-periodic flow with frequency $\omega_0$. We shall prove that this quasi-periodic invariant embedded torus is preserved by an arbitrary small perturbation, provided $h$ is non-degenerate, $H$ is $M$-ultra-differentiable with $M$ satisfying \eqref{H1} and \eqref{H2}, and $\omega_0$ satisfies the $\mathrm{BR_M}$-condition.

\begin{Main}\label{classicalKAM}
Let $H$ be as in~\eqref{Ham1}, where $H$ is $(M,s)$-ultra-differentiable with $M$ satisfying \eqref{H1} and \eqref{H2}, $\omega_0\in \R^n$ satisfying~\eqref{BRM} and $h$ is non-degenerate. Then there exists $0<s'<s$ such that for all $\varepsilon$ small enough, there exists an $(M,s')$-ultra-differentiable torus embedding
$\Theta_{\omega_0} : \T^n \rightarrow \T^n \times D$ such that
$\Theta_{\omega_0}(\T^n)$ is invariant by the Hamiltonian flow of $H$ and quasi-periodic with frequency $\omega_0$.  Moreover,
$\Theta_{\omega_0}$ is close to
  $\Theta_0$ in the sense that
  \[ |\Theta_{\omega_0}-\Theta_0|_{M,s'} \leq
  c\sqrt{\varepsilon} \]
  for some constant $c \geq 1$ independent of $\varepsilon$.
\end{Main}

Theorem~\ref{classicalKAM} is deduced from a KAM theorem for a
Hamiltonian with parameters, for which a quantitative statement is
given in \S~\ref{sec:KAMparam}. In the analytic case $M=M_{1}$, this
is a result of R{\"u}ssmann~\cite{Rus01} and in the Gevrey case
$M=M_\alpha$, $\alpha \geq 1$, this is a result we proved
in~\cite{BFa17}. Theorem~\ref{classicalKAM} is more general, and it
applies for instance to the family $M_{\alpha,\beta}$ given in
Example~\ref{exempleMalphabeta} or to the more rapidly growing
sequence described in Example~\ref{exempleexplog}. As
Example~\ref{exempleexpracine} shows, if the sequence $M$ grows too
fast, the above result is actually void; in the sequel when we will
refer to Theorem~\ref{classicalKAM} we will make the implicit
assumption that the $\mathrm{BR_M}$-condition is non-empty.

According to Theorem~\ref{classicalKAM}, the $\mathrm{BR_M}$-condition
is sufficient for the preservation of an invariant torus under an
$M$-ultra-differentiable perturbation. A natural question is: is it
necessary? To this question, here we only bring a partial
answer. Following Bessi~\cite{Bes00}, one can show that if
$\omega=\omega_0$ satisfies a condition (the
condition~\eqref{conddestruction}, which is defined below), the torus
can be destroyed. In general, this condition seems unrelated to the
$\mathrm{BR_M}$-condition. In the case of matching sequences
$M=(M_l)_{l \in \N}$ one can compare them easily. In the case of
moderate growth sequences $M$ one can also compare the two conditions,
albeit more loosely, since the sequence $M$ may then be upper and
lower bounded by matching sequences, as shown by
lemma~\ref{lm:MG}. Recall that associated to $M$ is a function
$\Omega$ which was defined in~\eqref{Ofunction}, and that by
definition $M$ is matching if
\begin{equation*}
C^{-1}(y) \asymp 1/\Omega(y), \quad y \sim +\infty.
\end{equation*}

\begin{Main}\label{destruction}
Given $\alpha \geq 1$, assume that the vector $\omega \in \R^n$ satisfies the following condition: there exists $c>0$ such that
\begin{equation}\label{conddestruction}\tag{$\mathrm{B_M}$}
\limsup_{Q \rightarrow +\infty} \frac{\ln(\Psi_\omega(Q))}{\Omega(cQ)}>0.
\end{equation}
Then an invariant torus with frequency $\omega$ can be destroyed by an
arbitrary small $M$-ultra-differentiable perturbation.
\end{Main}

Thus the condition that $\omega_0$ does not
satisfy~\eqref{conddestruction}, namely 
\begin{equation}\label{coho}\tag{$\mathrm{R_M}$}
\forall\, c>0, \quad \lim_{Q \rightarrow +\infty}
\frac{\ln(\Psi_\omega(Q))}{\Omega(cQ)}=0, 
\end{equation}
is a necessary condition for Theorem~\ref{classicalKAM} to hold
true. Actually, we expect that this condition~\eqref{coho} is a
necessary and sufficient condition to solve the cohomological equation
associated to $\omega$ in the $M$-ultra-differentiable case (in the
case $M=M_1$, that this condition is sufficient was shown
in~\cite{Rus75}).

Now in the case where $M$ is matching, condition~\eqref{coho} is
equivalent to  
\begin{equation}\label{coho2}
 \forall\, c>0, \quad \lim_{Q \rightarrow +\infty}
 C^{-1}(cQ)\ln(\Psi_\omega(Q)). 
\end{equation}
It is not hard to check that in general, condition~\eqref{BRM}
implies~\eqref{coho2} but clearly they are not equivalent to one
another, so, in the matching case, there remains a gap between the
sufficient and necessary conditions.

Looking at the family $M_\alpha$, the necessary condition~\eqref{coho}
reads 
\[ \lim_{Q \rightarrow +\infty}
\frac{\ln(\Psi_\omega(Q))}{Q^{1/\alpha}}=0\]
which shows that one cannot improve the value of the exponent in the
sufficient condition
\begin{equation*}
\int_{1}^{+\infty}\frac{\ln(\Psi_{\omega_0}(Q))}{Q^{1+1/\alpha}}dQ <
\infty. 
\end{equation*}
More generally, looking at the family $M_{\alpha,\beta}$,
$\alpha \geq 1$, $\beta \geq 0$, the necessary condition~\eqref{coho}
reads
\[ \lim_{Q \rightarrow +\infty}
\frac{\ln(Q)^{\beta/\alpha}\ln(\Psi_\omega(Q))}{Q^{1/\alpha}}=0,\]
which also shows that one cannot improve the values of the exponents
in the sufficient condition
\begin{equation*}
  \int_{\Delta
    (1)}^{+\infty}\frac{(\ln
    (\Delta^{-1}(x)))^{\beta/\alpha}}{x(\Delta^{-1}(x))^{1/\alpha}}dx 
  < \infty.    
\end{equation*}
So the gap between the condition~\eqref{BRM} and
condition~\eqref{coho} is narrow; in turn, this means that
Theorem~\ref{classicalKAM} is quite accurate if one restricts
attention to matching sequences. For non-matching sequences,
Theorem~\ref{classicalKAM} still gives a non-trivial result (see the
Example~\ref{exempleexplog}), albeit not necessarily accurate.

\bigskip

We also have variants of Theorem~\ref{classicalKAM} in the case where
$h$ is iso-energetically non-degenerate or the perturbation depends
periodically on time.

We say that the integrable Hamiltonian $h$ is \emph{iso-energetically
  non-degenerate} at $0$ if the so-called bordered Hessian of $h$,
\[\begin{pmatrix} 
\nabla^2 h(0) & ^{t}\nabla h(0) \\
\nabla h(0) & 0 
\end{pmatrix},\]
has a non-zero determinant. Under this assumption, the unperturbed
torus $I=0$, with energy $h(0)$, can be continued to a torus with the
same energy but with a frequency of the form $\lambda\omega_0$ for
$\lambda$ close to one.

\begin{Main}\label{isoKAM}
  Let $H$ be as in~\eqref{Ham1}, where $H$ is
  $(M,s)$-ultra-differentiable with $M$ satisfying \eqref{H1} and \eqref{H2},
  $\omega_0\in \R^n$ satisfying~\eqref{BRM} and $h$ is
  iso-energetically non-degenerate. Then there exists $0<s'<s$ such
  that for all $\varepsilon$ small enough, there exist
  $\lambda \in \R^*$ and an $(M,s')$-ultra-differentiable torus
  embedding $\Theta_{\omega_0} : \T^n \rightarrow \T^n \times D$ such
  that $\Theta_{\omega_0}(\T^n)$ is invariant by the Hamiltonian flow
  of $H$, contained in $H^{-1}(h(0))$ and quasi-periodic with
  frequency $\lambda\omega_0$.  Moreover, $\lambda$ is close to one
  and $\Theta_{\omega_0}$ is close to $\Theta_0$ in the sense that
  \[ |\lambda-1| \leq c\sqrt{\varepsilon}, \quad
  |\Theta_{\omega_0}-\Theta_0|_{M,s'} \leq c\sqrt{\varepsilon} \]
  for some constant $c \geq 1$ independent of $\varepsilon$.
\end{Main}

We can also look at the non-autonomous time-periodic version; we
consider a slightly different setting by looking at a Hamiltonian
function $\tilde{H} : \T^n \times D \times \T \rightarrow \R$ of the
form
\begin{equation}\label{HamT}
\begin{cases}\tag{$\tilde{*}$}
  \tilde{H}(\theta,I)= h(I) +\epsilon f(\theta,I,t), \\
  \nabla h(0):=\omega_0 \in \R^n.
\end{cases}  
\end{equation} 
It is better to consider the unperturbed torus $I=0$ as an invariant
torus for the integrable Hamiltonian
$\tilde{h}: D \times \R \rightarrow \R$ defined by
$\tilde{h}(I,J):=h(I)+J$: it is then quasi-periodic with frequency
$\tilde{\omega}_0:=(\omega_0,1)$, has dimension $n+1$ and is the image
of the trivial embedding
$\tilde{\Theta}_0 : \T^n \times \T \rightarrow \T^n \times D \times
\T$.

\begin{Main}\label{timeKAM}
  Let $\tilde{H}$ be as in~\eqref{HamT}, where $\tilde{H}$ is
  $(M,s)$-ultra-differentiable with $M$ satisfying \eqref{H1} and \eqref{H2},
  $\omega_0\in \R^n$ satisfying~\eqref{BRM} and $h$ is
  non-degenerate. Then there exists $0<s'<s$ such that for all
  $\varepsilon$ small enough, there exists an
  $(M,s')$-ultra-differentiable torus embedding
  $\tilde{\Theta}_{\omega_0} : \T^n \times \T \rightarrow \T^n \times
  D \times \T$
  such that $\tilde{\Theta}_{\omega_0}(\T^n \times \T)$ is invariant
  by the Hamiltonian flow of $\tilde{H}$ and quasi-periodic with
  frequency $\tilde{\omega}_0$.  Moreover, $\tilde{\Theta}_{\omega_0}$
  is close to $\tilde{\Theta}_0$ in the sense that
  \[ |\tilde{\Theta}_{\omega_0}-\tilde{\Theta}_0|_{M,s'} \leq
  c\sqrt{\varepsilon} \]
  for some constant $c \geq 1$ independent of $\varepsilon$.
\end{Main}

Theorem~\ref{isoKAM} and Theorem~\ref{timeKAM} are essentially equivalent statements and can be easily deduced from Theorem~\ref{classicalKAM}; in the analytic case details are given in~\cite{TreBook}, Chapter 2, but it is plain to observe that the arguments still work in the ultra-differentiable case. 

We can also state a discrete analogue of
Theorem~\ref{classicalKAM}. Given a function
$h : D \rightarrow \R$, we define the exact-symplectic map
\[ F_h : \T^n \times D \rightarrow \T^n \times D, \quad
(\theta,I) \mapsto (\theta+\nabla h (I),I).  \]

\begin{Main}\label{KAMmap1}
  Let $F : \T^n \times D \rightarrow \T^n \times D$ be an
  $(M,s)$-ultra-differentiable exact symplectic map, with $M$ satisfying \eqref{H1} and \eqref{H2} and such that
  \[ |F-F_h|_{\alpha,s_0} \leq \varepsilon. \]
Assume that $\omega_0=\nabla h(0)$ satisfies~\eqref{BRM} and that $h$ is non-degenerate. Then there exists $0<s'<s$ such that for all $\varepsilon$ small enough, there exists an $(M,s')$-ultra-differentiable torus embedding $\Theta_{\omega_0} : \T^n \rightarrow \T^n \times D$
  such that $\Theta_{\omega_0}(\T^n)$ is invariant by $F$ and
  $\Theta_{\omega_0}$ gives a conjugacy between the translation of
  vector $\omega_0$ on $\T^n$ and the restriction of $F$ to
  $\Theta_{\omega_0}(\T^n)$.  Moreover, $\Theta_{\omega_0}$ is close
  to $\Theta_0$ in the sense that
  \[ |\Theta_{\omega_0}-\Theta_0|_{M,s'} \leq
  c\sqrt{\epsilon} \]
  for some constant $c \geq 1$ independent of $\varepsilon$.
\end{Main}

Let us recall that Theorem~\ref{classicalKAM} follows from a KAM theorem for a Hamiltonian with parameters; from this last theorem we will also deduce Arnold's normal form theorem for vector fields on the torus close to constant, in the ultra-differentiable setting.

\begin{Main}\label{KAMvector}
  Let $M$ be a sequence satisfying \eqref{H1} and \eqref{H2}, $\omega_0 \in \R^n$
  satisfying~\eqref{BRM} and $X \in \mathcal{U}_{M,s}(\T^n,\R^n)$ a
  vector field on $\T^n$ of the form
  \[X = \omega_0 + B, \quad |B|_{M,s} \leq \mu.\]
  Then, for $\mu$ sufficiently small, there exist a vector
  $\omega_0^* \in \R^n$ and an $(M,s/2)$-ultra-differentiable diffeomorphism
  $\Xi : \T^n \to \T^n$ such that $X+\omega_0^*-\omega_0$ is conjugate
  to $\omega_0$ via $\Xi$:
  \[ \Xi^*(X+\omega_0^*-\omega_0)=\omega_0.\]
  Moreover, we have the estimate
  \[ |\omega_0^*-\omega_0| \leq c\mu, \quad |\Xi-\mathrm{Id}|_{M,s/2}
  \leq c\mu \] for some constant $c\geq 1$ independent of $\mu$.
\end{Main}

Observe that because of the shift of frequency $\omega_0^*-\omega_0$,
in general this result does not give any information on the vector
field $X$. Under some further assumption (for instance, if $\omega_0$
belongs to the rotation set of $X$, see~\cite{Kar16}), then this shift
vanishes and Theorem~\ref{KAMvector} implies that $X$ is conjugated to
$\omega_0$.

An even more restricted setting is when $X$ is proportional to
$\omega_0$ (so that the flow of $X$ is a re-parametrization of the
linear flow of frequency $\omega_0$ and thus $\omega_0$ is the unique
rotation vector of $X$); Theorem~\ref{KAMvector} applies in this case
to give a conjugacy to $\omega_0$, assuming that $\omega_0$
satisfies~\eqref{BRM}, but the proof is actually much simpler in this
case (it boils down to solve only once the cohomological equation) and
should require the weaker condition that $\omega_0$
satisfies~\eqref{coho}, as it is stated in the analytic case $M=M_1$
in~\cite{Fay02}. Still in~\cite{Fay02}, it is proved that for $M=M_1$,
if $\omega_0$ satisfies~\eqref{conddestruction}, then there is a dense
set of reparametrized linear flow which are weak-mixing (and so cannot
be conjugated to the linear flow); thus a necessary condition for
Theorem~\ref{KAMvector} to hold true is that $\omega_0$
satisfies~\eqref{coho} (and this is also a sufficient condition if we
impose that $X$ is proportional to $\omega_0$). Clearly, this should
extend to the general case of an matching sequence and
thus~\eqref{coho} is a necessary condition for Theorem~\ref{KAMvector}
to hold true, as in Theorem~\ref{classicalKAM}.

Finally, to conclude this section, let us a give the discrete version of Theorem~\ref{KAMvector}. Given $\omega_0 \in \R^n$, let $T_{\omega_0}$ be the translation of $\T^n$ of vector $\omega_0$:
\[ T_{\omega_0} : \T^n \rightarrow \T^n, \quad \theta \mapsto \theta+\omega_0. \]

\begin{Main}\label{KAMmap2}
Let $M$ be a sequence satisfying \eqref{H1} and \eqref{H2}, $\omega_0 \in \R^n$ satisfying~\eqref{BRM} and $T \in \mathcal{U}_{M,s}(\T^n,\T^n)$ a diffeomorphism of $\T^n$ of the form
\[T = T_{\omega_0} + B, \quad |B|_{M,s} \leq \mu.\]
Then, for $\mu$ sufficiently small, there exist a vector $\omega_0^* \in \R^n$ and an $(M,s/2)$-ultra-differentiable diffeomorphism
  $\Xi : \T^n \to \T^n$ such that $T+\omega_0^*-\omega_0$ is conjugate to $T_{\omega_0}$ via $\Xi$:
\[ \Xi^{-1} \circ (T+\omega_0^*-\omega_0) \circ \Xi=T_{\omega_0}.\] 
Moreover, we have the estimate
\[ |\omega_0^*-\omega_0| \leq c\mu, \quad |\Xi-\mathrm{Id}|_{M,s/2} \leq c\mu \]
for some constant $c\geq 1$ independent of $\mu$.
\end{Main}

\subsection{Hamiltonian normal forms and Nekhoroshev type results}\label{sec:Nekintro}

We are still considering a Hamiltonian $H$ as in~\eqref{Ham1} but we write $\omega$ instead of $\omega_0$, that is
\begin{equation*}
\begin{cases}
  H(\theta,I)= h(I) +\varepsilon f(\theta,I), \quad (\theta,I) \in \T^n \times D \\
  \nabla h(0):=\omega_0=\omega \in \R^n \setminus\{0\}.
\end{cases}  
\end{equation*}

We do not assume that $\omega$ is non-resonant, but we assume at least it is non-zero; without loss of generality (up to a linear symplectic change of coordinates) it can always be written as 
\[ \omega=(\bar{\omega},0) \in \R^d \times \R^{n-d} \]
for some $1 \leq d \leq n$ and where $\bar{\omega} \in \R^d$ is non-resonant. One can still define $\Psi_\omega$, $\Delta_\omega$ and $\Delta_\omega^*$ by 
\[ \Psi_\omega(Q):=\Psi_{\bar{\omega}}(Q), \quad \Delta_\omega(Q):=Q\Psi_\omega(Q), \quad Q \geq 1 \]
and 
\[ \Delta_\omega^*(x):=\sup\{ Q \geq 1 \; | \; \Delta_\omega(Q)\leq x\}, \quad x \geq \Psi_\omega(1). \]
Alternatively, one may also use their continuous and equivalent versions $\Psi$, $\Delta$ and $\Delta^{-1}$. 

From now on, let us denote by $L_\omega$ the linear integrable Hamiltonian of constant frequency $\omega$, that is $L_\omega(I)=\omega\cdot I$. Our first result is a normal form statement up to a small remainder, in the special case where the integrable Hamiltonian is linear.

\begin{Main}\label{lineaire}
Let $H$ be as in~\eqref{Ham1}, where $H$ is $(M,s)$-ultra-differentiable with $M$ satisfying \eqref{H1} and \eqref{H2} and $h=L_\omega$ with $\omega=(\bar{\omega},0) \in \R^d \times \R^{n-d}$ and $\bar{\omega}$ non-resonant. For all $\varepsilon$ sufficiently small, there exists a $(M,s/2)$-ultra-differentiable symplectic transformation
\[ \Phi : \T^n \times D_{1/2} \rightarrow \T^n \times D \]
such that
\[ H\circ\Phi =h+\bar{f}+\hat{f} \]
where $\{\bar{f},h\}=0$ and with the estimates
\begin{equation*}
\begin{cases}
|\Phi-\mathrm{Id}|_{s/2}\leq c_1\Psi(\Delta^{-1}(c_2s\varepsilon^{-1}))\varepsilon, \quad |\bar{f}|_{s/2} \leq c_1\varepsilon, \\
|\hat{f}|_{s/2} \leq c_1\varepsilon\exp\left(-c_3\left(C^{-1}(c_4s\Delta^{-1}(c_2s\varepsilon^{-1}))\right)^{-1}\right)
\end{cases}
\end{equation*}
where $c_1>1$, $c_2<1$, $c_3<1$ and $c_4<1$ depend only on $n$, $\omega$ and the function $C$.
\end{Main}

This result was known in the analytic case (see~\cite{Fas90},~\cite{Pos93}) and in the Gevrey case (\cite{Bou13a}). Theorem~\ref{lineaire} implies the following ``partial" stability result, in which we denote by $\Pi_d : \R^n \rightarrow \R^d$ the projection onto the first $d$ components.

\begin{corollary}\label{corlineaire}
Under the assumptions of Theorem~\ref{lineaire}, for any $I_0 \in D_{1/8}$ and any $r>0$ such that 
\[ 2c_1\Psi(\Delta^{-1}(c_2s\varepsilon^{-1}))\varepsilon \leq r \leq 1/4 ,\] 
any solution $(\theta(t),I(t))$ of $H$ with $I(0)=I_0$ satisfy
\[ |\Pi_d (I(t)-I_0)| \leq 2r, \quad |t| \leq \tilde{c}_1rs\varepsilon^{-1}\exp\left(-c_3\left(C^{-1}(c_4s\Delta^{-1}(c_2s\varepsilon^{-1}))\right)^{-1}\right) \]
as long as $I(t) \in D_{1/4}$, where $c_1>1$, $\tilde{c}_1<1$ $c_2<1$, $c_3<1$ and $c_4<1$ depend only on $n$, $\omega$ and the function $C$.
\end{corollary}

Let us point out that these stability estimates actually apply to any solution; indeed, if $I(0) \in D_{1-\delta}$ for an arbitrary but fixed $\delta>0$, one would obtain 
\[ |\Pi_d (I(t)-I(0))|<\delta \]
for the same interval of time as long as $I(t)$ stays in the domain, upon letting the constants and the smallness assumption depend on $\delta$. For convenience only, we sated the result for any solution with $I(0) \in D_{1/8}$ and the restriction on time was that $I(t)$ stays in $D_{1/4}$. In general, this result does not give a confinement of the action variables for a long interval of time because it may well happen that the orbit escape the domain very quickly; the only exception is when $d=n$, that is when $\omega$ is actually non-resonant. The next corollary follows at once from the preceding one.

\begin{corollary}\label{corlineaire2}
Under the assumptions of Theorem~\ref{lineaire} and in the case where $\omega=\bar{\omega} \in \R^n$ is non-resonant, for any $I_0 \in D_{1/8}$ and any $r>0$ such that 
\[ 2c_1\Psi(\Delta^{-1}(c_2s\varepsilon^{-1}))\varepsilon \leq r \leq 1/4 ,\] 
any solution $(\theta(t),I(t))$ of $H$ with $I(0)=I_0$ satisfy
\[ |(I(t)-I_0)| \leq 2r, \quad |t| \leq \tilde{c}_1rs\varepsilon^{-1}\exp\left(-c_3\left(C^{-1}(c_4s\Delta^{-1}(c_2s\varepsilon^{-1}))\right)^{-1}\right) \]
where $c_1>1$, $\tilde{c}_1<1$ $c_2<1$, $c_3<1$ and $c_4<1$ depend only on $n$, $\omega$ and the function $C$.
\end{corollary}

A natural question is whether these last estimates can be improved in general. Without any assumption on the sequence $M=(M_l)_{l \in \N}$, we can always construct an example of unstable orbits with controlled speed in the special case where the integrable Hamiltonian $h$ is linear. In the statement below, for convenience we will measure the size of the perturbation $f$ not by its norm but rather by the norm of its Hamiltonian vector field $X_f$.

\begin{Main}\label{difflineaire}
There exist a sequence of Hamiltonians $H_j=h+f_j=L_\omega+f_j$ as in~\eqref{Ham1} where $\omega=(\bar{\omega},0) \in \R^d \times \R^{n-d}$ with $\bar{\omega}$ non-resonant and
\[ |X_{f_j}|_s \leq 2c\varepsilon_j, \quad \lim_{j \rightarrow +\infty}\varepsilon_j=0, \]
such that $H_j$ has a solution $(\theta_j(t),I_j(t))$ which is globally defined and satisfies
\[ |t|\varepsilon_j\exp\left(-\Omega(16\pi s\Delta_\omega^*(2\varepsilon_j^{-1}))\right) \leq |\Pi_d(I_j(t)-I_j(0))| \leq |t|\varepsilon_j\exp\left(-\Omega(8\pi s\Delta_\omega^*((2\varepsilon_j)^{-1}))\right).   \]
\end{Main}

In the analytic case, this theorem was proved in~\cite{Bou12} and then extended in~\cite{Bou13a} in the Gevrey case. 

In general, this last statement is not related to
Theorem~\ref{lineaire} or Corollary~\ref{corlineaire} because it is
the function $\Omega$, and not $C^{-1}$, which measures the speed of
instability. Yet for matching sequences, for which by definition
$1/\Omega$ is equivalent to $C^{-1}$, Theorem~\ref{difflineaire} shows
that in general one cannot improve the conclusions of both
Corollary~\ref{corlineaire} and Theorem~\ref{lineaire}.

\bigskip

Theorem~\ref{lineaire} gives a global normal form for a linear integrable Hamiltonian $h=L_\omega$ and Corollary~\ref{corlineaire} gives a global stability result. For an arbitrary nonlinear integrable Hamiltonian $h$, these global results translate into local results, valid on a small ball of radius $\rho>0$ around the origin. The only results we can obtain have an anistropic character, so for a given $\rho>0$, it is convenient to introduce the scalings
\[ \sigma_\rho(\theta,I):=(\theta,\rho I), \quad \sigma_\rho^{-1}(\theta,I):=(\theta,\rho^{-1} I).  \]
Clearly, $\sigma_\rho$ is a diffeomorphism between $\T^n \times D$ and $\T^n \times D_\rho$.

\begin{Main}\label{nonlineaire}
Let $H$ be as in~\eqref{Ham1}, where $H$ is $(M,s)$-ultra-differentiable with $M$ satisfying \eqref{H1} and \eqref{H2} and $\omega=(\bar{\omega},0) \in \R^d \times \R^{n-d}$ with $\bar{\omega}$ non-resonant. For all $\varepsilon$  and all $\rho$ sufficiently small with
\[ \sqrt{\varepsilon} \leq \rho \leq 1,\]
there exists a $(M,s/4)$-ultra-differentiable symplectic transformation
\[ \Phi : \T^n \times D_{\rho/2} \rightarrow \T^n \times D_{\rho} \]
such that
\[ H\circ\Phi =h+\bar{f}+\hat{f} \]
where $\{\bar{f},L_\omega\}=0$ and with the estimates
\begin{equation*}
\begin{cases}
|\sigma_\rho^{-1} \circ \Phi \circ \sigma_\rho-\mathrm{Id}|_{s/4}\leq c_1\Psi(\Delta^{-1}(c_2s\rho^{-1}))\rho^{-1}\varepsilon, \quad |\bar{f} \circ \sigma_\rho|_{s/4} \leq c_1\varepsilon, \\
|\hat{f} \circ \sigma_\rho|_{s/4} \leq c_1\varepsilon\exp\left(-c_3\left(C^{-1}(c_4s\Delta^{-1}(c_2s\rho^{-1}))\right)^{-1}\right)
\end{cases}
\end{equation*}
where $c_1>1$, $c_2<1$, $c_3<1$ and $c_4<1$ depend only on $n$, $\omega$, $|h|_s$ and the function $C$.
\end{Main}

In the analytic or Gevrey case, a similar result but with weaker estimates was obtained in~\cite{Bou13a}. As before, Theorem~\ref{nonlineaire} implies the following ``partial" and local stability result.

\begin{corollary}\label{cornonlineaire}
Under the assumptions of Theorem~\ref{nonlineaire}, for any $I_0 \in D_{\rho/8}$ and any $r>0$ such that 
\[ 2c_1\Psi(\Delta^{-1}(c_2s\varepsilon^{-1}))\varepsilon \leq r \leq \rho/4 ,\] 
any solution $(\theta(t),I(t))$ of $H$ with $I(0)=I_0$ satisfy
\[ |I(t)-I_0| \leq 2r, \quad |t| \leq \tilde{c}_1rs\varepsilon^{-1}\exp\left(-c_3\left(C^{-1}(c_4s\Delta^{-1}(c_2s\rho^{-1}))\right)^{-1}\right) \]
as long as $I(t) \in D_{\rho/4}$, where $c_1>1$, $\tilde{c}_1<1$ $c_2<1$, $c_3<1$ and $c_4<1$ depend only on $n$, $\omega$, $|h|_s$ and the function $C$.
\end{corollary}

Again, we also have the special case where the frequency is non-resonant.

\begin{corollary}\label{cornonlineaire2}
Under the assumptions of Theorem~\ref{nonlineaire}  and in the case where $\omega=\bar{\omega} \in \R^n$ is non-resonant, for any $I_0 \in D_{\rho/8}$ and any $r>0$ such that 
\[ 2c_1\Psi(\Delta^{-1}(c_2s\varepsilon^{-1}))\varepsilon \leq r \leq \rho/4 ,\] 
any solution $(\theta(t),I(t))$ of $H$ with $I(0)=I_0$ satisfy
\[ |I(t)-I_0| \leq 2r, \quad |t| \leq \tilde{c}_1rs\varepsilon^{-1}\exp\left(-c_3\left(C^{-1}(c_4s\Delta^{-1}(c_2s\rho^{-1}))\right)^{-1}\right) \]
where $c_1>1$, $\tilde{c}_1<1$ $c_2<1$, $c_3<1$ and $c_4<1$ depend only on $n$, $\omega$, $|h|_s$ and the function $C$.
\end{corollary}

For an arbitrary non-linear integrable Hamiltonian, one can only obtain local and partial stability result; moreover we do not have an equivalent of Theorem~\ref{difflineaire} as Corollary~\ref{cornonlineaire} can be improved in many cases. 

In order to obtain global and full stability, one needs a further assumption on the integrable part. Let us say that $h : G \rightarrow \R$ is \emph{quasi-convex} on some domain $G \subseteq \R^n$ if there exist positive constants $l$ and $m$, such that for all $I \in G$, it satisfies the following two properties:
\begin{equation}\label{qc1}
|\nabla h(I)|\geq l 
\end{equation}
and at least of the inequalities
\begin{equation}\label{qc2}
|\nabla h(I)\cdot \xi|\geq l|\xi|, \quad |\nabla^2 h(I)\xi\cdot \xi| \geq m|\xi|^2 
\end{equation}
is satisfied for any $\xi \in \R^n$. To emphasize the role of the constants $l$ and $m$, we say that $h$ is $(l,m)$-quasi-convex. This definition is slightly more restrictive than the one used in~\cite{Pos93}, as~\eqref{qc1} is not assumed in the latter reference; we actually require $h$ not to have critical points just for convenience, in order to have a statement which is uniform in phase space (this is not the case in~\cite{Pos93}). Observe that if we assume~\eqref{qc1} and only the second part of~\eqref{qc2} for all $\xi \in \R^n$ orthogonal to $\nabla h(I)$, then~\eqref{qc2} is actually satisfied for a possibly smaller constant $l$ (see~\cite{MS02} for instance). 

Under this quasi-convexity assumption, we will prove the following result.

\begin{Main}\label{convex}
Let $H$ be as in~\eqref{Ham1}, where $H$ is $(M,s)$-ultra-differentiable with $M$ satisfying \eqref{H1} and \eqref{H2} and $h$ is $(l,m)$-quasi-convex. For all $\varepsilon$ sufficiently small, any $I_0 \in D_{1/2}$ and any solution $(\theta(t),I(t))$ of $H$ with $I(0)=I_0$, we have
\[ |I(t)-I_0| \leq c_1s(s^{-2}\varepsilon)^{\frac{1}{2n}}, \quad |t| \leq \tilde{c}_1s\exp\left(c_2\left(C^{-1}\left(c_3s(s^2\varepsilon^{-1})^{\frac{1}{2n}}\right)\right)^{-1}\right) \]
where $c_1>1$, $c_2<1$ and $c_3<1$ depend only on $n$, $|h|_s$, $l$, $m$ and the function $C$.
\end{Main}

In the analytic case $M=M_1$, this theorem is due to Nekhoroshev (\cite{Nek77},\cite{Nek79}) with an exponent worse than $1/(2n)$; the acutal value $1/(2n)$ was later obtained in~\cite{Loc92}, \cite{LN92}, \cite{LNN94} and~\cite{Pos93}. The Gevrey case $M=M_\alpha$ for $\alpha \geq 1$ was established in~\cite{MS02}, and so our Theorem~\ref{convex} further extend these results for any sequence $M$ satisfying \eqref{H1} and \eqref{H2}.

As before, the natural question is whether these stability estimates can be improved in general, and to do so, one has to construct unstable orbits. As a matter of fact, the exponent $1/(2n)$ in the time-scale of stability in Theorem~\ref{convex} can be improved if one is willing to allows a drift arbitrarily close to one (that is, replacing the drift in the action variables of order $\varepsilon^{1/(2n)}$ by a drift of order $\varepsilon^b$ with $b$ arbitrarily close to zero). In the Gevrey case $M=M_\alpha$ for $\alpha \geq 1$ (thus including the analytic case), it was proved in~\cite{BM11} that one can get an exponent arbitrarily close to $1/(2(n-1))$ and in the analytic case, it was proved in~\cite{ZZ17} that one can get an exponent arbitrarily close to $1/(2(n-2))$.

For analytic Hamiltonians, this ends the question of optimality as examples of unstable orbits with exponent $1/(2(n-2))$ where constructed in~\cite{Bes96} for $n=3$, ~\cite{Bes97} for $n=4$ and~\cite{Zha11} for $n\geq 5$. For Gevrey non-analytic Hamiltonians $\alpha>1$ a similar result was obtained in~\cite{MS02}, but the latter crucially use the assumption that the space of $\alpha$-Gevrey functions contains bump functions.   

For a given sequence $M=(M_l)_{l \in \N}$, it is said to be non-quasi-analytic if the following condition is satisfied:
\begin{align}
\mbox{The sequence $(1/\mu_l)_{l \in \N}=(M_l/M_{l+1})_{l \in \N}$ is summable}, i.e., \sum_{l \in \N}M_l/M_{l+1}<+\infty.\tag{H3}\label{H3}
\end{align} 
Due to a famous theorem of Denjoy-Carleman, this condition exactly
characterizes classes which are non-quasi-analytic; we recall here
that a function (say, on the real line) is quasi-analytic if it
vanishes identically if and only if its formal Taylor series vanishes
at some point. Actually, non-quasi-analytic classes contain bump
functions (with arbitrarily small support) and this follows from the
proof of Denjoy-Carleman theorem.

As it turns out, we are not able to use the mechanism of~\cite{Bes96},
~\cite{Bes97} and~\cite{Zha11} but we can still use the construction
of~\cite{MS02}, for a sequence $M$ which satisfy~ \eqref{H1}, \eqref{H2} but also
(H3). Here is a precise statement.

\begin{Main}\label{diffconvex}
  Let $n \geq 3$. There exist a sequence of Hamiltonians $H_j=h+f_j$
  as in~\eqref{Ham1}, where $H_j$ is $(M,s)$-ultra-differentiable with
  $M$ satisfying \eqref{H1}, \eqref{H2} and (H3), $h$ satisfies~\eqref{qc1}
  and~\eqref{qc2} and
  \[ |f_j|_s \leq \varepsilon_j, \quad \lim_{j \rightarrow
    +\infty}\varepsilon_j=0, \]
  such that the following holds. The Hamiltonian $H_j$ has a solution
  $(\theta_j(t),I_j(t))$ which satisfies
  \[ |I_j(\tau_j)-I_j(0)| \geq 1 \]
  where $\tau_j>0$ satisfies, for $j$ large enough, the estimates
  \[
  \exp\left(2\Omega\left(c_1\varepsilon_j^{-\frac{1}{2(n-2)}}\right)\right)
  \leq \tau _j \leq
  \exp\left(2\Omega\left(c_2\varepsilon_j^{-\frac{1}{2(n-2)}}\right)\right)\]
  for positive constants $c_1<c_2<1$ which depend only $n$, $s$ and
  the sequence $M$.
\end{Main}

Leaving aside the question of the exponent $1/(2(n-2))$ that we
already discussed, Theorem~\ref{diffconvex} says that
Theorem~\ref{convex} is essentially optimal, at least for matching
sequences which satisfy (H3).

\bigskip

To conclude, we recall that Nekhoroshev actually proved Theorem~\ref{convex} in the analytic case not only for quasi-convex functions, but for a much larger and generic class of functions that we now define.

We say that a differentiable function $h : G \rightarrow \R$ is \emph{steep} on some domain $G \subseteq \R^n$ if there exist positive constants $l,L,\delta,p_j$, for any integer $1 \leq j \leq n-1$, such that for all $I \in G$, we have $|\nabla h(I)| \geq l$ and, for all integer $1 \leq j \leq n-1$, for all vector space $\Lambda \in \R^n$ of dimension $j$, letting $\lambda=I+\Lambda$ the associated affine subspace passing through $I$ and $h_\lambda$ the restriction of $h$ to $\lambda$, the inequality
\[ \max_{0 \leq \eta \leq \xi}\;\min_{|I'-I|=\eta, \; I' \in \lambda}|\nabla h_\lambda(I')-\nabla h_\lambda(I)|>L\xi^{p_j} \]
holds true for all $0 < \xi \leq\delta$. As before, to emphasize the role of the constants we say that $h$ is $(l,L,\delta,(p_j)_{1 \leq j \leq n-1})$-steep on $G$ and, if all the $p_j=p$, we say that $h$ is $(l,L,\delta,p)$-steep on $G$.

One may easily check that quasi-convex functions are steep, and in fact they are the ``steepest" as one has $p=1$ in this case (in general $p_j$ are integers larger than one). But steep functions form a much larger class of functions and among sufficiently smooth functions they are generic in many senses (see~\cite{Nek79} or~\cite{BFN17}). 

The original proof of Nekhoroshev, in the analytic case, was dealing with steep integrable Hamiltonians but the values of the stability exponents, which depend on the steepness exponents $p_j$ for $1 \leq j \leq n$, were not very sharp and even restricted to the quasi-convex they would be of order $1/n^2$ which is obviously worse than the exponent $1/(2n)$ of Theorem~\ref{convex}. A different and somehow simpler proof was given in~\cite{BN09}, using only rough periodic approximations, but leading also to bad values for the exponents; the interest was that this method was flexible enough to give results in the Gevrey (and even finitely differentiable) case, see~\cite{Bou11}. Sharp exponents, generalizing those of the quasi-convex case, were finally obtained in~\cite{GCB16} but so far it is restricted to the analytic case. So in the more general setting of ultra-differentiable functions we are considering here, we were not able to extend the results of~\cite{GCB16}, but we will manage to use the method originally introduced in~\cite{BN09} to obtain the following result.  

\begin{Main}\label{steep}
Let $H$ be as in~\eqref{Ham1}, where $H$ is $(M,s)$-ultra-differentiable with $M$ satisfying \eqref{H1} and \eqref{H2} and $h$ is $(l,L,\delta,p)$-steep. For all $\varepsilon$ sufficiently small, any $I_0 \in D_{1/2}$ and any solution $(\theta(t),I(t))$ of $H$ with $I(0)=I_0$, we have
\[ |I(t)-I_0| \leq c_1 \varepsilon ^{\frac{1}{2na}}, \quad |t| \leq s\exp\left(c_2\left(C^{-1}\left(c_3s(\varepsilon^{-1})^{\frac{1}{2na}}\right)\right)^{-1}\right) \]
where $a:=(np)^{n-1}$ and $c_1>1$, $c_2<1$ and $c_3<1$ which depend only on $n$, $|h|_s$, $l$, $L$, $\delta$, $p$ and the function $C$.
\end{Main}

As we already explained before, the value of the exponent $1/(2na)$ is
very bad here and one should expect better exponents such as those
obtained in~\cite{GCB16}. However unlike the quasi-convex
Theorem~\ref{convex} which comes with Theorem~\ref{diffconvex} to test
the optimality of the exponents, we do not have such instability
result in the steep non-quasi-convex case; this is actually an open
problem in any regularity, including the analytic one.

\section{Estimates on ultra-differentiable functions}\label{sec:udiff}

Let us start by recalling some notations and definitions. Given an
integer $m \geq 1$ and $k=(k_1,\dots,k_m) \in \N^m$, we define
\[ |k|=\sum_{i=1}^m k_i, \quad k!=\prod_{i=1}^m k_i!. \]
Given $x \in \R^m$, we set
\[ x^k=\prod_{i=1}^m x_i^{k_i} \]
and for a smooth function $f : B \rightarrow \R$ defined on some open
ball $B \subset \R^m$ (see Remark~\ref{rmk:Tn} at the end of
\S~\ref{sec:majorant1} for functions defined on an open set of
$\R^m \times \T^n$) and $a \in B$, we set
\[ \partial^k
f(a)=\partial_{x_1}^{k_1}\cdots \partial_{x_m}^{k_m}f(a). \]
Given a sequence $M=(M_l)_{l \in \N}$ with $M_0=M_1=1$ and $s>0$,
recall from \S~\ref{sec:udiffintro} that the function $f$ is said to be
$(M,s)$-ultra-differentiable if
\begin{equation}\label{defnorm}
|f|_{M,s}:=c\sup_{a \in B}\left(\sup_{k \in \N^m}\frac{(|k|+1)^2s^{|k|}|\partial^k f(a)|}{M_{|k|}}\right)<\infty, \quad c:=4\pi^2/3.
\end{equation}
The space of such functions will be denoted by $\mathcal{U}_{M,s}(B)$,
and, equipped with the above norm, it is a Banach space. On the one
hand, the role of the factor $(|k|+1)^2$ is to make easier the
estimates for the product and for the composition of
ultra-differentiable functions (see respectively
Lemmas~\ref{lemmeprod} and~\ref{lemmecomp}). On the other hand, the
role of the normalizing constant $c>0$ in the definition is to ensure
that $\mathcal{U}_{M,s}(B)$ is a Banach algebra with respect to
multiplication (see Lemma~\ref{lemmeprod}). The factor $(|k|+1)^2$ is
only here for convenience and it can be removed (see
\S~\ref{sec:products} and \S~\ref{sec:composition}); then one has to
choose a different normalizing constant to turn $\mathcal{U}_{M,s}(B)$
into a Banach algebra.

The above definition can be extended to vector-valued functions
$f=(f_i)_{1 \leq i \leq p} : K \rightarrow \R^p$ for $p \geq 1$ by
setting
\[ |\partial^k f(a)|:=\sum_{1 \leq i \leq p}|\partial^k f_i(a)|, \quad
a \in K \]
in~\eqref{defnorm}. The space of such vector-valued functions is still
a Banach space with the above norm, and it will be denoted by
$\mathcal{U}_{M,s}(B,\R^p)$. The case of matrix-valued functions, say
with values in the space $M_{m,p}(\R)$ of matrices with $m$ rows and
$p$ columns, is reduced to the case of vector-valued functions by
simply identifying $M_{m,p}(\R)$ to $\R^{mp}$.

In the sequel, to lighten the notations, we will simply write
$|\,.\,|_{s}$ (respectively $\mathcal{U}_{s}(B)$ and
$\mathcal{U}_{s}(B,\R^p)$) instead of $|\,.\,|_{M,s}$ (respectively
$\mathcal{U}_{M,s}(B)$ and $\mathcal{U}_{M,s}(B,\R^p)$).

\subsection{Majorant series and ultra-differentiable functions}
\label{sec:majorant1}

The definition of ultra-differentiable functions can be conveniently
reformulated in terms of majorant series with one variable (see
\cite{Kom79}, \cite{Kom80} and also \cite{SCK03}).

But first let us consider a formal power series in $m$ variables
$X=(X_1,\dots,X_m)$ with coefficients in a normed real vector space
$(E,|\,.\,|_E)$, which is a formal sum of the form
\[ A(X)=\sum_{k \in \N^m}A_k X^k, \quad A_k=A_{k_1,\dots,k_m} \in
E.  \]
Such a formal series is said to be majorized by another formal power
series with real non-negative coefficients
\[ B(X)=\sum_{k \in \N^m}B_k X^k, \quad B_k=B_{k_1,\dots,k_m} \in \R^+,  \]
and we write $A \ll B$, if
\begin{equation}\label{major1}
|A_k|_E \leq B_k, \quad \forall \, k \in \N^m.
\end{equation}
Next, following~\cite{SCK03}, we introduce a notion of a smooth
function being majorized by a formal power series in one variable. So
let $f : B \rightarrow \R^p$ be a smooth function defined on some open
ball $B \subset \R^m$, and $F$ be a formal power series in one
variable with non-negative coefficients, that we shall write as
\[ F(X)=\sum_{l=0}^{+\infty}\frac{F_l}{l!}X^l. \]
We will say that $f$ is majorized by $F$ on $B$, and we will write
$f \ll_B F$ (or $f(x) \ll_B F(X)$), if for all $a \in B$ and all
$k \in \N^m$, we have
\begin{equation}\label{major2}
|\partial^k f(a)| \leq F_{|k|}. 
\end{equation}  
It will be usefull to rephrase this relation in terms of the formal
Taylor series of $f$ at points $a \in B$ (taking values in $\R^p$),
defined by
\[ T_af(X):=\sum_{k \in \N^m}\frac{\partial^kf(a)}{k!} X^k.\]
To the formal series $F$ in one variable, one can associate a formal
series $\hat{F}$ in $m$ variables simply by setting
\[ \hat{F}(X_1,\dots,X_m):=F(X_1+\cdots+X_m). \]
Then
\[f \ll_B F \quad \mbox{in the sense of~\eqref{major2}}
\Leftrightarrow \forall a \in B \quad T_af \ll \hat{F} \quad \mbox{in
  the sense of~\eqref{major1}}\]
(with $E=\R^p$ and $|\,.\,|_E = |\,.\,|_{l^1}$).

Now, given a sequence $M$ and $s>0$, let us define the following
formal power series in one variable
\begin{equation}\label{defM}
  U_{s}(X) 
  :=c^{-1}\sum_{l=0}^{+\infty}\frac{1}{(l+1)^2}\frac{M_l}{l!}
  \left(\frac{X}{s}\right)^l =
  c^{-1}\sum_{l=0}^{+\infty}\frac{N_l}{(l+1)^2}\left(\frac{X}{s}\right)^l, 
  \quad c=4\pi^2/3
\end{equation}
(recall from~\eqref{sequences} that we have defined
$N_l:=M_l/l!$). The following characterization of ultra-differentiable
functions is evident from the definitions~\eqref{defnorm}
and~\eqref{major2}.

\begin{proposition}\label{propnorme}
 If $f : B \to \R^p$ is a smooth function and $s>0$, then
  \[|f|_{s} = \inf \left\{ C \in [0,+\infty]\, | \; f \ll_B C
    U_{s}\right\}\]
  and
  \[f \in \mathcal{U}_s(B,\R^p) \Leftrightarrow |f|_s < \infty.\]
\end{proposition}

\begin{remark}
  \label{rmk:Tn}%
  Note that being an ultra-differentiable function is a local
  property. So the above immediately translates for functions defined
  on an open subet of the product of $\T^n$ and $\R^m$, by merely
  considering any local chart at any point of the torus. These
  definitions thus encompass the space
  $\mathcal{U}_s(\T^n \times D,\R^p)$ introduced in
  \S~\ref{sec:udiffintro}.
\end{remark}

\subsection{Properties of majorant series}\label{sec:majorant2}

We collect here some properties of majorant series that will be used later on. It is clear how to define the derivatives of a formal power series in one variable, and also a linear combination and the product of two such formal power series. We then have the following lemma.

\begin{lemma}\label{properties} 
Let $f,g : B \rightarrow \R^p$ be smooth functions, $F,G$ be two formal power series in one variable, and assume that
\[ f \ll_B F, \quad g \ll_B G.  \]
Then
\begin{equation}\label{deriv}
\partial^k f \ll_B \partial^{|k|} F, \quad k \in \N^m,
\end{equation}
\begin{equation}\label{somme}
\lambda f + \mu g \ll_B |\lambda|F+|\mu|G, \quad \lambda \in \R, \; \mu \in \R.
\end{equation}
Moreover, the scalar product $f\cdot g = \sum_{i=1}^pf_ig_i : B
\rightarrow \R$ satisfies
\begin{equation}\label{prod}
f\cdot g \ll_B FG.
\end{equation}
\end{lemma}

For a proof, we refer to~\cite{SCK03}, Lemma $2.2$, in which the case $p=1$ is considered; but the general case $p \geq 1$ is entirely similar.

Given two formal power series in one variable $F$ and $G$, we define the composition $F \odot G$ of $F$ and $G$ by
\[ F \odot G(X):=\sum_{l=0}^{+\infty}\frac{F_l}{l!}\left(G(X)-G(0)\right)^l. \]

\begin{lemma}\label{comp}
Let $f : B \rightarrow \R^p$, $g : D \rightarrow \R^m$ be smooth functions where $B$ and $D$ are open balls of $\R^m$ such that $g(D)\subseteq B$, and assume that
\[ f \ll_B F, \quad g \ll_{D} G. \]
Then
\[ f \circ g \ll_{D} F \odot G.  \]
\end{lemma}

Once again, for a proof we refer to~\cite{SCK03}, Lemma $2.3$.

\subsection{Derivatives}\label{sec:derivatives}

In this section, we will show that derivatives of a $M$-ultra-differentiable function are still $M$-ultra-differentiable, at the expense of reducing the width parameter $s>0$; these are analogues of Cauchy estimates for analytic functions. It is precisely at this point that the assumption \eqref{H2} on the sequence $M$ is needed and that the Cauchy function $C$ comes into play.

\begin{proposition}\label{derivative}
Let $f \in \mathcal{U}_{s}(B,\R^p)$ and $0<\sigma<1$. Then for any $k \in \N^m$ with $|k|=1$, $\partial^k f \in G_{s(1-\sigma)}(B,\R^p)$ and we have
\[ |\partial^k f|_{s(1-\sigma)} \leq s^{-1}C(\sigma) |f|_{s}. \] 
\end{proposition}

We have stated the above proposition only for $k \in \N^m$ with $|k|=1$ as this is the only case we shall need; but clearly one could obtain an estimate for any $k \in \N^m$ by a straightforward induction. 

\begin{proof}
From Proposition~\ref{propnorme} and~\eqref{deriv} of Lemma~\ref{properties}, it is sufficient to prove that
\begin{equation}\label{aprouver}
\partial^1 U_s \ll s^{-1}C(\sigma)U_{s(1-\sigma)}
\end{equation}
where $U_s$ is the formal power series defined in~\eqref{defM}. We have
\[ \partial^1 U_{s}(X)=c^{-1}\sum_{l=1}^{+\infty}\frac{1}{(l+1)^2}\frac{M_l}{(l-1)!} \frac{X^{l-1}}{s^l}=c^{-1}\sum_{l=0}^{+\infty}\frac{1}{(l+2)^2}\frac{M_{l+1}}{l!} \frac{X^{l}}{s^{l+1}} \]
and hence~\eqref{aprouver} is true if, for all $l \in \N$,
\begin{equation*}
\frac{M_{l+1}}{M_l}(1-\sigma)^l \leq C(\sigma)
\end{equation*}
which is trivially true by the definition of the function C in~\eqref{Cfunction} and since $1-\sigma \leq e^{-\sigma}$ for $0 \leq \sigma \leq 1$.
\end{proof}

For $f : B \rightarrow \R$, let $\nabla f : B \rightarrow \R^m$ be the vector-valued function formed by the partial derivatives of $f$ of order one, and more generally, for $f : B \rightarrow \R^p$, we let $\nabla f : B \rightarrow M_{m,p}(\R) \simeq \R^{mp}$ be the matrix-valued function whose columns are given by $\nabla f_i$ where $f=(f_i)_{1 \leq i \leq p}$. Then we have the following obvious corollary of Proposition~\ref{derivative}.

\begin{corollary}\label{corderivative}
Let $0<\sigma<1$. If $f \in \mathcal{U}_{s}(B,\R)$, then $\nabla f \in \mathcal{U}_{s(1-\sigma)}(B,\R^m)$ and
\[ |\nabla f|_{s(1-\sigma)} \leq ms^{-1}C(\sigma)|f|_{s}\]
and if $f \in \mathcal{U}_{s}(B,\R^p)$, then $\nabla f \in \mathcal{U}_{s(1-\sigma)}(B,\R^{mp})$ and
\[ |\nabla f|_{s(1-\sigma)} \leq mps^{-1}C(\sigma)|f|_{s}.\] 
\end{corollary}

\subsection{Products}\label{sec:products}

In this section, we shall prove that the space of $M$-ultra-differentiable functions is stable under multiplication, provided \eqref{H1} is satisfied.

\begin{proposition}\label{produit}
Let $f,g \in \mathcal{U}_{s}(B,\R^p)$, with $M$ satisfying \eqref{H1}. Then $f\cdot g \in \mathcal{U}_{s}(B,\R)$ and we have
\[ |f\cdot g|_{s} \leq |f|_{s}|g|_{s}. \]
\end{proposition}

Once again, in view of Proposition~\ref{propnorme} and~\eqref{prod} of Proposition~\ref{properties}, Proposition~\ref{produit} is a direct consequence of the following lemma.

\begin{lemma}\label{lemmeprod}
We have
\[ U_s^2 \ll U_s. \]
\end{lemma}

To prove this lemma, we will need another elementary lemma which makes use of the condition \eqref{H1}.

\begin{lemma}\label{lemmeprodavant}
If $M$ satisfies \eqref{H1}, then 
\[ N_lN_k \leq N_{l+k}, \quad (l,k)\in \N^2. \]
\end{lemma}

\begin{proof}
Since the sequence $\nu$ is increasing and $N_0=1$, we have
\[ N_lN_k=\nu_0\dots\nu_{l-1}\nu_0\dots\nu_{k-1}\leq\nu_0\dots\nu_{l-1}\nu_l\dots\nu_{l+k-1}=N_{l+k}.  \]
\end{proof}

The proof of Lemma~\ref{lemmeprod} given below follows~\cite{Lax53}. Observe that we only need the conclusion of Lemma~\ref{lemmeprodavant}, which is thus weaker than the condition \eqref{H1}. It is this latter lemma that motivates the introduction of the normalizing constant in $U_s$ (and thus in the ultra-differentiable norm); without this constant one would have $U_s^2 \ll c U_s$. Let us point out that the proof given below is elementary thanks to the factor $(|k|+1)^2$ in the definition of the norm~\eqref{defnorm} (which gives the factor $(l+1)^2$ in the defintion of $U_s$); without this factor, the statement is true (with a different normalizing constant) but the proof is more involved (see Lemma $2.7$ of~\cite{SCK03} for the sequence $M_\alpha$, $\alpha \geq 1$). 

\begin{proof}[Proof of Lemma~\ref{lemmeprod}]
Recall that
\[ U_{s}(X)=c^{-1}\sum_{l=0}^{+\infty}\frac{1}{(l+1)^2}\frac{M_l}{l!} \left(\frac{X}{s}\right)^l=c^{-1}\sum_{l=0}^{+\infty}\frac{N_l}{(l+1)^2}\left(\frac{X}{s}\right)^l\]
and so the assertion of the lemma amounts to prove that for all $l \in \N$,
\begin{equation}\label{toprove}
\sum_{j=0}^l\frac{N_jN_{l-j}}{(j+1)^2(l-j+1)^2} \leq c \frac{N_l}{(l+1)^2}, \quad c=\frac{4\pi^2}{3}. 
\end{equation}
Observe that the sum in the left-hand side of~\eqref{toprove} is symmetric with respect to $j \mapsto l-j$, and that from Lemma~\ref{lemmeprodavant}, $N_jN_{l-j} \leq N_l$ for all $l \in \N$ and all $0 \leq j \leq l$. Hence,
\[  \sum_{j=0}^l\frac{N_jN_{l-j}}{(j+1)^2(l-j+1)^2} \leq 2\sum_{j=0}^{l/2}\frac{N_jN_{l-j}}{(j+1)^2(l-j+1)^2} \leq 2N_l \sum_{j=0}^{l/2}\frac{1}{(j+1)^2(l-j+1)^2}.  \]
Then for any $0 \leq j \leq l/2$, $(l-j+1)^2 \geq (l/2+1)^2 \geq (l+1)^2/4$, and therefore 
\[  \sum_{j=0}^l\frac{N_jN_{l-j}}{(j+1)^2(l-j+1)^2} \leq \frac{8N_l}{(l+1)^2}\sum_{j=0}^{l/2}\frac{1}{(j+1)^2} \leq \frac{8N_l}{(l+1)^2}\sum_{j=0}^{+\infty}\frac{1}{(j+1)^2}=\frac{4\pi^2}{3}\frac{N_l}{(l+1)^2} \]
which is the inequality we wanted to prove.
\end{proof}

Proposition~\ref{produit} can be extended to matrix-valued
functions. More precisely, given $f : B \rightarrow M_{m,p}(\R)$ and
$g : B \rightarrow M_{p,q}(\R)$ where
$f=(f_{i,j})_{1 \leq i \leq m, \; 1 \leq j \leq p}$ and
$g=(g_{j,k})_{1 \leq j \leq p, \; 1 \leq k \leq q}$, we define
$f \cdot g : B \rightarrow M_{m,q}(\R)$ by
$f\cdot g:=((f\cdot g)_{i,k})_{1 \leq i \leq m, \; 1 \leq k \leq q}$
where
\[ (f\cdot g)_{i,k}:=\sum_{j=1}^p f_{i,j}g_{j,k}. \]
Then the following statement is an obvious corollary of
Proposition~\ref{produit}.

\begin{corollary}\label{corproduit}
  Let $f \in \mathcal{U}_{s}(B,M_{m,p}(\R))$,
  $g \in \mathcal{U}_{s}(B,M_{p,q}(\R))$. Then
  $f\cdot g \in \mathcal{U}_{s}(B,M_{m,q}(\R))$ and we have
\[ |f\cdot g|_{s} \leq |f|_{s}|g|_{s}. \]
\end{corollary}

\subsection{Compositions}\label{sec:composition}

Our goal here is to prove that the space of $M$-ultra-differentiable
functions is stable under composition, provided $M$ satisfies \eqref{H1}, at
the expense of reducing the width $s>0$ as with derivatives. The main
point is that we will be able to reduce the width by a factor
arbitrarily close to one provided the composition involves a map
arbitrarily close to the identity.

But first we need to define two additional formal power series
associated to $U_s$, the latter being defined in~\eqref{defM}. So we
define $\bar{U}_s$ by
\begin{equation}\label{defbarM}
\bar{U}_s(X):=U_s(X)-U_s(0)=c^{-1}\sum_{l=1}^{+\infty}\frac{N_l}{(l+1)^2}
\left(\frac{X}{s}\right)^l 
\end{equation}
and $\tilde{U}_s$ by
\begin{equation}\label{deftildeM}
\tilde{U}_s(X):=c^{-1}\sum_{l=0}^{+\infty}\frac{N_{l+1}}{(l+2)^2}
\left(\frac{X}{s}\right)^l. 
\end{equation}
It is clear that
\begin{equation}\label{barMtildeM}
s^{-1}X\tilde{U}_s(X)=\bar{U}_s(X).
\end{equation}

\begin{lemma}\label{lemmecomp}
We have
\[ \tilde{U}_s^2 \ll \tilde{U}_s, \quad \tilde{U}_s\bar{U}_s \ll  \bar{U}_s. \]
\end{lemma}

As for Lemma~\ref{lemmeprod}, the factor $(l+1)^2$ in the definition of $U_s$ makes the proof simple, but the statement is still true with this factor (see Lemma $2.4$ of~\cite{SCK03} for the sequence $M_\alpha$, $\alpha \geq 1$). To prove Lemma~\ref{lemmecomp}, we shall need the following preliminary lemma, in which we use the assumption \eqref{H1}.

\begin{lemma}\label{lemmecompavant}
If $M$ satisfies \eqref{H1}, then 
\[ N_lN_k \leq N_{l+k-1}, \quad (l,k)\in \N^2\setminus\{0\}. \]
\end{lemma}

\begin{proof}
We assume, without loss of generality, that $l \geq 1$. Since the sequence $\nu$ is increasing and $N_0=1$, we have
\[ N_lN_k=\nu_0\dots\nu_{l-1}\nu_0\nu_1\dots\nu_{k-1}\leq\nu_0\dots\nu_{l-1}\nu_0\nu_l\dots\nu_{l+k-2}=\nu_0N_{l+k-1} \]
and the result follows since we normalized $\nu_0=N_1/N_0=1$.
\end{proof}

Observe that exactly as for Lemma~\ref{lemmeprod}, we only need the conclusion of Lemma~\ref{lemmecompavant} to hold true, which is thus weaker than the condition \eqref{H1}.

\begin{proof}[Proof of Lemma~\ref{lemmecomp}]
  It is enough to prove the first part of the statement, as the second
  part of the statement follows from it; indeed, if
  $\tilde{U}_s^2 \ll \tilde{U}_s$, then using~\eqref{barMtildeM} we
  have
  \[ \tilde{U}_s(X)\bar{U}_s(X)=s^{-1}X\tilde{U}_s(X)\tilde{U}_s(X)
  \ll s^{-1}X\tilde{U}_s(X)=\bar{U}_s(X). \]
  As in Lemma~\ref{lemmeprod}, to prove that
  $\tilde{U}_s^2 \ll \tilde{U}_s$ one needs to show
  \begin{equation}\label{toprove2}
    \sum_{j=0}^l\frac{N_{j+1}N_{l-j+1}}{(j+2)^2(l-j+2)^2} \leq c
    \frac{N_{l+1}}{(l+2)^2}, \quad c=\frac{4\pi^2}{3}.  
  \end{equation}
  The sum in the left-hand side of~\eqref{toprove2} is still symmetric
  with respect to $j \mapsto l-j$, and for any $l \in \N$ and any
  $0 \leq j \leq l$, we have $N_{j+1}N_{l-j+1} \leq N_{l+1}$ by
  Lemma~\ref{lemmecompavant}. Therefore
\[
\sum_{j=0}^l\frac{N_{j+1}N_{l-j+1}}{(j+2)^2(l-j+2)^2}
\leq 2\sum_{j=0}^{l/2}\frac{N_{j+1}N_{l-j+1}}{(j+2)^2(l-j+2)^2} \leq 
2N_{l+1}\sum_{j=0}^{l/2}\frac{1}{(j+2)^2(l-j+2)^2}.  \]
Moreover, as in Lemma~\ref{lemmeprod}, for any $0 \leq j \leq l/2$, $(l-j+2)^2 \geq (l/2+2)^2 \geq (l+2)^2/4$, and so the last inequality leads to
\[  \sum_{j=0}^l\frac{N_{j+1}N_{l-j+1}}{(j+2)^2(l-j+2)^2} \leq \frac{8N_{l+1}}{(l+2)^2}\sum_{j=0}^{l/2}\frac{1}{(j+2)^2} \leq \frac{4\pi^2}{3} \frac{N_{l+1}}{(l+2)^2}\]
which concludes the proof.
\end{proof}

\begin{proposition}\label{composition}
  Let $f \in \mathcal{U}_{s}(B,\R^p)$, $0 < \sigma < 1$,
  $g \in \mathcal{U}_{s(1-\sigma)}(D,\R^p)$ where $B$ and $D$ are open
  balls of $\R^m$ such that $g(D)\subseteq B$. If
  \begin{equation}\label{small}
    |g-\mathrm{Id}|_{s(1-\sigma)} \leq s\sigma,
  \end{equation}  
  then $f \circ g \in \mathcal{U}_{s(1-\sigma)}(D,\R^p)$ and
 \[ |f \circ g|_{s(1-\sigma)} \leq |f|_{s}. \]
\end{proposition}

As it will be clear in the proof, the conclusions of
Proposition~\ref{composition} holds true under the slightly weaker
assumption that
\[ |u|_{s(1-\sigma)}-\sup_{x \in B}|u(x)| \leq s\sigma \]
but this will not be needed.

\begin{proof}
  Let $u=g-\mathrm{Id} \in \mathcal{U}_{s(1-\sigma)}(D,\R^p)$ and
  \[ a:=|f|_{s}, \quad b:=|u|_{s(1-\sigma)} \]
  so that, from Proposition~\ref{propnorme},
  \[ f(x) \ll_B aU_s(X), \quad u(x) \ll_D bU_{s(1-\sigma)}(X) \]
  and consequently
  \[ f(x) \ll_B aU_s(X), \quad g(x) \ll_D X+bU_{s(1-\sigma)}(X). \]
  We now apply Lemma~\ref{comp} and, recalling the definition of
  $\bar{U}_s$ and $\tilde{U}_s$ given respectively in~\eqref{defbarM}
  and~\eqref{deftildeM}, we obtain
\begin{eqnarray}
f(g(x)) & \ll_B & aU_s\left(X+b\bar{U}_{s(1-\sigma)}(X)\right) \nonumber \\ \nonumber
& = & aU_s(0)+a\bar{U}_s\left(X+b\bar{U}_{s(1-\sigma)}(X)\right) \\ \nonumber
& = & aU_s(0)+a\bar{U}_s\left(X+b(s(1-\sigma))^{-1}X\tilde{U}_{s(1-\sigma)}(X)\right) \\ \nonumber
& = & aU_s(0)+a\sum_{l=1}^{+\infty}\frac{N_l}{(l+1)^2} \left(\frac{X+b(s(1-\sigma))^{-1}X\tilde{U}_{s(1-\sigma)}(X)}{s}\right)^l \\ \label{estder}
& = & aU_s(0)+a\sum_{l=1}^{+\infty}\frac{N_l}{(l+1)^2} \left(\frac{X}{s}\right)^l\left(1+b(s(1-\sigma))^{-1}\tilde{U}_{s(1-\sigma)}(X)\right)^l.
\end{eqnarray}
From the first part of Lemma~\ref{lemmecomp}, for any $j \in \N$, we have
\[ \tilde{U}_{s(1-\sigma)}^j \ll \tilde{U}_{s(1-\sigma)} \]
and therefore
\begin{eqnarray*}
\left(1+b(s(1-\sigma))^{-1}\tilde{U}_{s(1-\sigma)}(X)\right)^l
& = & \sum_{j=0}^l \binom{l}{j}b^j(s(1-\sigma))^{-j}\tilde{U}_{s(1-\sigma)}(X)^j \\
& \ll & \tilde{U}_{s(1-\sigma)}(X)\sum_{j=0}^l \binom{l}{j}b^j(s(1-\sigma))^{-j} \\
& = & \tilde{U}_{s(1-\sigma)}(X)\left(1+b(s(1-\sigma))^{-1}\right)^l.
\end{eqnarray*}
Now, from~\eqref{small} we get
\[ b=|u|_{s(1-\sigma)} \leq s\sigma \]
and thus
\[ 1+ b(s(1-\sigma))^{-1} \leq 1+\sigma(1-\sigma)^{-1}=(1-\sigma)^{-1}.  \]
This gives
\[ \left(1+b(s(1-\sigma))^{-1}\tilde{U}_{s(1-\sigma)}(X)\right)^l \ll \tilde{U}_{s(1-\sigma)}(X)(1-\sigma)^{-l}  \]
which, together with~\eqref{estder}, yields
\begin{eqnarray*}
f(g(x)) & \ll_D & aU_s(0)+a\tilde{U}_{s(1-\sigma)}(X)\sum_{l=1}^{+\infty}\frac{N_l}{(l+1)^2} \left(\frac{X}{(s(1-\sigma))}\right)^l \\
& = & aU_s(0)+a\tilde{U}_{s(1-\sigma)}(X)\bar{U}_{s(1-\sigma)}(X).
\end{eqnarray*}
Using the second part of Lemma~\ref{lemmecomp}, this gives
\[ f(g(x))  \ll_D aU_s(0)+a\bar{U}_{s(1-\sigma)}(X)\]
and since $U_s(0)=U_{s(1-\sigma)}(0)$, we arrive at
\[ f(g(x))  \ll_D a(U_{s(1-\sigma)}(0)+\bar{U}_{s(1-\sigma)}(X))=aU_{s(1-\sigma)}(X).\]
Using Proposition~\ref{propnorme}, we eventually obtain
\[ |f \circ g|_{s(1-\sigma)} \leq a= |f|_{s} \]
and this concludes the proof.
\end{proof}

\subsection{Flows}\label{sec:flows}

The goal of this section is to show that the flow of a $M$-ultra-differentiable vector field is $M$-ultra-differentiable, under the assumptions \eqref{H1} and \eqref{H2}, and to obtain a norm estimate on this flow in terms of the norm of the vector field; to prove this we will use the estimates on derivatives, products and composition we already obtained in \S~\ref{sec:derivatives}, \S~\ref{sec:products} and \S~\ref{sec:composition} and the classical contraction fixed point theorem. 

Here we will consider a situation adapted to the applications in \S~\ref{sec:KAM} and \S~\ref{sec:Nek}; that is we consider functions $H=H(\theta,I,\omega)$ which are defined and ultra-differentiable on a domain of the form
\[ \T^n \times D_r \times D_h(\omega_0) \subseteq \T^n \times \R^n \times \R^n \]
where $D_r$ is the open ball of radius $r>0$ centered at the origin and $D_h(\omega_0)$ is the open ball of radius $h>0$ centered at a vector $\omega_0 \in \R^n$. In the statements below, the variables $\omega \in D_h(\omega_0)$ play the role of a fixed parameter, hence to simplify the notations we will explicitly suppress the dependence on $\omega \in D_h(\omega_0)$. 

Let us first start with a vector-valued function $D: \T^n \rightarrow \R^n$ which depends only on $\theta \in \T^n$, and that we shall considered as a vector field on $\T^n$.

\begin{lemma}\label{flot}
Given $D \in \mathcal{U}_{s}(\T^n,\R^n)$, let $0 < \sigma < 1$ and assume that
\begin{equation}\label{small2}
|D|_{s} \leq s\sigma.
\end{equation}
Then for any $t \in [0,1]$, the time-$t$ map $D^t$ of the flow of $D$ belongs to $\mathcal{U}_{s(1-\sigma)}(\T^n,\T^n)$ and we have the estimate
\begin{equation}\label{estflot}
|D^t-\mathrm{Id}|_{s(1-\sigma)^2} \leq t|D|_{s}.
\end{equation} 
\end{lemma} 

Observe that the width decreases from $s$ to $s(1-\sigma)^2$ because
we will have to apply consecutively Propositions~\ref{composition}
and~\ref{derivative} (or, more precisely,
Corollary~\ref{corderivative}), and each application induces a
decrease of $s$ by a factor $1-\sigma$.

\begin{proof}
The fact that $D^t$ is smooth and defined for all $t \in [0,1]$ (in fact, for all $t \in \R$) follows from the compactness of $\T^n$ and the classical result on the existence and uniqueness of solutions of differential equations (even though this will essentially be re-proved below); the only thing we need to prove is the estimate~\eqref{estflot}, and clearly it is sufficient to prove that for all $t \in [0,1]$,
\begin{equation*}
|D^t-\mathrm{Id}|_{s(1-\sigma)^2} \leq |D|_{s}.
\end{equation*} 
So let us consider the space $V:=C([0,1],\mathcal{U}_{s(1-\sigma)^2}(\T^n,\T^n))$ of continuous map from $[0,1]$ to $\mathcal{U}_{s(1-\sigma)^2}(\T^n,\T^n)$: given an element $\Phi \in V$ and $t \in [0,1]$, we shall write $\Phi^t:=\Phi(t)$ and consequently $\Phi=(\Phi^t)_{t \in [0,1]}$. We equip $V$ with the following norm:
\[ ||\Phi||:=\max_{t \in [0,1]}|\Phi^t|_{s(1-\sigma)^2} \]
which makes it a Banach space, and if we set $\rho:=|D|_{s}$, we define
\[ B_\rho V:=\{ \Phi \in V \; | \; ||\Phi-\mathrm{Id}|| \leq \rho \}.  \]
We can eventually define a Picard operator $P$ associated to $D$ by
\[ P : B_\rho V \rightarrow V, \quad \Phi \mapsto P(\Phi) \]
where $P(\Phi)=(P(\Phi)^t)_{t \in [0,1]}$ is defined by
\[ P(\Phi)^t:=\mathrm{Id}+\int_0^t D \circ \Phi^\tau d\tau. \]
To prove the lemma, it is sufficient to prove that $P$ has a unique fixed point $\Phi_* \in B_\rho V$, as necessarily $(\Phi_*^t)_{t \in [0,1]}=(D^t)_{t \in [0,1]}$. Therefore it is sufficient to prove that $P$ sends $B_\rho V$ into itself and that one of its iterate is a contraction, as $B_\rho V$ is a complete subset of the Banach space $V$.

First we need to show that $P$ maps $B_\rho V$ into itself. So assume $\Phi \in B_\rho V$, using~\eqref{small2} this implies that for all $t \in [0,1]$, 
\[ |\Phi^t-\mathrm{Id}|_{s(1-\sigma)^2} \leq \rho \leq s\sigma. \]
Letting $\sigma^*$ be such that $(1-\sigma)^2=1-\sigma^*$, we have $\sigma<\sigma^*$ and therefore
\[ |\Phi^t-\mathrm{Id}|_{s(1-\sigma^*)}=|\Phi^t-\mathrm{Id}|_{s(1-\sigma)^2} \leq s\sigma \leq s\sigma^*   \]
so that~\eqref{small} of Proposition~\ref{composition} is satisfied (with $f=D$ and $g=\Phi^t$ for any $t \in [0,1]$) and the latter proposition applies: this gives
\[ |D \circ \Phi^t|_{s(1-\sigma)^2}=|D \circ \Phi^t|_{s(1-\sigma^*)} \leq |D|_{s}=\rho, \quad t \in [0,1] \]
hence
\[ \left|P(\Phi)^t-\mathrm{Id}\right|_{s(1-\sigma)^2}=\left|\int_0^t D \circ \Phi^\tau d\tau\right|_{s(1-\sigma)^2} \leq t\rho \leq \rho, \quad t \in [0,1] \]
and therefore
\[ ||P(\Phi)-\mathrm{Id}|| \leq \rho. \]
This proves that $P$ maps $B_\rho V$ into itself.

It remains to show that some iterate of $P$ is a contraction. So let $\Phi_1,\Phi_2 \in B_\rho V$, then for any $t \in [0,1]$,
\begin{eqnarray*}
 P(\Phi_1)^t-P(\Phi_2)^t & = & \int_0^t \left(D \circ \Phi_1^\tau-D \circ \Phi_2^\tau \right)d\tau \\
& = & \int_0^t \left(\int_0^1 \nabla D \circ (s\Phi_1^\tau+(1-s)\Phi_2^\tau)ds \right) \cdot (\Phi_1^\tau-\Phi_2^\tau) d\tau \\
& = & \int_0^t Y^\tau  \cdot (\Phi_1^\tau-\Phi_2^\tau) d\tau. 
 \end{eqnarray*}
Using~\eqref{small}, we can apply Proposition~\ref{composition}, and together with Corollary~\ref{corderivative}, we obtain, for any $0 \leq \tau \leq t$,
\[ |Y^\tau|_{s(1-\sigma)^2} \leq |\nabla D|_{s(1-\sigma)} \leq n^2 s^{-1}C(\sigma)|D|_{s}.   \]
Using Corollary~\ref{corproduit}, this gives, for any $t \in [0,1]$,
\[ |P(\Phi_1)^t-P(\Phi_2)^t|_{s(1-\sigma)^2} \leq t n^2 s^{-1}C(\sigma)|D|_{s}\max_{0 \leq \tau \leq t}|\Phi_1^\tau-\Phi_2^\tau|_{s(1-\sigma)^2}  \]
and hence, setting $\kappa:=n^2s^{-1}C(\sigma)|D|_s$,
\[ \max_{0 \leq \tau \leq t}|P(\Phi_1)^\tau-P(\Phi_2)^\tau|_{s(1-\sigma)^2} \leq \kappa t\max_{0 \leq \tau \leq t}|\Phi_1^\tau-\Phi_2^\tau|_{s(1-\sigma)^2}\]
A well-known induction then gives, for any $j \in \N$,
\[ \max_{0 \leq \tau \leq t}|P^j(\Phi_1)^\tau-P^j(\Phi_2)^\tau|_{s(1-\sigma)^2} \leq \frac{(\kappa t)^j}{j!}\max_{0 \leq \tau \leq t}|\Phi_1^\tau-\Phi_2^\tau|_{s(1-\sigma)^2}\]
and hence setting $t=1$
\[ ||P^j(\Phi_1)-P^j(\Phi_2)||\leq \frac{\kappa^j}{j!}||\Phi_1-\Phi_2||. \]
This proves that some iterate of $P$ is a contraction, which concludes the proof.  
\end{proof}

Now let us consider a Hamiltonian function $X$ on $\T^n \times D_r$, of the form
\begin{equation}\label{HamX}
X(\theta,I):=C(\theta)+D(\theta) \cdot I, \quad  C:\T^n \rightarrow \R, \quad D : \T^n \rightarrow \R^n.
\end{equation}
The Hamiltonian equations associated to $X$ are given by:
\[ 
\begin{cases}
\dot{\theta}(t)=\nabla_I X(\theta(t),I(t))=D(\theta(t)), \\
\dot{I}(t)=-\nabla_\theta X(\theta(t),I(t))=-\nabla C(\theta(t))-\nabla D(\theta(t))\cdot I.
\end{cases}
\]
The equations for $\theta$ are uncoupled from the equations of $I$ (and hence can be integrated independently), while the equations for $I$ are affine in $I$; it is well-known that these facts lead to a simple form of the Hamiltonian flow associated to $X$ (see, for instance,~\cite{Vil08}).

\begin{proposition}\label{flotX}
Let $X$ be as in~\eqref{HamX} with $C \in \mathcal{U}_{s}(\T^n,\R)$ and $D \in \mathcal{U}_{s}(\T^n,\R^n)$. Let $0 < \sigma < 1$ and assume that
\begin{equation}\label{smallX}
|D|_{s} \leq s\sigma.
\end{equation}
Then for any $t \in [0,1]$, the time-$t$ map $X^t$ of the Hamiltonian flow of $X$ is of the form
\[ X^t(\theta,I)=(\theta+E^t(\theta),I+F^t(\theta)\cdot I+G^t(\theta)) \]
where $E^t \in \mathcal{U}_{s(1-\sigma)^2}(\T^n,\R^n)$, $F^t \in \mathcal{U}_{s(1-\sigma)^2}(\T^n,\R^{n^2})$ and $G^t \in \mathcal{U}_{s(1-\sigma)^2}(\T^n,\R^n)$ with the estimates
\begin{equation}\label{estflotX}
\begin{cases}
|E^t|_{s(1-\sigma)^2} \leq t|D|_{s}, \\
|F^t|_{s(1-\sigma)^2} \leq n^2 s^{-1}C(\sigma)|D|_{s}\exp(t n^2 s^{-1}C(\sigma)|D|_{s}), \\ 
|G^t|_{s(1-\sigma)^2} \leq n s^{-1}C(\sigma)|C|_{s}\exp(t n^2 s^{-1}C(\sigma)|D|_{s}). 
\end{cases}
\end{equation} 
As a consequence, given $0<\delta<r$, if we further assume that 
\begin{equation}\label{smallXX}
ns^{-1}C(\sigma)\exp(t n^2 s^{-1}C(\sigma)|D|_{s})(rn|D|_s+|C|_s) \leq \delta 
\end{equation}
then $X^t$ maps $\T^n \times D_{r-\delta}$ into $\T^n \times D_{r}$. 
\end{proposition} 

\begin{proof}
The second part of the statement clearly follows from the first part, so let us prove the latter. From the specific form of the Hamiltonian equations associated to $X$, one has, for any $t \in [0,1]$,
\[ X^t(\theta,I)=(\theta+E^t(\theta),I+F^t(\theta)\cdot I+G^t(\theta)) \]
with
\[  
\begin{cases}
E^t(\theta)=\int_{0}^t D(\theta+E^\tau(\theta))d\tau, \\
F^t(\theta)=-\int_{0}^t \nabla D(\theta+E^\tau(\theta))d\tau-\int_{0}^t \nabla D(\theta+E^\tau(\theta))\cdot F^\tau(\theta)d\tau \\
G^t(\theta)=-\int_{0}^t \nabla C(\theta+E^\tau(\theta))d\tau-\int_{0}^t \nabla D(\theta+E^\tau(\theta))\cdot G^\tau(\theta)d\tau. 
\end{cases}
\]
Because of~\eqref{smallX}, Lemma~\ref{flot} applies and the flow $D^t(\theta)=\theta+E^t(\theta)$ satisfies~\eqref{estflot}, and therefore 
\[|E^t|_{s(1-\sigma)^2}=|D^t-\mathrm{Id}|_{s(1-\sigma)^2} \leq t|D|_{s}\]
which gives the first estimate of~\eqref{estflotX}. Using this estimate and~\eqref{smallX}, we can apply Proposition~\ref{composition} and Corollary~\ref{corderivative} to obtain, for any $0 \leq \tau \leq  t \leq 1$,
\[ |\nabla D \circ D^\tau|_{s(1-\sigma)^2} \leq |\nabla D|_{s(1-\sigma)} \leq n^2 s^{-1}C(\sigma)|D|_{s}.   \]
Looking at the expression of $F^t$, this gives
\[ |F^t|_{s(1-\sigma)^2} \leq n^2 s^{-1}C(\sigma)|D|_{s}\left(1+\int_0^t |F^\tau|_{s(1-\sigma)^2} d\tau \right) \]
which, by Gronwall's inequality, implies that for all $t \in [0,1]$,
\[ |F^t|_{s(1-\sigma)^2} \leq n^2 s^{-1}C(\sigma)|D|_{s}\exp(t n^2 s^{-1}C(\sigma)|D|_{s}) \]
which gives the second estimate of~\eqref{estflotX}. For the third estimate of~\eqref{estflotX}, observe that the same argument yields
\[ |G^t|_{s(1-\sigma)^2} \leq n s^{-1}C(\sigma)|C|_{s}+n^2 s^{-1}C(\sigma)|D|_{s}\int_0^t |G^\tau|_{s(1-\sigma)^2} d\tau  \]
and again, by Gronwall's inequality, for all $t \in [0,1]$ we have
\[ |G^t|_{s(1-\sigma)^2} \leq n s^{-1}C(\sigma)|C|_{s}\exp(t n^2 s^{-1}C(\sigma)|D|_{s}) . \]
This concludes the proof.
\end{proof}

The last lemma gives estimates on the Hamiltonian flow associated to a
Hamiltonian $X$ on $\T^n \times D_r$ which is affine in $I$; the flow
has then a special form and precise estimates can be obtained (they
are needed in \S~\ref{sec:KAM}). On the other hand, in
\S~\ref{sec:Nek}, we will need estimates valid for an arbitrary
Hamiltonian $Y$ on $\T^n \times D_r$; the lemma below gives rougher
estimates in this general case but they will prove sufficient for our
purpose.

\begin{proposition}\label{flotF}
Given $Y \in \mathcal{U}_{s}(\T^n \times D_r)$, let $0 < \sigma < 1$ and assume that
\begin{equation}\label{smallF}
|Y|_{s} \leq (2n)^{-1}s^2\sigma C(\sigma)^{-1}.
\end{equation}
Then for any $t \in [0,1]$, the time-$t$ map $Y^t$ of the Hamiltonian flow of $Y$ belongs to $\mathcal{U}_{s(1-\sigma)^3}(\T^n \times D_{r(1-\sigma)^3},\T^n \times D_{r(1-\sigma)^2})$ and we have the estimate
\begin{equation}\label{estflotF}
|Y^t-\mathrm{Id}|_{s(1-\sigma)^3} \leq t|Y|_{s}.
\end{equation} 
\end{proposition} 

\begin{proof}
The proof is just a variation of the proof of Lemma~\ref{flot}, so we just point out the differences. As before, we know that the flow $Y^t$ exists and is smooth, so we only need to prove that $Y^t$ maps $\T^n \times D_{r(1-\sigma)^3}$ into $\T^n \times D_{r(1-\sigma)^2}$ and that the estimate
\begin{equation}\label{estflotFF}
|Y^t-\mathrm{Id}|_{s(1-\sigma)^3} \leq |Y|_{s}
\end{equation} 
holds true. Let $X_Y$ be the Hamiltonian vector field: it follows from Corollary~\ref{corderivative} and~\eqref{smallF} that
\begin{equation}\label{trucF}
|X_Y|_{s(1-\sigma)} \leq 2ns^{-1}C(\sigma)|Y|_s \leq s\sigma  
\end{equation}
and thus $Y^t$ maps $\T^n \times D_{r(1-\sigma)^i}$ into $\T^n \times D_{r(1-\sigma)^{i-1}}$ for $i=1,2,3$. To prove the estimate~\eqref{estflotFF}, at this point one can just repeat the proof of Lemma~\ref{flot}, with $s(1-\sigma)$ instead of $s$, $X_Y$ instead of $D$; the required condition~\eqref{small2} translates into $|X_Y|_{s(1-\sigma)} \leq s\sigma$ which is nothing but~\eqref{trucF}.
\end{proof}

\subsection{Inverse functions}\label{sec:inverse}

In this last section, we shall prove that if an
$M$-ultra-differentiable map is sufficiently close to the identity,
then its local inverse is still $M$-ultra-differentiable, provided $M$
satisfies \eqref{H1} and \eqref{H2}. Exactly like in \S~\ref{sec:flows}, we will
use the estimates on derivatives, products and composition we obtained
in \S~\ref{sec:derivatives}, \S~\ref{sec:products} and
\S~\ref{sec:composition} and the classical contraction fixed point
theorem.

Again, to prove this in a setting adapted to \S~\ref{sec:KAM}, let us consider a map $\phi$ which depends only on $\omega \in D_h(\omega_0)$, that
is $\phi : D_h(\omega_0) \rightarrow \R^n$.

\begin{proposition}\label{propomega}
  Given $\phi \in \mathcal{U}_{s}(D_h(\omega_0),\R^n)$, let $0 < \sigma < 1$ and
  assume that
  \begin{equation}\label{smallomega}
    |\phi-\mathrm{Id}|_{s} < n^{-2}sC(\sigma)^{-1}, \quad
    |\phi-\mathrm{Id}|_{s} \leq h/2.  
  \end{equation}
  Then there exists a unique
  $\varphi \in G_{s(1-\sigma)^2}(D_{h/2}(\omega_0),D_h(\omega_0))$ such that
  $\phi \circ \varphi=\mathrm{Id}$ and
  \begin{equation}\label{estomega}
    |\varphi-\mathrm{Id}|_{s(1-\sigma)^2} \leq |\phi-\mathrm{Id}|_{s}. 
  \end{equation}
\end{proposition}

\begin{proof}
  Let us define $V:=\mathcal{U}_{s(1-\sigma)^2}(D_{h/2}(\omega_0),\R^n)$, which is a
  Banach space with the norm $||\,.\,||=|\,.\,|_{s(1-\sigma)^2}$,
  and for $\rho:=|\phi-\mathrm{Id}|_{s}$, we set
\[ B_\rho V:=\{\psi \in V \; | \; ||\psi-\mathrm{Id}|| \leq \rho\}. \]
Let us define the following Picard operator $P$ associated to $\phi$:
\[ P : B_\rho V \rightarrow V, \quad \psi \mapsto P(\psi)=\mathrm{Id}-(\phi-\mathrm{Id})\circ\psi. \]
It is clear that $\phi \circ \varphi=\mathrm{Id}$ if and only if $\varphi$ is a fixed point of $P$, and therefore the proposition will be proved once we have shown that $P$ has a unique fixed point in $B_\rho V$, and to do this it is enough to prove that $P$ is a well-defined contraction of $B_\rho V$.

First let us prove that $P$ maps $B_\rho V$ into itself. So let $\psi \in B_\rho V$, and using the second part of~\eqref{smallomega}, observe that since
\[ \sup_{\omega \in D_{h/2}(\omega_0)}|\psi(\omega)-\omega|\leq ||\psi-\mathrm{Id}||\leq \rho \leq h/2  \]
then $\psi$ maps $D_{h/2}(\omega_0)$ into $D_h(\omega_0)$. Now recall that for any $0<\sigma<1$, $C(\sigma) \geq (e\sigma)^{-1}$ so the first part of~\eqref{smallomega} gives in particular
\[ ||\psi-\mathrm{Id}||\leq \rho \leq s\sigma \]
and allows us to apply Proposition~\ref{composition} (with $\sigma^*>\sigma$ defined by $(1-\sigma)^2=(1-\sigma^*)$) to get
\[ ||(\phi-\mathrm{Id})\circ\psi||=|(\phi-\mathrm{Id})\circ\psi|_{s(1-\sigma)^2} \leq |\phi-\mathrm{Id}|_{s}=\rho \]
and thus 
\[ ||P(\psi)-\mathrm{Id}||=||(\phi-\mathrm{Id})\circ\psi|| \leq \rho, \]
that is, $P$ maps $B_\rho V$ into itself. To show that $P$ is a contraction, using Corollary~\ref{corderivative}, Corollary~\ref{corproduit} and Proposition~\ref{composition} one gets (exactly as in the proof of Lemma~\ref{flot}) for any $\psi_1,\psi_2 \in B_\rho V$:
\[ ||P(\psi_1)-P(\psi_2)||=||(\phi-\mathrm{Id})\circ\psi_1-(\phi-\mathrm{Id})\circ\psi_2|| \leq n^2s^{-1}C(\sigma)|\phi-\mathrm{Id}|_{\alpha,s} ||\psi_1-\psi_2|| \] 
and thus the fact that $P$ is a contraction follows from the first part of~\eqref{smallomega}. This ends the proof.  
\end{proof}

\section{Application to KAM theory}\label{sec:KAM}

In this section, we give applications to KAM theory, namely we prove Theorem~\ref{classicalKAM}, Theorem~\ref{KAMvector} and Theorem~\ref{destruction}; Theorem~\ref{classicalKAM} and Theorem~\ref{KAMvector} will be deduced from a KAM theorem with parameters (Theorem~\ref{KAMparameter} below) and Theorem~\ref{destruction} is an adaptation of a result of~\cite{Bes00}. All the other KAM theorems we stated can be deduced from Theorem~\ref{classicalKAM}, Theorem~\ref{KAMvector} or Theorem~\ref{KAMparameter} so we will not give details about their proof.    

In the special case where the Hamiltonians are Gevrey regular, that is $M=M_\alpha$ for $\alpha \geq 1$, these results are contained in~\cite{BFa17} and so it is our purpose here to extend them to the more general setting we are considering. 

\subsection{Statement of the KAM theorem with parameters}
\label{sec:KAMparam}

Let us now consider the following setting. Fix
$\omega_0 \in \R^n \setminus \{0\}$. Re-ordering the components of $\omega_0$ and re-scaling the Hamiltonian allow us to assume without loss of generality that
\[\omega_0=(1,\bar{\omega}_0) \in \R^n, \quad \bar{\omega}_0 \in
[-1,1]^{n-1}.\]
Given real numbers $r>0$ and $h>0$, we let
\[ D_r:=\{I \in \R^n \; | \; |I| < r \}, \quad D_h(\omega_0):=\{\omega \in
\R^n \; | \; |\omega-\omega_0| < h \}. \]
Since $\omega_0$ is fixed, for simplicity we shall remove it from the notation write $D_h=D_h(\omega_0)$, and we set
\[ D_{r,h}:=D_r \times
D_h.  \]
Our Hamiltonians will be defined on $\T^n \times D_{r,h}$, a
neighborhood of $\T^n \times \{0\} \times \{\omega_0\}$ in
$\T^n \times \R^n \times \R^n$.

Let $s>0$, $\eta \geq 0$ a fixed parameter,
$\varepsilon \geq 0$ and $\mu \geq 0$ two small parameters. We
consider a function $H \in \mathcal{U}_{s}(\T^n \times D_{r,h})$ of the
form
\begin{equation}\label{HAM}
\begin{cases}
  H(\theta,I,\omega) = \underbrace{e(\omega)+\omega \cdot
    I}_{N(I,\omega)} +
  \underbrace{A(\theta,\omega)+B(\theta,\omega)\cdot
    I}_{P(\theta,I,\omega)}+ 
  \underbrace{M(\theta,I,\omega) \cdot I^2}_{R(\theta,I,\omega)} \\
|A|_{s}\leq \varepsilon, \quad |B|_{s} \leq \mu, \quad
  |\nabla_I^2 R|_{s} \leq \eta
\end{cases}     
  \tag{$**$}
\end{equation}
where the notation $M(\theta,I,\omega) \cdot I^2$ stands for the vector $I$
given twice as an argument to the symmetric bilinear form
$M(\theta,I,\omega)$. Observe that
$A : \T^n \times D_h \rightarrow \R$,
$B : \T^n \times D_h \rightarrow \R^n$ whereas
$M : \T^n \times D_{r,h} \rightarrow M_n(\R)$ with
$M_n(\R)$ the ring of real square matrices of size $n$.

The function $H$ in~\eqref{HAM} should be considered as an ultra-differentiable
Hamiltonian on $\T^n \times D_r$, depending on a parameter
$\omega \in D_h$; for a fixed parameter $\omega \in D_h$, when
convenient, we will write
\[ H_\omega (\theta,I)=H(\theta,I,\omega), \quad
N_\omega(I)=N(I,\omega), \quad P_\omega(\theta,I)=P(\theta,I,\omega),
\quad R_\omega(\theta,I)=R(\theta,I,\omega). \]
The image of the map $\Phi_0 : \T^n \rightarrow \T^n \times D_r$,
$\theta \mapsto (\theta,0)$ is an ultra-differentiable embedded torus in
$\T^n \times D_r$, invariant by the Hamiltonian flow of
$N_{\omega_0}+R_{\omega_0}$ and quasi-periodic with frequency
$\omega_0$. The next theorem asserts that this quasi-periodic torus
will persist, being only slightly deformed, as an invariant torus not
for the Hamiltonian flow of $H_{\omega_0}$ but for the Hamiltonian
flow of $H_{\omega_0^*}$, where $\omega_0^*$ is a parameter close to
$\omega_0$, provided $\varepsilon$ and $\mu$ are sufficiently small
and $\omega_0$ satisfies the $\mathrm{BR_M}$-condition. Here
is the precise statement.

\begin{Main}\label{KAMparameter}
 Let $H$ be as in~\eqref{HAM}, with $M=(M_l)_{l \in \N}$ satisfying \eqref{H1} and \eqref{H2}, and $\omega_0$ satisfying~\eqref{BRM}.
There exist positive constants $c_1 \leq 1$, $c_2\leq 1$ and $c_3 \geq 1$ depending only on $n$ such that if
  \begin{equation}\label{cond0}
    \sqrt{\varepsilon} \leq \mu \leq h/2, \quad \sqrt{\varepsilon} \leq r, \quad h \leq c_1 (Q_0\Psi(Q_0))^{-1}
  \end{equation}
  where $Q_0 \geq n+2$ is sufficiently large so that
  \begin{equation}\label{eqQ0}
C^{-1}(c_2(1+\eta)^{-1}sQ_0)+(\ln 2)^{-1}\int_{\Delta(Q_0)}^{+\infty}C^{-1}(c_2(1+\eta)^{-1}s\Delta^{-1}(x))\frac{dx}{x} \leq \ln 2(4n+2)^{-1}
  \end{equation}
then the following holds true. There exist a vector $\omega_0^* \in \R^n$ and an
  $(M,s/2)$-ultra-differentiable embedding
  \[ \Phi^*_{\omega_0} : \T^n \times D_{r/2} \rightarrow \T^n \times D_r \] 
  of the form
  \[ \Phi^*_{\omega_0}(\theta,I)=(\theta+E^*(\theta),I+F^*(\theta)\cdot I+G^*(\theta)) \]
  with the estimates
  \begin{equation}\label{estim0}
    |\omega_0^*-\omega_0| \leq c_3 \mu, \quad |E^*|_{s/2} \leq
    c_3 \Psi(Q_0)\mu, \quad |F^*|_{s/2}\leq c_3
    \Delta(Q_0)\mu,   \quad |G^*|_{s/2}\leq c_3
    \Delta(Q_0)\varepsilon   
  \end{equation}
  and  such that 
  \[ H_{\omega_0^*} \circ \Phi_{\omega_0}^*(\theta,I)=e_0^*+\omega_0 \cdot I
  + R^*(\theta,I), \quad R^*(\theta,I) = M^*(\theta,I) \cdot I^2,\]
  with the estimates
  \[ |e_0^*-e_{\omega_0^*}| \leq c_3\varepsilon, \quad |\nabla_I^2 R^*-
  \nabla_I^2 R_{\omega_0^*}|_{s/2} \leq c_3\eta \Delta(Q_0)\mu.\]
\end{Main}

Theorem~\ref{classicalKAM} follows quite directly from
Theorem~\ref{KAMparameter}, introducing the frequencies
$\omega = \nabla h (I)$ as independent parameters, taking
$\mu=\sqrt{\varepsilon}$, and tuning the shift of frequency
$\omega_0^*-\omega_0$ using the non-degeneracy assumption on the
unperturbed Hamiltonian. Theorem~\ref{KAMvector} follows also from
Theorem~\ref{KAMparameter} by realizing $X$ as the restriction of a
Hamiltonian vector field on an invariant torus, setting
$\varepsilon=\eta=0$ and letting $\mu$ be the only small
parameter. These arguments are made precise in \S~\ref{sec:deduc1} -
\S~\ref{sec:deduc2}, but we first concentrate on the proof of
Theorem~\ref{KAMparameter}.

In all this section, we do not pay attention to how constants depend on the
dimension $n$ as it is fixed. Hence
in the sequel we shall write
$$u \MP v \quad (\mbox{respectively } u \PM v, \quad u \EP v, \quad u \PE v)$$
if, for some constant $c\geq 1$ depending only on $n$, we
have 
$$u \leq c v \quad (\mbox{respectively } cu \leq v, \quad u =c v, \quad cu = v).$$

\subsection{Approximation by rational vectors}\label{sec:approx}

Recall that we have written $\omega_0=(1,\bar{\omega}_0) \in \R^n$
with $\bar{\omega}_0 \in [-1,1]^{n-1}$. For a given $Q\geq 1$, it is
always possible to find a rational vector $v=(1,p/q) \in \Q^n$, with
$p \in \Z^{n-1}$ and $q \in \N$, which is a $Q$-approximation in the
sense that $|q\omega_0 - qv|\leq Q^{-1}$, and for which the
denominator $q$ satisfies the upper bound $q \leq Q^{n-1}$: this is
essentially the content of Dirichlet's theorem on simultaneous
rational approximations, and it holds true without any assumption on
$\omega_0$. In our situation, since we have assumed that $\omega_0$ is
non-resonant, there exist not only one, but $n$ linearly independent
rational vectors in $\Q^n$ which are $Q$-approximations. Moreover, one
can obtain not only linearly independent vectors, but rational vectors
$v_1,\dots,v_n$ of denominators $q_1, \dots,q_n$ such that the
associated integer vectors $q_1v_1,\dots,q_nv_n$ form a $\Z$-basis of
$\Z^n$. However, the upper bound on the corresponding denominators
$q_1, \dots, q_n$ is necessarily larger than $Q^{n-1}$, and is given
by a function of $Q$ that we can call here $\Psi_{\omega_0}'$
(see \cite{BF13} for more precise and general information, but note
that in this reference, $\Psi_{\omega_0}'$ was denoted by $\Psi_{\omega_0}$ and $\Psi_{\omega_0}$,
which we defined in~\eqref{eqpsi}, was denoted by $\Psi_{\omega_0}'$). A consequence of the main Diophantine result of \cite{BF13} is that this function $\Psi_{\omega_0}'$ is in fact essentially equivalent to the function $\Psi_{\omega_0}$.  

\begin{proposition}\label{dio}
Let $\omega_0=(1,\bar{\omega}_0) \in \R^n$ be a non-resonant vector, with $\bar{\omega}_0 \in [-1,1]^{n-1}$. For any $Q\geq n+2$, there exist $n$ rational vectors $v_1, \dots, v_n$, of denominators $q_1, \dots, q_n$, such that $q_1v_1, \dots, q_nv_n$ form a $\Z$-basis of $\Z^n$ and for $j\in\{1,\dots,n\}$,
\[ |\omega_0-v_j|\MP(q_j Q)^{-1}, \quad 1 \leq q_j \MP \Psi(Q).\]
\end{proposition}

For a proof of the above proposition with $\Psi_{\omega_0}$ instead of $\Psi$, we refer to \cite{BF13}, Theorem 2.1 and Proposition $2.3$; now by~\eqref{foncpsi}, $\Psi_{\omega_0} \leq \Psi$ and so one may replace $\Psi_{\omega_0}$ by $\Psi$. 

Now given a $q$-rational vector $v$ and a smooth function $H$ defined on $\T^n \times D_{r,h}$, we define
\begin{equation}\label{moyv}
[H]_v(\theta,I,\omega)= \int^{1}_{0} H(\theta+tqv,I,\omega)dt.
\end{equation}
Given $n$ rational vectors $v_1,\dots,v_n$, we let $[H]_{v_1,\dots,v_d}=[\cdots[H]_{v_1}\cdots]_{v_d}$. Finally we define
\begin{equation}\label{moy}
[H](I,\omega)=\int_{\T^n}H(\theta,I,\omega)d\theta.
\end{equation}
The following proposition is a consequence of the fact that the vectors $q_1v_1, \dots, q_nv_n$ form a $\Z$-basis of $\Z^n$. 

\begin{proposition}[{\cite[Corollary 6]{Bou13}}]\label{cordio}
Let $v_1,\dots,v_n$ be rational vectors, of denominators $q_1,\dots,q_n$, such that $q_1v_1, \dots, q_nv_n$ form a $\Z$-basis of $\Z^n$, and $H$ a function defined on $\T^n \times D_{r,h}$. Then
\[ [H]_{v_1,\dots,v_n}=[H]. \]  
\end{proposition} 

\subsection{KAM step}\label{sec:step}

Now we describe an elementary step of our iterative procedure. Such a
step consists in pulling back the Hamiltonian $H$ by a transformation
of the form
\[ \mathcal{F}=(\Phi,\varphi): (\theta,I,\omega) \mapsto
(\Phi(\theta,I,\omega),\varphi(\omega));\]
$\Phi$ is a parameter-depending change of coordinates and $\varphi$ a
change of parameters. Moreover, our change of coordinates will be of
the form
\[\Phi(\theta,I,\omega)=
\Phi_\omega(\theta,I)=(\theta+E(\theta,\omega),I+F(\theta,\omega)\cdot
I+G(\theta,\omega)) \] with
\[ E : \T^n \times D_h \rightarrow \R^n, \quad F : \T^n \times D_h
\rightarrow M_n(\R), \quad G : \T^n \times D_h \rightarrow \R^n \]
and for each fixed parameter $\omega$, $\Phi_\omega$ will be
symplectic. For simplicity, we shall write $\Phi=(E,F,G)$; the
composition of such transformations
$\mathcal{F}=(\Phi,\varphi)=(E,F,G,\varphi)$ is again a transformation
of the same form, and we shall denote by $\mathcal{G}$ the groupoid of
such transformations.

\begin{proposition}\label{kamstep}
  Let $H$ be as in~\eqref{HAM}, with $M=(M_l)_{l \in \N}$ satisfying \eqref{H1} and \eqref{H2} and
  $\omega_0=(1,\bar{\omega}_0) \in \R^n$ non-resonant. Consider
  $0<\sigma< 1$, $0<\delta<r$, $Q \geq n+2$, and assume that
  \begin{equation}\label{cond1}
    \sqrt{\varepsilon} \leq \mu \leq h/2, \quad
    \sqrt{\varepsilon} \leq r, \quad h \PM (Q\Psi(Q))^{-1}, \quad r\mu \PM \delta(Q\Psi(Q))^{-1},
    \quad (1+\eta) \PM QsC(\sigma)^{-1}. 
  \end{equation}
  Then there exists an $(M,s(1-\sigma)^{2n+1})$-ultra-differentiable 
  transformation
  \[ \mathcal{F}=(\Phi,\varphi)=(E,F,G,\varphi) : \T^n \times
  D_{r-\delta,h/2} \rightarrow \T^n \times D_{r,h} \in \mathcal{G}, \]
  with the estimates
  \begin{equation}\label{stepest1}
    \begin{cases}
|E|_{s(1-\sigma)^{2n}} \MP \Psi(Q)\mu, \quad |\nabla E|_{s(1-\sigma)^{2n+1}} \MP s^{-1}C(\sigma)\Psi(Q)\mu, \\
|F|_{s(1-\sigma)^{2n}} \MP s^{-1}C(\sigma)\Psi(Q)\mu, \quad |\nabla F|_{s(1-\sigma)^{2n+1}} \MP s^{-2}C(\sigma)^2\Psi(Q)\mu, \\
|G|_{s(1-\sigma)^{2n}} \MP s^{-1}C(\sigma)\Psi(Q)\varepsilon, \quad |\nabla G|_{s(1-\sigma)^{2n+1}} \MP s^{-2}C(\sigma)^2\Psi(Q)\varepsilon,   \\
|\varphi-\mathrm{Id}|_{s(1-\sigma)^{2n}} \leq \mu, \quad |\nabla \varphi-\mathrm{Id}|_{s(1-\sigma)^{2n+1}} \MP s^{-1}C(\sigma) \mu
\end{cases}
  \end{equation}
  such that 
  \[H \circ \mathcal{F}(\theta,I,\omega) =
  \underbrace{e^+(\omega)+\omega\cdot I }_{N^+(I,\omega)} +
  \underbrace{A^+(\theta)+B^+(\theta)\cdot I}_{P^+(\theta,I,\omega)} +
  \underbrace{M^+(\theta,I,\omega)\cdot I^2}_{R^+(\theta,I,\omega)}, 
  \]
  with the estimates
  \begin{equation}\label{stepest2}
    \begin{cases}
      |A^+|_{s(1-\sigma)^{2n+1}} \leq \varepsilon/16, \quad
      |B^+|_{s(1-\sigma)^{2n+1}} \leq \mu/4, \\
      |e^+-e \circ \varphi|_{s(1-\sigma)^{2n+1}} \leq |A|_{s},
      \quad |\nabla_I^2 R^+- \nabla_I^2 R \circ
      \mathcal{F}|_{s(1-\sigma)^{2n+1}}\MP \eta |F|_{s(1-\sigma)^{2n}}.
    \end{cases}
  \end{equation}
\end{proposition}

\begin{proof}
  We divide the proof of the KAM step into five small steps. Except
  for the last one, the parameter $\omega \in D_h$ will be fixed, so
  for simplicity, in the first four steps we will drop the dependence
  on the parameter $\omega \in D_h$. Let us first notice
  that~\eqref{cond1} clearly implies the following seven inequalities:
  \begin{align}
    &h \PM (Q\Psi(Q))^{-1} \label{eq:cond1}\\ 
    &\Psi(Q)\mu \PM sC(\sigma)^{-1} \label{eq:cond2}\\ 
    &\varepsilon \leq r\mu \label{eq:cond3}\\
    &r\mu s^{-1}C(\sigma)\Psi(Q) \PM \delta  \label{eq:cond4}\\ 
    &\mu \PM (Q\Psi(Q))^{-1} \label{eq:cond5}\\ 
    &(1+\eta) \PM QsC(\sigma)^{-1} \label{eq:cond6}\\ 
    &\mu \leq h/2. \label{eq:cond7}
  \end{align}
It is also important to notice that the implicit constant appearing in~\eqref{eq:cond6} is independent of the other ones; we may choose it as large as we want without affecting the other implicit constants. In the first three steps, the term $R$ which contains terms of order at least $2$ in $I$ will be ignored, that is we will only consider $\hat{H}=H-R=N+P$.

\medskip

\textit{1. Rational approximations of $\omega_0$ and $\omega \in D_h$}

\medskip

Since $\omega_0$ is non-resonant, given $Q \geq n+2$, we can apply Proposition~\ref{dio}: there exist $n$ rational vectors $v_1, \dots, v_n$, of denominators $q_1, \dots, q_n$, such that $q_1v_1, \dots, q_nv_n$ form a $\Z$-basis of $\Z^n$ and for $j\in\{1,\dots,n\}$,
\[ |\omega_0-v_j|\MP(q_j Q)^{-1}, \quad 1 \leq q_j \MP \Psi(Q).\]
For any $\omega \in D_h$, using~\eqref{eq:cond1} and
$q_j \MP \Psi(Q)$, we have
\begin{equation}\label{shift}
|\omega-v_j| \leq |\omega-\omega_0|+|\omega_0-v_j| \MP h + (q_j Q)^{-1} \MP (Q\Psi(Q))^{-1} + (q_j Q)^{-1} \MP (q_j Q)^{-1}.
\end{equation}

\medskip

\textit{2. Successive rational averagings}

\medskip

Let us set $A_1:=A$, $B_1:=B$ so that
$P_1(\theta,I):=A_1(\theta)+B_1(\theta) \cdot I$ satisfies
$P_1=P$. Recalling that $[\,.\,]_v$ denotes the averaging along the
periodic flow associated to a periodic vector $v \in \R^n$
(see~\eqref{moyv}), we define inductively, for $1 \leq j \leq n$,
\[ A_{j+1}:=[A_j]_{v_j}, \quad B_{j+1}:=[B_j]_{v_j}, \quad
P_{j+1}:=[P_j]_{v_j} \]
so in particular $P_j(\theta,I)=A_j(\theta)+B_j(\theta)\cdot I$ for
$1 \leq j \leq n$. Let us also define $X_j$, for $1 \leq j \leq n$, by
\[ X_j(\theta,I):=C_j(\theta)+D_j(\theta)\cdot I \]
where
\[ C_j(\theta)=q_j \int_0^1 (A_j-A_{j+1})(\theta+tq_jv_j)tdt, \quad
D_j(\theta)=q_j \int_0^1 (B_j-B_{j+1})(\theta+tq_jv_j)tdt. \] 
If we further define $N_j$ by $N_j(I)=e(\omega)+v_j \cdot I$, it is
then easy to check, by a simple integration by parts, that the
equations 
\begin{equation}\label{eqhomo}
\{C_j,N_j\}=A_j-A_{j+1}, \quad \{D_j,N_j\}=B_j-B_{j+1},  \quad 1 \leq
j \leq n, 
\end{equation}
are satisfied and then
\begin{equation}\label{eqhomo2}
\{X_j,N_j\}=P_j-P_{j+1}, \quad 1 \leq j \leq n,
\end{equation}
are also satisfied, where $\{\,.\,,\,.\,\}$ denotes the usual Poisson
bracket. Moreover, we have the estimates
\begin{equation}\label{Estim0}
|A_j|_{s} \leq |A|_{s} \leq \varepsilon, \quad
|B_j|_{s} \leq |B|_{s} \leq \mu,  
\end{equation}
and then
\begin{equation}\label{Estim00}
|C_j|_{s} \leq q_j|A_j|_{s} \leq q_j\varepsilon \MP \Psi(Q)\varepsilon, \quad |D_j|_{s} \leq q_j|B_j|_{s} \leq q_j \mu \MP \Psi(Q)\mu. 
\end{equation}
Next, for any $0 \leq j \leq n$, define 
\[ r_j:=r-n^{-1}j\delta, \quad \hat{s}_j:=s(1-\sigma)^{2j}, \quad s_j=s(1-\sigma)^{2j+1}. \]  
Let $X_j^t$ be the
time-$t$ map of the Hamiltonian flow of $X_j$. Using~\eqref{Estim00},
together with inequalities~\eqref{eq:cond2}, \eqref{eq:cond3}, \eqref{eq:cond4} and the fact that $C(\sigma)^{-1} \MP \sigma$ (this is~\eqref{Csigma}), the condition~\eqref{smallX} and~\eqref{smallXX} of
Proposition~\ref{flotX}, \S~\ref{sec:udiff}, are satisfied, so the
latter proposition can be applied: for $1 \leq j \leq n$, $X_j^t$ is $(M,\hat{s}_j)$-ultra-differentiable, it maps
$\T^n \times D_{r_{j}}$ into $\T^n \times D_{r_{j-1}}$ for all
$t \in [0,1]$ and it is of the form
\[ X_j^t(\theta,I)=(\theta+E_j^t(\theta),I+F_j^t(\theta)\cdot
I+G_j^t(\theta)) \] 
with
\begin{equation}\label{estflot1}
\begin{cases}
|E_j^t|_{\hat{s}_j} \leq |D_j|_{s_{j-1}} \MP \Psi(Q)\mu, \quad |\nabla E_j^t|_{s_j} \MP s^{-1}C(\sigma)\Psi(Q)\mu   \\  
|F_j^t|_{\hat{s}_j} \MP s^{-1}C(\sigma)|D_j|_{s_{j-1}} \MP
s^{-1}C(\sigma)\Psi(Q)\mu, \quad |\nabla F_j^t|_{s_j} \MP s^{-2}C(\sigma)^2\Psi(Q)\mu \\ 
|G_j^t|_{\hat{s}_j} \MP s^{-1}C(\sigma)|C_j|_{s_{j-1}} \MP
s^{-1}C(\sigma)\Psi(Q)\varepsilon, \quad |\nabla G_j^t|_{s_j} \MP s^{-2}C(\sigma)^2\Psi(Q)\varepsilon.
\end{cases}
\end{equation} 
Now we define $\Phi^0:=\mathrm{Id}$ to be the identity and inductively
$\Phi^{j}:=\Phi^{j-1} \circ X_{j}^1$ for $1 \leq j \leq n$. Then
$\Phi^{j}$ maps $\T^n \times D_{r_{j}}$ into $\T^n \times D_{r}$ and
one easily checks, by induction using the estimates~\eqref{estflot1},
that $\Phi^{j}$ is still of the form
\[ \Phi^j(\theta,I)=(\theta+E^j(\theta),I+F^j(\theta)\cdot I+G^j(\theta)) \]
with the estimates, for $j=1,...,n$,
\begin{equation}\label{estflot2}
\begin{cases}
  |E^j|_{\hat{s}_j}  \MP \Psi(Q)\mu, \quad |\nabla E^j|_{s_j}  \MP s^{-1}C(\sigma)\Psi(Q)\mu,\\
  |F^j|_{\hat{s}_j} \MP
  s^{-1}C(\sigma)\Psi(Q)\mu, \quad  |\nabla F^j|_{s_j} \MP
  s^{-2}C(\sigma)^2\Psi(Q)\mu,\\
  |G^j|_{\hat{s}_j} \MP
  s^{-1}C(\sigma)\Psi(Q)\varepsilon, \quad |\nabla G^j|_{s_j} \MP
  s^{-2}C(\sigma)^2\Psi(Q)\varepsilon. 
  \end{cases}
\end{equation} 

\medskip

\textit{3. New Hamiltonian}

\medskip

Let us come back to the Hamiltonian $\hat{H}=H-R=N+P=N+P_1$. We claim
that for all $0 \leq j \leq n$, we have
\[ \hat{H} \circ \Phi^j= N+P_{j+1}+P_{j+1}^+, \quad P_{j+1}^+(\theta,I)=A_{j+1}^+(\theta)+B_{j+1}^+(\theta)\cdot I \] 
with the estimates
\begin{equation}\label{estAB}
|A_{j+1}^+|_{\hat{s}_j} \MP (Qs)^{-1}C(\sigma)\varepsilon, \quad
|B_{j+1}^+|_{\hat{s}_j} \MP (Qs)^{-1}C(\sigma)\mu. 
\end{equation}
Let us prove the claim by induction on $0 \leq j \leq n$. For $j=0$,
we may set $P_1^+:=0$ and there is nothing to prove. So let us assume
that the claim is true for some $j-1 \geq 0$, and we need to show it
is still true for $j \geq 1$. By this inductive assumption, we have
\[ \hat{H} \circ \Phi^j=\hat{H} \circ \Phi^{j-1} \circ X_j^1=(N+P_j+P_j^+) \circ X_j^1 \]
with 
\begin{equation}\label{estABind}
  |A_{j}^+|_{\hat{s}_{j-1}} \MP (Qs)^{-1}C(\sigma)\varepsilon,
  \quad |B_{j}^+|_{\hat{s}_{j-1}} \MP (Qs)^{-1}C(\sigma)\mu. 
\end{equation}
Let $S_j=\omega\cdot I - v_j\cdot I$ so that $N=N_j+S_j$ and thus
\[ \hat{H} \circ \Phi^j=(N_j+S_j+P_j+P_j^+) \circ X_j^1=(N_j+S_j+P_j)
\circ X_j^1+P_j^+ \circ X_j^1. \]
Let us consider the first summand of the last sum. Using the
equality~\eqref{eqhomo2}, a standard computation based on Taylor's
formula with integral remainder gives
\[ (N_j+S_j+P_j) \circ X_j^1 = N+[P_j]_{v_j} +\tilde{P}_{j+1}= N+P_{j+1}+\tilde{P}_{j+1} \]
with
\[ \tilde{P}_{j+1}=\int_0^1 W_{j+1}^t \circ X_{j}^t dt, \quad W_{j+1}^t:=\{ (1-t)P_{j+1}+t P_j +S_j, X_j \}.  \] 
Clearly, $W_{j+1}^t$ is still of the form
\[ W_{j+1}^{t}(\theta,I)=W_{j+1}^{t}(\theta,0)+\nabla_I
W_{j+1}^t(\theta,0)\cdot I \]
as this is true for $P_j$, $S_j$, $X_j$ and that this form is
preserved under Poisson bracket. Using the estimates for
$P_j(\theta,0)$, $\nabla_I P_j(\theta,0)$, $X_j(\theta,0)$,
$\nabla_I X_j(\theta,0)$ (given respectively in~\eqref{Estim0} and
in~\eqref{Estim00}), the fact that
\[ S_j(\theta,0)=0, \quad \nabla_IS_j(\theta,0)=\omega-v_j \]
with the inequality~\eqref{shift}, and the estimates for the
derivatives and the product of Gevrey functions (given respectively in
Proposition~\ref{derivative}, Corollary~\ref{corderivative} and
Proposition~\ref{produit}, Corollary~\ref{corproduit},
\S~\ref{sec:udiff}), one finds, for all $t \in [0,1]$
\begin{eqnarray*}
|W_{j+1}^t(\theta,0)|_{\hat{s}_j} & \MP &
(s^{-1}C(\sigma) q_j\varepsilon \mu+s^{-1}C(\sigma)q_j\varepsilon
\mu+s^{-1}C(\sigma)q_j\varepsilon(q_jQ)^{-1}) \\
& \MP & s^{-1}C(\sigma)q_j\varepsilon \mu
+(Qs)^{-1}C(\sigma)\varepsilon.
\end{eqnarray*}
Since $q_j \MP \Psi(Q)$, using~\eqref{eq:cond5} the latter estimate reduces to
\[ |W_{j+1}^t(\theta,0)|_{\hat{s}_j} \MP (Qs)^{-1}C(\sigma)\varepsilon.  \]
Similarly, one obtains
\[ |\nabla_I W_{j+1}^t(\theta,0)|_{\hat{s}_j} \MP
(Qs)^{-1}C(\sigma)\mu.  \]
Then, using the expression of $X_j^t$ and the associated
estimates~\eqref{estflot1}, a direct computation, still
using~\eqref{eq:cond5}, gives
\[ |\tilde{P}_{j+1}(\theta,0)|_{\hat{s}_j} \MP (Qs)^{-1}C(\sigma)\varepsilon \]
and
\[ |\nabla_I\tilde{P}_{j+1}(\theta,0)|_{\hat{s}_j} \MP
(Qs)^{-1}C(\sigma)\mu.  \]
Using again the estimates of $X_j^t$ given by~\eqref{estflot1},
and the inductive assumption~\eqref{estABind}, we also find
\[ |P_j^+ \circ X_j^1(\theta,0)|_{\hat{s}_j} \MP (Qs)^{-1}C(\sigma)\varepsilon \]
and
\[ |\nabla_I(P_j^+ \circ X_j^1)(\theta,0)|_{\hat{s}_j} \MP (Qs)^{-1}C(\sigma)\mu.  \]
Eventually, we may define
\[ P_{j+1}^+:=\tilde{P}_{j+1}+P_j^+ \circ X_j^1 \]
so that
\[ \hat{H} \circ \Phi^j= N+P_{j+1}+P_{j+1}^+, \quad P_{j+1}^+(\theta,I)=A_{j+1}^+(\theta)+B_{j+1}^+(\theta)\cdot I \] 
and these last estimates imply that
\begin{equation*}
|A_{j+1}^+|_{\hat{s}_j} \MP (Qs)^{-1}C(\sigma)\varepsilon, \quad |B_{j+1}^+|_{\hat{s}_j} \MP (Qs)^{-1}C(\sigma)\mu.
\end{equation*}
The claim is proved. So we may set 
\[ \Phi:=\Phi^n, \quad (E,F,G):=(E^n,F^n,G^n), \]
with, as~\eqref{estflot2} tells us with $j=n$,
\begin{equation}\label{estflot3}
\begin{cases}
  |E|_{\hat{s}_n}  \MP \Psi(Q)\mu, \quad |\nabla E|_{s_n}  \MP s^{-1}C(\sigma)\Psi(Q)\mu \\
  |F|_{\hat{s}_n} \MP s^{-1}C(\sigma)\Psi(Q)\mu, \quad |\nabla F|_{s_n} \MP s^{-2}C(\sigma)^2\Psi(Q)\mu \\
  |G|_{\hat{s}_n} \MP s^{-1}C(\sigma)\Psi(Q)\varepsilon, \quad |\nabla G|_{s_n} \MP s^{-2}C(\sigma)^2\Psi(Q)\varepsilon.  
  \end{cases}
\end{equation} 
Observe that $P_{n+1}=[\cdots[P]_{v_1}\cdots]_{v_n}=[P]_{v_1,\dots,v_n}$, and thus by Proposition~\ref{cordio}, $P_{n+1}=[P]$, and as a consequence
\[ \hat{H} \circ \Phi(\theta,I)=e+\omega\cdot I+[A]+[B]\cdot I+A^+_{n+1}(\theta)+B^+_{n+1}(\theta)\cdot I \]
with the estimates
\begin{equation}\label{estAB2}
|A_{n+1}^+|_{\hat{s}_n} \MP (Qs)^{-1}C(\sigma)\varepsilon, \quad |B_{n+1}^+|_{\hat{s}_n} \MP (Qs)^{-1}C(\sigma)\mu.
\end{equation}

\medskip

\textit{4. Estimate of the remainder}

\medskip

Now we take into account the remainder term $R$ that we previously ignored: we have $H=\hat{H}+R$, and therefore
\[ H \circ \Phi(\theta,I)=e+\omega\cdot I+[A]+[B]\cdot
I+A^+_{n+1}(\theta)+B^+_{n+1}(\theta)\cdot I+R \circ
\Phi(\theta,I).  \] Let us decompose
\[ R\circ \Phi(\theta,I)=
\underbrace{R\circ\Phi(\theta,0)}_{R_0(\theta)} +
\underbrace{\nabla_I(R \circ \Phi)(\theta,0)}_{R_1(\theta)} \cdot I
+\tilde{R}(\theta,I)\]
and let us define
\[ \tilde{A}:=A^+_{n+1}+R_0, \quad
\tilde{B}:=B^+_{n+1}+R_1. \]
We have $R(\theta,I)=M(\theta,I) \cdot I^2$ and as $H$ and $R$ differ only by terms of order at most one in $I$, $\nabla_I^2 H=\nabla_I^2 R$ so
\[ M(\theta,I)=\int_0^1(1-t)\nabla_I^2 H(\theta,tI)dt=\int_0^1(1-t)\nabla_I^2 R(\theta,tI)dt \]
and therefore $|M|_{s}\leq \eta$. Then, as $\Phi(\theta,0)=(\theta+E(\theta),G(\theta))$, we have the expression
\[ R_0(\theta)=R(\Phi(\theta,0))=M(\theta+E(\theta),G(\theta))\cdot
G(\theta)^2 \]
and so using the above estimate on $M$, together with the estimates on
$E$, $G$ and the estimates for the product and compostion of Gevrey
functions (given respectively in Proposition~\ref{produit} and
Proposition~\ref{composition}, \S~\ref{sec:udiff}), we find
\[ |R_0|_{\hat{s}_n} \MP \eta |G|^2_{\hat{s}_n} \MP
\eta(s^{-1}C(\sigma)\Psi(Q)\mu)^2 \varepsilon \MP \eta
(Qs)^{-2}C(\sigma)^2\varepsilon \MP \eta
(Qs)^{-1}C(\sigma)\varepsilon. \] 
Then, we have $\nabla_I R(\theta,I)=\hat{M}(\theta,I)\cdot I^2$ with 
\[ \hat{M}(\theta,I)=\int_0^1 \nabla_I^2 H(\theta,tI)dt=\int_0^1\nabla_I^2 R(\theta,tI)dt \]
and hence $|\hat{M}|_{s}\leq \eta$ also. Since 
\begin{equation}\label{F}
|\nabla_I \Phi-\mathrm{Id}|_{\hat{s}_n}=
|F|_{\hat{s}_n} \MP s^{-1}C(\sigma)\Psi(Q)\mu\PM 1  
\end{equation}
we obtain, using the fact that
$\varepsilon \leq \mu^2$ and proceeding as before,
\[ |R_1|_{\hat{s}_n} \MP \eta|G|_{\hat{s}_n} \MP
\eta s^{-1}C(\sigma)\Psi(Q)\varepsilon \MP\eta (s^{-1}C(\sigma)\Psi(Q)\mu)\mu
\MP\eta (Qs)^{-1}C(\sigma)\mu. \] 
These last estimates on $R_0$ and $R_1$, together with~\eqref{estAB2}, imply
\begin{equation*}
|\tilde{A}|_{\hat{s}_n} \MP(1+\eta)
(Qs)^{-1}C(\sigma)\varepsilon, \quad
|\tilde{B}|_{\hat{s}_n} \MP(1+\eta) (Qs)^{-1}C(\sigma)\mu. 
\end{equation*}
We can finally now use~\eqref{eq:cond6} to ensure that
\begin{equation}\label{estABF}
|\tilde{A}|_{\hat{s}_n} \leq \varepsilon/16, \quad
|\tilde{B}|_{\hat{s}_n} \leq \mu/4. 
\end{equation}
It is important to recall here that we may choose the implicit constant in~\eqref{eq:cond6} as large as we want (in order to achieve~\eqref{estABF}) without affecting any of the other implicit constants. Then observe also that $H \circ \Phi$ and $\tilde{R}$ differ only by terms of order at most one in $I$, so
\[   \nabla_I^2 (H\circ \Phi)=\nabla_I^2 \tilde{R}, \quad \nabla_I^2 H=\nabla_I^2 R 
\]
and therefore using the formula for the Hessian of a composition,~\eqref{F} and the fact that $\nabla_I^2 \Phi$ is identically zero, one finds
\begin{equation}\label{estM}
|\nabla_I^2 \tilde{R}-\nabla_I^2 R \circ \Phi|_{\hat{s}_n} \MP \eta |F|_{s}.
\end{equation}
We also set $\tilde{e}:=e+[A]$ and observe that
\begin{equation}\label{este}
|\tilde{e}-e|_{\hat{s}_n} \leq |[A]|_{s} \leq |A|_{s}.
\end{equation}

\medskip

\textit{5. Change of frequencies and final estimates}

\medskip

Let us now write explicitly the dependence on the parameter $\omega \in D_h$: we have 
\[ H \circ \Phi(\theta,I,\omega)=\tilde{e}(\omega)+(\omega+[B](\omega))\cdot I+\tilde{A}(\theta,\omega)+\tilde{B}(\theta,\omega)\cdot I+ \tilde{R}(\theta,I,\omega).\]
Consider the map $\phi(\omega):=\omega+[B(\omega)]$, it satisfies
\[ |\phi-\mathrm{Id}|_{s} \leq |[B]|_{s} \leq
|B|_{s} \leq \mu \]
and therefore the conditions~\eqref{smallomega} of
Proposition~\ref{propomega} are satisfied: the first condition
of~\eqref{smallomega} follows, from instance, from
condition~\eqref{eq:cond2} and the fact that $\Psi(Q)\geq Q \geq 1$,
whereas the second condition of~\eqref{smallomega} is implied by
condition~\eqref{eq:cond7}. Hence Proposition~\ref{propomega} applies
and one finds a unique
$\varphi \in G_{\hat{s}_n}(D_{h/2},D_h)$ such that
$\phi \circ \varphi=\mathrm{Id}$ and
\begin{equation}\label{estom}
|\varphi-\mathrm{Id}|_{\hat{s}_n} \leq |\phi-\mathrm{Id}|_{s} \leq \mu.
\end{equation}  
We do have $\varphi(\omega)+[B(\varphi(\omega))]=\omega$ and thus,
setting $\mathcal{F}:=(\Phi,\varphi)$, this implies that
\[ H \circ \mathcal{F}(\theta,I,\omega)= H \circ
\Phi(\theta,I,\varphi(\omega)) =
\tilde{e}(\varphi(\omega))+\omega\cdot
I+\tilde{A}(\theta,\varphi(\omega)) +
\tilde{B}(\theta,\varphi(\omega))\cdot I+
\tilde{R}(\theta,I,\varphi(\omega))\] and at the end we set
\[ e^+:=\tilde{e}\circ \varphi, \quad A^+:=\tilde{A} \circ \varphi,
\quad B^+:=\tilde{B} \circ \varphi, \quad R^+:=\tilde{R} \circ
\varphi. \] 
Using once again Proposition~\ref{composition}, the
inequalities~\eqref{estABF},~\eqref{estM} and~\eqref{este} imply
\begin{equation*}
\begin{cases}
|A^+|_{s_n}=|\tilde{A} \circ \varphi|_{s_n}
\leq |\tilde{A}|_{\hat{s}_n} \leq \varepsilon/16, \\  
|B^+|_{s_n}=|\tilde{B} \circ \varphi|_{s_n}
\leq |\tilde{B}|_{\hat{s}_n} \leq \mu/4, \\  
|e^+-e\circ \varphi|_{s_n}=|(\tilde{e}-e)\circ
\varphi|_{s_n} \leq |\tilde{e}-e|_{\hat{s}_n} \leq
|A|_{s},  
\\ 
|\nabla_I^2 R^+-\nabla_I^2 R \circ \mathcal{F}|_{s_n}=|(\nabla_I^2 \tilde{R}-\nabla_I^2 R \circ \Phi)\circ \varphi|_{s_n}
\leq |\nabla_I^2 \tilde{R}-\nabla_I^2 R \circ \Phi|_{\hat{s}_n} \MP \eta |F|_{\hat{s}_n},
\end{cases}
\end{equation*}
which were the estimates~\eqref{stepest2} we needed to prove. The
transformation
$\mathcal{F}=(\Phi,\varphi)=(E,F,G,\varphi) \in \mathcal{G}$ maps
$\T^n \times D_{r-\delta,h/2}$ into $\T^n \times D_{r,h}$ and it
follows from~\eqref{estflot3} and~\eqref{estom} that
\begin{equation*}
\begin{cases}
|E|_{\hat{s}_n} \MP \Psi(Q)\mu, \quad |\nabla E|_{s_n} \MP s^{-1}C(\sigma)\Psi(Q)\mu, \\
|F|_{\hat{s}_n} \MP s^{-1}C(\sigma)\Psi(Q)\mu, \quad |\nabla F|_{s_n} \MP s^{-2}C(\sigma)^2\Psi(Q)\mu, \\
|G|_{\hat{s}_n} \MP \sigma^{-\alpha}\Psi(Q)\varepsilon, \quad |\nabla G|_{s_n} \MP s^{-2}C(\sigma)^2\Psi(Q)\varepsilon,   \\
|\varphi-\mathrm{Id}|_{\hat{s}_n} \leq \mu, \quad |\nabla \varphi-\mathrm{Id}|_{s_n} \MP s^{-1}C(\sigma) \mu
\end{cases}
\end{equation*}
which were the wanted estimates~\eqref{stepest1}. This concludes the proof. 
\end{proof}

\subsection{Iterations and convergence}\label{sec:iteration}

We now define, for $i \in \N$, the following decreasing geometric sequences:
\begin{equation}\label{defsuite1}
\varepsilon_i:=16^{-i} \varepsilon, \quad \mu_i:=4^{-i}\mu, \quad \delta_i:=2^{-i-2}r, \quad h_i=2^{-i} h.
\end{equation}
Next, for a constant $Q_0$ to be chosen below, we define $\Delta_i$ and $Q_i$, $i\in \N$, by 
\begin{equation}\label{defsuite2}
\Delta_i=2^i \Delta(Q_0), \quad Q_i=\Delta^{-1}(\Delta_i)=\Delta^{-1}(2^i \Delta(Q_0)).
\end{equation}
Let us now choose $0 < c_2\leq 1$, which depends only on $n$, for which the last inequality of~\eqref{cond1} is satisfied with implicit constant $2c_2$, that is
\begin{equation}\label{lastcond1}
(1+\eta)(2c_2)^{-1} \leq QsC(\sigma)^{-1}.
\end{equation}
We can now define $\sigma_i$, $i\in \N$, by
\begin{equation}\label{defsuite3}
\sigma_i= C^{-1}(c_2(1+\eta)^{-1}sQ_i)
\end{equation}
and finally, we define $\hat{s}_i$, $s_i$ and $r_i$, $i\in\N$, by 
\begin{equation}\label{defsuite4}
\hat{s}_0=s_0=s, \quad \hat{s}_{i+1}=\hat{s}_i(1-\sigma_i)^{2n}, \quad s_{i+1}=s_i(1-\sigma_i)^{2n+1}, \quad r_0=r, \quad r_{i+1}=r_i-\delta_i.  
\end{equation}
Obviously, we have
\[ \lim_{i \rightarrow +\infty} r_i=r-\sum_{i \in \N}\delta_i=r/2. \]
We claim that, assuming that $\omega_0$ satisfies~\eqref{BRM}, if we choose $Q_0$ as in~\eqref{eqQ0} we have
\begin{equation}\label{sigmai}
\lim_{i \rightarrow +\infty} s_i\geq s/2 \Longleftrightarrow \prod_{i \in \N} (1-\sigma_i)^{2n+1} \geq 1/2.
\end{equation}
Indeed, we have
\begin{eqnarray*}
\sum_{i \in \N}\sigma_i & = & \sigma_0 + \sum_{i \geq 1}\sigma_i \\
& = & C^{-1}(c_2(1+\eta)^{-1}sQ_0)+ \sum_{i \geq 1}C^{-1}(c_2(1+\eta)^{-1}s\Delta^{-1}(2^i\Delta(Q_0))) \\
& \leq & C^{-1}(c_2(1+\eta)^{-1}sQ_0)+ \int_{0}^{+\infty}C^{-1}(c_2(1+\eta)^{-1}s\Delta^{-1}(2^y\Delta(Q_0)))dy \\  
& = & C^{-1}(c_2(1+\eta)^{-1}sQ_0)+ (\ln 2)^{-1}\int_{\Delta(Q_0)}^{+\infty}C^{-1}(c_2(1+\eta)^{-1}s\Delta^{-1}(x))\frac{dx}{x} \\  
& \leq & \ln 2(4n+2)^{-1}
\end{eqnarray*}
where we used~\eqref{eqQ0} in the last line. In particular, for any $i \in \N$ we have $\sigma_i \leq \ln 2(4n+2)^{-1}$ which implies
\[ \ln (1-\sigma_i)^{2n+1}=(2n+1)\ln(1-\sigma_i) \geq -(4n+2)\sigma_i \]
and thus
\[ \sum_{i \in \N}\ln (1-\sigma_i)^{2n+1} \geq -(4n+2)\sum_{i \in \N} \sigma_i \geq -\ln 2 \]
and finally
\[ \prod_{i \in \N}(1-\sigma_i)^{2n+1}=\exp \left(\sum_{i \in \N}\ln (1-\sigma_i)^{2n+1}\right) \geq 1/2. \]
This shows that~\eqref{sigmai} holds true with our choice of $\sigma_i$. Applying inductively Proposition~\ref{kamstep} we will easily obtain the following proposition.

\begin{proposition}\label{kamiter}
Let $H$ be as in~\eqref{HAM}, with $M=(M_l)_{l \in \N}$ satisfying \eqref{H1} and \eqref{H2} and $\omega_0$ satisfying~\eqref{BRM}. Fix $Q_0 \geq n+2$ so that~\eqref{eqQ0} is satisfied, and assume that
\begin{equation}\label{cond2}
\sqrt{\varepsilon} \leq \mu \leq h/2, \quad \sqrt{\varepsilon} \leq r, \quad h \PM \Delta(Q_0)^{-1}.
\end{equation}
Then, for each $i\in \N$, there exists an $(M,s_i)$-ultra-differentiable transformation 
\[ \mathcal{F}^i=(\Phi^i,\varphi^i)=(E^i,F^i,G^i,\varphi^i) : \T^n \times D_{r_i,h_i} \rightarrow \T^n \times D_{r,h} \in \mathcal{G}, \]
such that $\mathcal{F}^{i+1}=\mathcal{F}^{i} \circ \mathcal{F}_{i+1}$, with
\[ \mathcal{F}_{i+1}=(\Phi_{i+1},\varphi_{i+1})=(E_{i+1},F_{i+1},G_{i+1},\varphi_{i+1}) : \T^n \times D_{r_{i+1},h_{i+1}} \rightarrow \T^n \times D_{r_{i},h_{i}} \in \mathcal{G}, \]
satisfying the following estimates
\begin{equation}\label{stepest1i}
\begin{cases}
|E_{i+1}|_{\hat{s}_{i+1}} \MP \Psi(Q_{i})\mu_{i}, \quad |\nabla E_{i+1}|_{s_{i+1}} \MP s_i^{-1}C(\sigma_i)\Psi(Q_i)\mu_i, \\ 
|F_{i+1}|_{\hat{s}_{i+1}} \MP s_i^{-1}C(\sigma_i)\Psi(Q_i)\mu_i, \quad |\nabla F_{i+1}|_{s_{i+1}} \MP s_i^{-2}C(\sigma_i)^2\Psi(Q_i)\mu_i, \\ 
|G_{i+1}|_{\hat{s}_{i+1}}\MP s_i^{-1}C(\sigma_i)\Psi(Q_i)\varepsilon_i, \quad |\nabla G_{i+1}|_{s_{i+1}} \MP s_i^{-2}C(\sigma_i)^2\Psi(Q_i)\varepsilon_i, \\ 
|\varphi_{i+1}-\mathrm{Id}|_{\hat{s}_{i+1}} \leq \mu_i, \quad  |\nabla \varphi_{i+1}-\mathrm{Id}|_{s_{i+1}} \MP s_i^{-1}C(\sigma_i) \mu_i
\end{cases}
\end{equation}
and such that 
\[ 
  H \circ \mathcal{F}^i(\theta,I,\omega) =
  \underbrace{e^i(\omega)+\omega\cdot I}_{N^i(I,\omega)} +
  \underbrace{A^i(\theta)+B^i(\theta)\cdot I}_{P^i(\theta,I,\omega)} +
  \underbrace{M^i(\theta,I,\omega)\cdot I^2}_{R^i(\theta,I,\omega)} 
\]
with the estimates
\begin{equation}\label{stepest2i}
\begin{cases}
  |A^i|_{s_{i+1}} \leq \varepsilon_i, \quad |B^i|_{s_{i+1}} \leq \mu_i, \\
  |e^{i+1}-e^{i} \circ \varphi_{i+1}|_{s_{i+1}} \leq |A^{i}|_{\alpha,s_{i}}, \\
  |\nabla_I^2 R^{i+1}-\nabla_I^2 R^{i} \circ \mathcal{F}_{i+1}|_{s_{i+1}}\MP
  \eta|F_{i+1}|_{\hat{s}_{i+1}}. 
\end{cases}
\end{equation}
\end{proposition}

Let us emphasize that the implicit constants in the above proposition depend only on $n$ and are thus independent of $i \in \N$. 

\begin{proof}
  For $i=0$, we let $\mathcal{F}^0$ be the identity, $e^0:=e$,
  $A^0:=A$, $B^0:=B$, $R^0:=R$, $M^0:=M$ and there is nothing to
  prove. The general case follows by an easy induction. Indeed, assume
  that the statement holds true for some $i\in \N$ so that
  $H \circ \mathcal{F}^i$ is $(M,s_i)$-ultra-differentiable on the domain
  $\T^n \times D_{r_i,s_i}$. We want to apply
  Proposition~\ref{kamstep} to this Hamiltonian, with
  $\varepsilon=\varepsilon_i$, $\mu=\mu_i$, $r=r_i$, $s=s_i$, $h=h_i$,
  $\sigma=\sigma_{i}$ and $Q=Q_i$. First, by our choice of $Q_0$ and
  $\delta_0$ it is clear that $0<\sigma_{i}< 1$,
  $0 < \delta_i < r_i$, and $Q_i \geq n+2$. Then, recalling the definition of the constant $c_2$, we need to check that
  the conditions
  \begin{equation*}
  \begin{cases}
    \sqrt{\varepsilon_i} \leq \mu_i \leq h_i/2,
    \quad \sqrt{\varepsilon_i} \leq r_i, \quad h_i \PM \Delta(Q_i)^{-1}, \\ 
    r_i\mu_i \PM \delta_i
    \Delta(Q_i)^{-1},  \quad  (1+\eta)(2c_2)^{-1} \leq Q_is_iC(\sigma_i)^{-1} 
    \end{cases}
  \end{equation*}
  are satisfied. From the definition of $Q_i$ and $\sigma_i$ we have
  \begin{equation*}
    \Delta(Q_i) =\Delta(\Delta^{-1}(\Delta_i))=\Delta_i, \quad C(\sigma_i)=C(C^{-1}(c_2(1+\eta)^{-1}sQ_i))=c_2(1+\eta)^{-1}sQ_i 
  \end{equation*}
and as $2^{-i}r \leq r_i \leq r$, it is sufficient to check the conditions
  \begin{equation}\label{condit}
    \sqrt{\varepsilon_i} \leq \mu_i \leq h_i/2,
    \quad \sqrt{\varepsilon_i} \leq 2^{-i}r, \quad h_i \PM \Delta_i^{-1}, \quad r\mu_i \PM \delta_i
    \Delta_i^{-1},  \quad s/2 \leq s_i. 
  \end{equation}
  The last condition of~\eqref{condit} is satisfied, for all
  $i \in \N$, simply by the choice of $Q_0$ satisfying~\eqref{eqQ0}, as this implies~\eqref{sigmai}. As for the other four conditions of~\eqref{condit},
  using the fact that the sequences $\varepsilon_i$, $\mu_i$, $h_i$,
  $\Delta_i^{-1}$ and $\delta_i$ decrease at a geometric rate with
  respective ratio $1/16$, $1/4$, $1/2$, $1/2$ and $1/2$, it is clear
  that they are satisfied for any $i \in \N$ if and only if they are
  satisfied for $i=0$. The first three conditions of~\eqref{condit} for
  $i=0$ are nothing but~\eqref{cond2}. Moreover, using our choice of
  $\delta_0=r/4$, the fourth condition of~\eqref{condit} for $i=0$ reads
  $\mu \PM \Delta_0^{-1}$ and this also follows from~\eqref{cond2}.

Hence Proposition~\ref{kamstep} can be applied, and all the conclusions of the statement follow at once from the conclusions of Proposition~\ref{kamstep}.
\end{proof}

We can finally conclude the proof of Theorem~\ref{KAMparameter}, by showing that one can pass to the limit $i \rightarrow +\infty$ in Proposition~\ref{kamiter}.

\begin{proof}[Proof of Theorem~\ref{KAMparameter}]
Recall that we are given $\varepsilon>0$, $\mu>0$, $r>0$, $s>0$, $h>0$ and that we define the sequences $\varepsilon_i,\mu_i,\delta_i,h_i$ in~\eqref{defsuite1}, and then we chose $Q_0 \geq n+2$ satisfying~\eqref{eqQ0} to define the sequences $\Delta_i,Q_i$ in~\eqref{defsuite2} and $\sigma_i$ in~\eqref{defsuite3} and finally, $\hat{s}_i$, $s_i$ and $r_i$ were defined in~\eqref{defsuite4}. Moreover, we have 
\begin{equation}\label{limits}
\begin{cases}
\lim_{i \rightarrow +\infty}\varepsilon_i=\lim_{i \rightarrow +\infty}\mu_i=\lim_{i \rightarrow +\infty}h_i=0, \\ 
\lim_{i \rightarrow +\infty}r_i=r-\sum_{i \in \N}\delta_i=r/2, \quad \lim_{i \rightarrow +\infty}s_i=s\prod_{i \in \N}(1-\sigma_i)^{2n+1}\geq s/2 
\end{cases}
\end{equation}
and for later use, let us observe that the following series are convergent and can be made as small as one wishes thanks to condition~\eqref{cond0} of Theorem~\ref{KAMparameter}:
\begin{equation}\label{serie1}
\sum_{i=0}^{+\infty} s^{-1}C(\sigma_i)\mu_i \leq \sum_{i=0}^{+\infty} Q_i \mu_i = \sum_{i=0}^{+\infty} (\Psi(Q_i))^{-1}\Delta_i\mu_i \leq 2 (\Psi(Q_0))^{-1}\Delta_0 \mu =2Q_0\mu
\end{equation}
\begin{equation}\label{serie2}
\sum_{i=0}^{+\infty} \mu_i \leq 2\mu 
\end{equation}
\begin{equation}\label{serie3}
\sum_{i=0}^{+\infty} s^{-1}C(\sigma_i)\Psi(Q_i)\mu_i \leq \sum_{i=0}^{+\infty} \Delta_i \mu_i \leq 2\Delta_0 \mu=2 Q_0\Psi(Q_0) \mu 
\end{equation}
\begin{equation}\label{serie4}
\sum_{i=0}^{+\infty} \Psi(Q_i)\mu_i \leq \sum_{i=0}^{+\infty} Q_i^{-1}\Delta_i \mu_i \leq 2 Q_0^{-1}\Delta_0 \mu =2\Psi(Q_0)\mu
\end{equation}
\begin{equation}\label{serie5}
\sum_{i=0}^{+\infty} s^{-1}C(\sigma_i)\Psi(Q_i)\varepsilon_i \leq \sum_{i=0}^{+\infty} \Delta_i \varepsilon_i \leq 2\Delta_0 \varepsilon=2 Q_0\Psi(Q_0) \varepsilon. 
\end{equation}
Now the condition~\eqref{cond0} of Theorem~\ref{KAMparameter} implies that the condition~\eqref{cond2} of Proposition~\ref{kamiter} is satisfied; what we need to prove is that the sequences given by this Proposition~\ref{kamiter} do convergence in the Banach space of $(M,s/2)$-ultra-differentiable functions. Recall that $\mathcal{F}^0=(E^0,F^0,G^0,\varphi^0)$ is the identity, while for $i \geq 0$, 
\[ (E^{i+1},F^{i+1},G^{i+1},\varphi^{i+1})=\mathcal{F}^{i+1}=\mathcal{F}^{i} \circ \mathcal{F}_{i+1}=(E^{i},F^{i},G^{i},\varphi^{i}) \circ (E_{i+1},F_{i+1},G_{i+1},\varphi_{i+1})\]
from which one easily obtains the following inductive expressions:
\begin{equation}
\begin{cases}
E^{i+1}(\theta,\omega)=E_{i+1}(\theta,\omega)+E^{i}(\theta+E_{i+1}(\theta,\omega),\varphi_{i+1}(\omega)) \\
F^{i+1}(\theta,\omega)=F_{i+1}(\theta,\omega)+F^{i}(\theta+E_{i+1}(\theta,\omega),\varphi_{i+1}(\omega))\cdot (\mathrm{Id}+F_{i+1}(\theta,\omega)) \\
G^{i+1}(\theta,\omega)=(F^{i}(\theta+E_{i+1}(\theta,\omega),\varphi_{i+1}(\omega))+\mathrm{Id})\cdot G_{i+1}(\theta,\omega)+G^{i}(\theta+E_{i+1}(\theta,\omega),\varphi_{i+1}(\omega)) \\
\varphi^{i+1}=\varphi^{i} \circ \varphi_{i+1}.
\end{cases}
\end{equation} 
Let us first prove that the sequence $\varphi^i$ converges. We claim that for all $i \in \N$, we have
\[ |\nabla \varphi^{i}|_{s_{i}} \MP \prod_{l=0}^{i}(1+s^{-1}C(\sigma_l)\mu_l) \MP 1 \]
where the fact that the last product is bounded uniformly in $i \in \N$ follows from~\eqref{serie1}. For $i=0$, $\varphi^0=\mathrm{Id}$ and there is nothing to prove; for $i \in \N$ since $\varphi^{i+1}=\varphi^i \circ \varphi_{i+1}$ we have
\[ \nabla \varphi^{i+1}=\left(\nabla \varphi^{i} \circ \varphi_{i+1}\right) \cdot \nabla \varphi_{i+1} \]
so that using the estimate for $\varphi_{i+1}$ and $\nabla \varphi_{i+1}$ given in~\eqref{stepest1i}, Proposition~\ref{kamiter}, the claim follows by induction. Using this claim, and writing
\[ \varphi^{i+1}-\varphi^{i}=\varphi^i \circ \varphi_{i+1}-\varphi^i=\left(\int_0^1\nabla \varphi^i \circ (t \varphi_{i+1}+(1-t)\mathrm{Id})dt\right)\cdot (\varphi_{i+1}-\mathrm{Id}) \]
one finds
\[ |\varphi^{i+1}-\varphi^{i}|_{s_{i+1}} \MP |\varphi_{i+1}-\mathrm{Id}|_{s_{i+1}}, \]
and therefore
\[ |\varphi^{i+1}-\varphi^{i}|_{s_{i+1}} \MP \mu_i. \]
Using the convergence of~\eqref{limits} and~\eqref{serie2}, one finds that the sequence $\varphi^i$ converges to a trivial map
\[ \varphi^* : \{\omega_0\} \rightarrow D_h, \quad \varphi^*(\omega_0):=\omega_0^* \]
such that
\[ |\omega_0^*-\omega_0| \MP \mu. \]
Now let us define
\[ V_{i+1}(\theta,\omega):=(\theta+E_{i+1}(\theta,\omega),\varphi_{i+1}(\omega)), \quad V_{i+1}=(\mathrm{Id}+E_{i+1},\varphi_{i+1}) \]
and observe that since $\Psi(Q_i)\geq 1$, then the estimates for $E_{i+1}$, $\nabla E_{i+1}$, $\varphi_{i+1}$ and $\nabla \varphi_{i+1}$ given in Proposition~\ref{kamiter} implies that
\[ |V_{i+1}-\mathrm{Id}|_{s_{i+1}} \MP \Psi(Q_{i})\mu_{i}, \quad |\nabla V_{i+1}-\mathrm{Id}|_{s_{i+1}} \MP s^{-1}C(\sigma_i)\Psi(Q_i)\mu_i.  \]
Using these estimates, and the fact that $E^{i+1}$ can be written as
\[ E^{i+1}=E_{i+1}+E^i \circ V_{i+1} \]
we can proceed as before, using the convergence of~\eqref{serie3} to show first that
\[ |\nabla E^i|_{s_i} \MP \sum_{l=0}^i s^{-1}C(\sigma_l)\Psi(Q_l)\mu_l \MP 1 \]
and then
\[ |E^{i+1}-E^i|_{s_{i+1}} \MP |E_{i+1}|_{s_{i+1}} \MP \Psi(Q_i)\mu_i. \] 
Using the convergence of~\eqref{limits} and~\eqref{serie4}, this shows that $E^i$ converges to a map
\[ E^* : \T^n \times \{\omega_0\} \rightarrow \T^n \times D_h \]
such that
\[ |E^*|_{s/2} \MP \Psi(Q_0)\mu. \]
For the $F^i$, we do have the expression
\[ F^{i+1}=F_{i+1}+(F^i \circ V_{i+1})\cdot (\mathrm{Id}+F_{i+1}) \]
or alternatively
\[ F^{i+1}=(\mathrm{Id}+F^i \circ V_{i+1})\cdot F_{i+1}+ F^i \circ V_{i+1} \]
and thus
\[ F^{i+1}-F^i=(\mathrm{Id}+F^i \circ V_{i+1})\cdot F_{i+1}+ F^i \circ V_{i+1}-F^i. \]
As before, using the estimates on $F_{i+1}$ and $\nabla F_{i+1}$ given in Proposition~\ref{kamiter}, one shows that
\[ |\nabla F^i|_{s_i} \MP \sum_{l=0}^i s^{-2}C(\sigma_l)^2\Psi(Q_l)\mu_l \]
but however, here, the sum above is not convergent. Yet we do have
\[ sC(\sigma_i)^{-1}|\nabla F^i|_{s_i} \MP sC(\sigma_i)^{-1}\sum_{l=0}^i s^{-2}C(\sigma_l)^2\Psi(Q_l)\mu_l \MP \sum_{l=0}^i s^{-1}C(\sigma_l)\Psi(Q_l)\mu_l \MP 1  \]
from~\eqref{serie4} and using the fact that the estimate for $V_{i+1}$ can be written as
\[ |V_{i+1}-\mathrm{Id}|_{s_{i+1}} \MP sC(\sigma_i)^{-1}s^{-1}C(\sigma_i)\Psi(Q_{i})\mu_{i} \]
one obtains
\[ |F^i \circ V_{i+1}-F^i|_{s_{i+1}}  \MP s^{-1}C(\sigma_i)\Psi(Q_{i})\mu_{i}.  \]
By induction, one shows that 
\[ |F^i|_{s_i} \MP \sum_{l=0}^i s^{-1}C(\sigma_l)\Psi(Q_l)\mu_l \MP 1 \]
from which one obtains
\[ |\mathrm{Id}+F^i \circ V_{i+1}|_{s_i} \MP 1 \]
and as a consequence, 
\[ |F^{i+1}-F^i|_{s_{i+1}}  \MP s^{-1}C(\sigma_i)\Psi(Q_{i})\mu_{i}. \]
Using the convergence of~\eqref{limits} and~\eqref{serie3}, this shows that $F^i$ converges to a map
\[ F^* : \T^n \times \{\omega_0\} \rightarrow \T^n \times D_h \]
such that
\[ |F^*|_{s/2} \MP Q_0\Psi(Q_0)\mu. \]
For $G^i$, we have the expression
\[ G^{i+1}=(F^i \circ V_{i+1}+\mathrm{Id})\cdot G_{i+1}+G^i \circ V_{i+1} \]
and thus
\[ G^{i+1}-G^i=(F^i \circ V_{i+1}+\mathrm{Id})\cdot G_{i+1}+G^i \circ V_{i+1}-G^i. \]
Proceeding exactly as we did for $E^i$ and $F^i$, using the convergence of~\eqref{limits},~\eqref{serie3} and~\eqref{serie5}, one finds that $G^i$ converges to a map
\[ G^* : \T^n \times \{\omega_0\} \rightarrow \T^n \times D_h \]
such that
\[ |G^*|_{s/2} \MP Q_0\Psi(Q_0)\varepsilon. \]
In summary, the map $\mathcal{F}^i$ converges to a map
\begin{equation*}
\mathcal{F}^* : \T^n \times D_{r/2} \times \{\omega_0\} \rightarrow \T^n \times D_{r,h} 
\end{equation*}
which belongs to $\mathcal{G}$, of the form
\[ 
\begin{cases}
\mathcal{F}^*(\theta,I,\omega_0)=(\Phi^*_{\omega_0}(\theta,I),\omega_0^*), \\
\Phi_{\omega_0}^*(\theta,I)=(\theta+E^*(\theta),I+F^*(\theta)\cdot I +G^*(\theta))
\end{cases}
 \]
with the estimates
\begin{equation}\label{estF}
|E^*|_{s/2} \MP \Psi(Q_0)\mu, \quad 
|F^*|_{s/2} \MP Q_0\Psi(Q_0)\mu, \quad |G^*|_{s/2} \MP Q_0\Psi(Q_0)\varepsilon, \quad 
|\omega_0^*-\omega_0| \MP \mu.
\end{equation} 
Then from the estimates
\[ |A^i|_{s_i} \leq \varepsilon_i, \quad |B^i|_{s_i} \leq \mu_i, \]
given in~\eqref{stepest2i}, Proposition~\ref{kamiter}, and the convergence~\eqref{limits}, it follows that both $A^i$ and $B^i$ converge to zero. Next from the estimates
\begin{equation*}
\begin{cases}
  |e^{i+1}-e^{i} \circ \varphi_{i+1}|_{s_{i+1}} \leq |A^{i}|_{s_{i}}, \\
  |\nabla_I^2R^{i+1}-\nabla_I^2R^{i} \circ  \mathcal{F}_{i+1}|_{s_{i+1}}\MP \eta
  |F_{i+1}|_{s_{i+1}} 
\end{cases}
\end{equation*}
still given in~\eqref{stepest2i}, Proposition~\ref{kamiter}, one can
prove in the same way as we did before, that $e^i$
converges to a trivial map
\[ e^* : \{\omega_0\} \rightarrow D_h, \quad e^*(\omega_0):=e_0^* \]
such that
\begin{equation}\label{estFF}
|e_0^*-e_{\omega_0^*}| \MP \varepsilon 
\end{equation}
whereas $M^i$ converges to a map
\[ M^* : \T^n \times D_{r/2} \times \{\omega_0\} \rightarrow \T^n
\times D_{r,h} \] 
such that, setting $R^*(\theta,I)=M^*(\theta,I)I\cdot I$,
\begin{equation}\label{estFFFF}
|\nabla_I^2 R^*-\nabla_I^2 R_{\omega_0^*}|_{s/2} \MP \eta Q_0\Psi(Q_0)\mu.
\end{equation} 
Therefore we have
\[ H \circ \mathcal{F}^*(\theta,I,\omega_0)= H_{\omega_0^*} \circ \Phi_{\omega_0}^*(\theta,I)=e_0^*+\omega_0 \cdot I
+R^*(\theta,I), \]
which, together with the previous estimates~\eqref{estF},~\eqref{estFF} and~\eqref{estFFFF}, is what we wanted to prove.
\end{proof}

\subsection{Proof of Theorem~\ref{classicalKAM}}\label{sec:deduc1}

In this section, we show how Theorem~\ref{KAMparameter} implies Theorem~\ref{classicalKAM}, following~\cite{Pos01}.

\begin{proof}[Proof of Theorem~\ref{classicalKAM}]
For simplicity, let us change the notation in~\eqref{Ham1} to consider a Hamiltonian $H : \T^n \times D \rightarrow \R$ of the form
\begin{equation*}
\begin{cases}
H(q,p)= h(p) +\epsilon f(q,p), \\
\nabla h(0):=\omega_0 \in \R^n
\end{cases}  
\end{equation*}
where $H$ is $(M,s_0)$-ultra-differentiable for some $s_0>0$: concretely, we just replaced the variables $(\theta,I)$ by $(q,p)$, $\varepsilon$ by $\epsilon$ and $s$ by $s_0$. Recall that we may assume that $\omega_0$ is of the form
  \[\omega_0=(1,\bar{\omega}_0) \in \R^n, \quad \bar{\omega}_0 \in
  [-1,1]^{n-1}.\]
For $p_0 \in B$, we expand $h$ in a small neighborhood of $p_0$: writing $p=p_0+I$ for $I$ close to zero, we get
\[ h(p)=h(p_0)+ \nabla_p h(p_0)\cdot I + \int_{0}^{1}(1-t)\nabla_p^2
h(p_0 +tI)\cdot I^2 \, dt. \]
Similarly, we expand $\epsilon f$ with respect to $p$, in a small neighborhood of $p_0$:
\[ \epsilon f(q,p)=\epsilon f(q,p_0)+ \epsilon \nabla_p f(q,p_0)\cdot
I + \epsilon\int_{0}^{1}(1-t)\nabla_p^2 f(q,p_0 +tI) \cdot I^2\, dt. \]
Since $\nabla_p h : B \rightarrow \Omega$ is a diffeomorphism, instead
of $p_0$ we can use $\omega=\nabla_p h(p_0)$ as a new variable, and
letting $\nabla_\omega g:=(\nabla h)^{-1}$, we write
\[ h(p)=e(\omega)+ \omega\cdot I + R_h(I,\omega) \] 
with 
\[ e(\omega):=h(\nabla_\omega g(\omega)), \quad
R_h(I,\omega):=\int_{0}^{1}(1-t)\nabla_p^2 h(\nabla_\omega
g(\omega)+tI) \cdot I^2 \, dt\]
and also, letting $\theta=q$, 
\[ \epsilon f(q,p)=\epsilon \tilde{A}(\theta,\omega)+ \epsilon \tilde{B}(\theta,\omega)\cdot I + \epsilon R_f(\theta,I,\omega) \]
with
\[ \tilde{A}(\theta,\omega):=f(\theta,\nabla_\omega g(\omega)), \quad \tilde{B}(\theta,\omega):=\nabla_p f(\theta,\nabla_\omega g(\omega))\] 
and
\[ R_f(\theta,I,\omega):=\epsilon\int_{0}^{1}(1-t)\nabla_p^2
f(\theta,\nabla_\omega g(\omega) +tI) \cdot I^2\, dt.   \]
Finally, we can set
\[ A:=\epsilon \tilde{A}, \quad B:=\epsilon \tilde{B}, \quad
R:=R_h+\epsilon R_f  = M(\theta,I,\omega)\cdot I^2, \] 
so that $h+\epsilon f$ can be written as
\[ H(\theta,I,\omega)=e(\omega)+\omega \cdot
I+A(\theta,\omega)+B(\theta,\omega)\cdot I+R(\theta,I,\omega),\]
and we have
\[ \nabla_I^2 R(\theta,I,\omega)
=\nabla_I^2 h(\nabla_\omega
  g(\omega)+I)+\epsilon\nabla_I^2 f(\theta,\nabla_\omega g(\omega)+I).\]
By assumption, $h$ and $f$ are $(M,s_0)$-ultra-differentiable on
$\T^n \times B$, and since the space of ultra-differentiable functions is
closed under taking derivatives, products, composition and inversion
(by the assumptions \eqref{H1} and \eqref{H2}, up to restricting the parameter $s_0$, see \S~\ref{sec:udiff} for
the relevant estimates), we claim that we can find $s>0$, $r>0$,
$h>0$ and $\tilde{c}>0$ which are independent of $\epsilon$ such that $H$ is $(M,s)$-ultra-differentiable on the domain $\T^n \times D_{r,h}$ with the estimates
\[ |A|_{s} \leq \tilde{c}\epsilon, \quad |B|_{s} \leq
\tilde{c}\epsilon,\quad
|\nabla_I^2 R|_{s} \leq \tilde{c}.  \] 
We may set
\[ \varepsilon:=\tilde{c}\epsilon, \quad \mu:=\sqrt{\varepsilon}, \quad \eta:=\tilde{c} \]
and assuming $\epsilon$ small enough, we have $\tilde{c}\epsilon \leq \mu=\sqrt{\varepsilon}$. Thus we have
\[ |A|_{s} \leq \varepsilon, \quad |B|_{s} \leq \mu, \quad |\nabla_I^2 R|_{s} \leq \eta.    \]
Having fixed $s>0$ and $r>0$, we may choose $Q_0$ sufficiently large so that~\eqref{eqQ0} holds true, and then by further restricting first $h$, and then $\epsilon$, we may assume that the condition~\eqref{cond0} is satisfied. Theorem~\ref{KAMparameter} applies: there exist an $(M,s/2)$-ultra-differentiable embedding $\Upsilon_{\omega_0} : \T^n \rightarrow \T^n \times D_r$, defined by 
\[ \Upsilon_{\omega_0}(\theta):=\Phi_{\omega_0}(\theta,0)=(\theta+E^*(\theta),G^*(\theta))\] 
where $\Phi_{\omega_0}$ is given by Theorem~\ref{KAMparameter}, and a vector $\omega_0^* \in \R^n$ such that $\Upsilon_{\omega_0}(\T^n)$ is invariant by the Hamiltonian flow of $H_{\omega_0^*}$ and quasi-periodic with frequency $\omega_0$. Moreover, $\omega_0^*$ and $\Upsilon_{\omega_0}$ satisfy the estimates
\begin{equation*}
|\omega_0^*-\omega_0| \leq c \mu, \quad |E^*|_{s/2} \leq c \Psi(Q_0)\mu, \quad |G^*|_{s/2}\leq c Q_0\Psi(Q_0)\varepsilon 
\end{equation*}
for some large constant $c>1$. Since $h$ is non-degenerate, there exists $p_0^*$ such that $\nabla h (p_0^*)=\omega_0^*$ and, up to taking $c>1$ larger and recalling that $\mu=\sqrt{\varepsilon}$, the above estimates imply 
\begin{equation}\label{estFFF}
|p_0^*| \leq c \sqrt{\varepsilon}, \quad |E^*|_{s/2} \leq c \Psi(Q_0)\sqrt{\varepsilon}, \quad |G^*|_{s/2}\leq c Q_0\Psi(Q_0)\varepsilon.  
\end{equation}
Now observe that an orbit $(\theta(t),I(t))$ for the Hamiltonian $H_{\omega_0^*}$ corresponds to an orbit $(q(t),p(t))=(\theta(t),I(t)+p_0^*)$ for our original Hamiltonian. Hence, if we define $T : \T^n \times \R^n \rightarrow \T^n \times \R^n$ by $T(\theta,I)=(\theta,I+p_0^*)$ and 
\[ \Theta_{\omega_0} = T \circ \Upsilon_{\omega_0} : \T^n \rightarrow \T^n
\times \R^n, \quad
\Theta_{\omega_0}(\theta)=(\theta+E^*(\theta),G^*(\theta)+p_0^*) \]
then $\Theta_{\omega_0}$ is an $(M,s/2)$-ultra-differentiable torus embedding
such that $\Theta_{\omega_0}(\T^n)$ is invariant by the Hamiltonian
flow of $H$ and quasi-periodic with frequency $\omega_0$.  The
estimates on the distance between $\Theta_{\omega_0}$ and the trivial
embedding $\Theta_0$ follows directly from~\eqref{estFFF}, which
finishes the proof.
\end{proof}

\subsection{Proof of Theorem~\ref{KAMvector}}\label{sec:deduc2}

Now we show how Theorem~\ref{KAMvector} follows from Theorem~\ref{KAMparameter}.

\begin{proof}[Proof of Theorem~\ref{KAMvector}]
Consider the vector field $X=\omega_0+B \in \mathcal{U}_{s}(\T^n,\R^n)$ as in the statement. It can be trivially included into a parameter-depending vector field: given $h>0$, let $\hat{X} \in \mathcal{U}_{s}(\T^n \times D_h,\R^n)$ be such that  
\[ \hat{X}(\theta,\omega)=\hat{X}_\omega(\theta)=\omega+B(\theta), \quad \omega \in D_h, \quad \hat{X}_{\omega_0}=X.  \]
Now given any $r>0$, consider the Hamiltonian $H$ defined on $\T^n \times D_{r,h}$ by
\begin{equation}\label{Hamvector}
H(\theta,I,\omega)=H_\omega(\theta,I):=\omega\cdot I +B(\theta)\cdot I. 
\end{equation}
Clearly, for any parameter $\omega$, the torus $\T^n \times \{0\}$ is invariant by the Hamiltonian vector field $X_{H_\omega}$, and, upon identifying $\T^n \times \{0\}$ with $\T^n$, the restriction of $X_{H_\omega}$ to this torus coincides with $\hat{X}_\omega$.  

Now the Hamiltonian $H$ defined in~\eqref{Hamvector} is of the form~\eqref{HAM} with $\varepsilon=\eta=0$ (and $e=0$) and therefore for $\mu$ sufficiently small, Theorem~\ref{KAMparameter} applies: there exist a vector $\omega_0^* \in \R^n$ and an $(M,s/2)$-ultra-differentiable embedding
\[ \Phi^*_{\omega_0} : \T^n \times D_{r/2} \rightarrow \T^n \times D_r \] 
here of the form
\[ \Phi^*_{\omega_0}(\theta,I)=(\theta+E^*(\theta),I+F^*(\theta)\cdot I) \]
with the estimates
  \begin{equation*}
    |\omega_0^*-\omega_0| \leq c_3 \mu, \quad |E^*|_{s/2} \leq
    c_3 \Psi(Q_0)\mu, \quad |F^*|_{s/2}\leq c_3
    \Delta(Q_0)\mu   
  \end{equation*}
  and  such that 
\begin{equation}\label{relH}
H_{\omega_0^*} \circ \Phi_{\omega_0}^*(\theta,I)=\omega_0 \cdot I.
\end{equation}
The embedding $\Phi^*_{\omega_0}$ clearly leaves invariant the torus $\T^n \times \{0\}$ and induces a diffeomorphism of this torus that can be identified to $\Xi:=\mathrm{Id}+E^*$. Writing the equality~\eqref{relH} in terms of Hamiltonian vector fields, we have, upon restriction to the invariant torus and recalling that the restriction of $X_{H_\omega}$ coincides with $\hat{X}_\omega$, 
\[ \Xi^*(\hat{X}_{\omega_0^*})=\omega_0. \]    
But $\hat{X}_{\omega_0^*}=\hat{X}_{\omega_0}+\omega_0^*-\omega_0=X+\omega_0^*-\omega_0$
and therefore
\[ \Xi^*(X+\omega_0^*-\omega_0)=\omega_0 \]
which, together with the estimates on $\omega_0^*$ and $\Xi-\mathrm{Id}=E^*$, was the statement we wanted to prove.    
\end{proof}

\subsection{Proof of Theorem~\ref{destruction}}\label{sec:bessi}

The goal of this short section is to show how Theorem~\ref{destruction} follows directly from the work of Bessi in~\cite{Bes00}.

In Bessi, one starts with a non-resonant vector $\omega \in \R^n$ which is assumed to be ``exponentially Liouville" in the following sense: there exists $s_0>0$ and a sequence $k_j \in \Z^n$ with $|k_j| \rightarrow +\infty$ as $j \rightarrow +\infty$ for which
\begin{equation}\label{condBessi}
0<|k_j\cdot \omega| \leq e^{-s_0|k_j|}\tag{$C_{1,s_0}$}.
\end{equation}
Given this sequence of $k_j \in \Z^n$, one can find another sequence $\tilde{k}_j \in \Z^n$ such that for all $j \in \N$, $|\tilde{k}_j|\leq |k_j|$, $\tilde{k}_j \cdot k_j=0$ and $|\tilde{k}_j\cdot \omega| \geq c|\tilde{k}_j|$ for some constant $c>0$ independent of $j$. 

Then one defines the following sequence of Hamiltonians on $\R^{n}/(2\pi\Z^n) \times \R^n$ (which are similar to the Hamiltonian considered by Arnold in~\cite{Arn64}):
\begin{equation}\label{HamBessi}\tag{$H_{1,j,s}$}
\begin{cases}
H^{1,j}_{\varepsilon,\mu}(\theta,I):=\frac{1}{2}I\cdot I+F^{1,j}_{\varepsilon,\mu}(\theta), \; (\theta,I) \in \R^{n}/(2\pi\Z^n) \times \R^n \\
F^{1,j}_{\varepsilon,\mu}(\theta):=\varepsilon\nu_{1,j,s}(1-\cos(k_j\cdot\theta))(1+\mu\tilde{\nu}_{1,j,s}\cos(\tilde{k}_j\cdot\theta)) \\
0<\varepsilon \leq 1, \; 0 < \mu \leq 1, \; \nu_{1,j,s}:=e^{-s|k_j|}, \; \tilde{\nu}_{1,j,s}:=e^{-s|\tilde{k}_j|}.  
\end{cases}
\end{equation}
Observe that the only role of the sequences $\nu_{1,j,s}$ and $\tilde{\nu}_{1,j,s}$ is to ensure that the sequence of perturbations $F^{1,j}_{\varepsilon,\mu}$ satisfy, for all $j \in \N $ and all $0 \leq \mu \leq 1$:
\[ |F^{1,j}_{\varepsilon,\mu}|_s:=\sup_{\theta \in \C^n /(2\pi\Z^n), \; |\mathrm{Im}(\theta)\leq s|}|F^{1,j}_{\varepsilon,\mu}(\theta)|\leq 4\varepsilon. \]
In~\cite{Bes00}, Bessi proved the following theorem.

\begin{theorem}[Bessi]
Assume that $\omega \in \R^n$ satisfy~\eqref{condBessi}. Then, if $s_0>s$, for any $0 \leq \varepsilon \leq 1$, there exists $\mu_\varepsilon>0$ and $j_\varepsilon \in \N$ such that for any $0<\mu \leq \mu_\varepsilon$ and any $j \geq j_\varepsilon$, the Hamiltonian system defined in~\eqref{HamBessi} does not have any invariant torus $\mathcal{T}$ satisfying
\begin{itemize}
\item[$(i)$] $\mathcal{T}$ projects diffeomorphically to $\T^n$;
\item[$(i)$] There is a $C^1$ diffeomorphism between $\T^n$ and $\mathcal{T}$ which conjugates the flow on $\mathcal{T}$ to the linear flow on $\T^n$ of frequency $\omega$. 
\end{itemize}
\end{theorem}

It is clear that it is the regularity of the perturbation, here the analyticity which causes the exponential decay of the Fourier coefficients, that forces the condition~\eqref{condBessi}. If the perturbation is assumed to be only of class $C^r$ for some $r \in \N$, then~\eqref{condBessi} can be weakened to cover frequencies $\omega$ which are Diophantine with an exponent $\tau$ which is related to $r$ (this can be obtained from Bessi's work; one can find a better quantitative result in~\cite{CW13}, which also uses ideas of~\cite{Bes00}).

Here we would like to consider the case where the perturbation is $M$-ultra-differentiable, for a given sequence $M=(M_l)_{l \in \N}$; we will consider a slight modification of the family of Hamiltonians~\eqref{HamBessi} to a family of Hamiltonians~\eqref{HamBessiG} depending on $M$, which are still analytic but for which the perturbation are bounded and small only in a $M$-ultra-differentiable norm.

First we need to compute the $M$-norm of the function $P_k(\theta):=\cos(k\cdot \theta)$ for an arbitrary $k \in \Z^n$. Using the fact that $(l+1)^2 \leq 4^l$ for any $l \in \N$ and recalling the definition of the function $\Omega$ in~\eqref{Ofunction}, we have
\begin{equation*}
|P_k|_{s}=c\sup_{l \in \N}\frac{(l+1)^2s^l |k|^l}{M_l} \leq c \sup_{l \in \N}\frac{(4|k|s)^l}{M_l}=c\exp(\Omega(4|k|s)).
\end{equation*}
Now we introduce the following condition on the non-resonant vector $\omega \in \R^n$: there exists $s_0>0$ and a sequence $k_j \in \Z^n$ with $|k_j| \rightarrow +\infty$ as $j \rightarrow +\infty$ for which
\begin{equation}\label{condBessiG}\tag{$C_{M,s_0}$}
0<|k_j\cdot \omega| \leq \exp(-\Omega(4|k|s_0)).
\end{equation}
For the sequence $M=M_{1}$, this condition reduces to~\eqref{condBessi}, up to an unimportant factor $4$. Then we consider the following modified sequence of Hamiltonians:
\begin{equation}\label{HamBessiG}\tag{$H_{M,j,s}$}
\begin{cases}
H^{M,j}_{\varepsilon,\mu}(\theta,I):=\frac{1}{2}I\cdot I+F^{M,j}_{\varepsilon,\mu}(\theta), \; (\theta,I) \in \R^{n}/(2\pi\Z^n) \times \R^n \\
F^{M,j}_{\varepsilon,\mu}(\theta):=\varepsilon\nu_{M,j,s}(1-\cos(k_j\cdot\theta))(1+\mu\tilde{\nu}_{M,j,s}\cos(\tilde{k}_j\cdot\theta)) \\
0<\varepsilon,\mu \leq 1, \; \nu_{M,j,s}:=\exp(-\Omega(4|k_j|s_0)), \; \tilde{\nu}_{M,j,s}:=\exp(-\Omega(4|\tilde{k}_j|s_0)).  
\end{cases}
\end{equation}
With these choices of $\nu_{M,j,s}$ and $\tilde{\nu}_{M,j,s}$ we have that, for all $j \in \N $ and all $0 \leq \mu \leq 1$:
\[ |F^{M,j}_{\varepsilon,\mu}|_{s} \leq 4c\varepsilon. \]
The argument of Bessi goes exactly the same of way for this family of Hamiltonians~\eqref{HamBessiG} under the condition~\eqref{condBessiG}, up to the following minor point one has to check: recall from~\eqref{Omegalog} that $\Omega$ satisfies 
\[ \lim_{t \rightarrow +\infty} \frac{\Omega(t)}{\ln(t)}=+\infty \]
and thus, using the condition that $s_0>s$, one obtains
\[ \limsup_{j \rightarrow +\infty}\left(\exp(\Omega(4|k_j|s_0)-\Omega(4|k_j|s))\right)=+\infty \]
which is necessary for Bessi's argument to go through. So we have the following statement.

\begin{theorem}\label{BessiG}
Assume that $\omega \in \R^n$ satisfy~\eqref{condBessiG}. Then, if $s_0> s$, for any $0 \leq \varepsilon \leq 1$, there exists $\mu_\varepsilon>0$ and $j_\varepsilon \in \N$ such that for any $0<\mu \leq \mu_\varepsilon$ and any $j \geq j_\varepsilon$, the Hamiltonian system defined in~\eqref{HamBessiG} does not have any invariant torus $\mathcal{T}$ satisfying
\begin{itemize}
\item[$(i)$] $\mathcal{T}$ projects diffeomorphically to $\T^n$;
\item[$(i)$] There is a $C^1$ diffeomorphism between $\T^n$ and $\mathcal{T}$ which conjugates the flow on $\mathcal{T}$ to the linear flow on $\T^n$ of frequency $\omega$. 
\end{itemize}
\end{theorem}

Now Theorem~\ref{BessiG} implies Theorem~\ref{destruction}, as if $\omega$ satisfies~\eqref{conddestruction}, then it satisfies~\eqref{condBessiG} for some $s_0>0$ and it is sufficient to consider a Hamiltonian system as in~\eqref{HamBessiG} with $s<s_0$.

\section{Application to Hamiltonian normal forms and Nekhoroshev theory}\label{sec:Nek}

In this section, we give applications to Hamiltonian normal forms and Nekhoroshev theory: we prove Theorem~\ref{lineaire} and Theorem~\ref{difflineaire}, Theorem~\ref{nonlineaire}, Theorem~\ref{convex} and Theorem~\ref{diffconvex}, and finally Theorem~\ref{steep}.

In the special case where the Hamiltonians are Gevrey regular, proofs of all those results are contained in~\cite{MS02}, \cite{Bou13} and~\cite{Bou11}. Our purpose here is to give an extension to ultra-differentiable Hamiltonians; our crucial technical result will be a flexible normal form around a periodic frequency (Proposition~\ref{propperiodic} in \S~\ref{sec:periodic}), as all the results will be deduced from it.

As before, we shall write
$$u \MP v \quad (\mbox{respectively } u \PM v, \quad u \EP v, \quad u \PE v)$$
if, for some unimportant ``implicit" constant $c\geq 1$, we
have 
$$u \leq c v \quad (\mbox{respectively } cu \leq v, \quad u =c v, \quad cu = v).$$
Clearly, what we will consider as being unimportant depends on the precise setting, and it will be made clear at the beginning of each of the next sections.

\subsection{Periodic averaging}\label{sec:periodic}

Recall that a vector $v \in \R^n\setminus\{0\}$ is called periodic if there exists $t>0$ such that $tv \in \Z^n$; its period $T$ is then defined as the smallest such $t>0$. Let $v$ be a $T$-periodic vector, and let us denote by $L_v(I)=v\cdot I$ the linear integrable Hamiltonian with frequency $v$.

In this section, given positive real parameters $s,\delta,\mu,\rho,\nu$, we shall consider the following Hamiltonian: 
\begin{equation}\label{HamN}
\begin{cases}\tag{$\sharp$}
H(\theta,I)=L_v(I)+S(I)+R(I)+F(\theta,I), \\
H : \T^n \times D_{\delta} \rightarrow \R \\
|\nabla S|_{s} \leq \mu, \quad |\nabla R|_{s} \leq \rho, \quad |F|_{s} \leq \nu, \\
\eta:=\max\{\mu,\rho\},
\end{cases}
\end{equation}
where $D_{\delta} \subset \R^n$ is the ball of radius $\delta$ around the origin. 

The goal of this section is to prove the following result, in which implicit constants will depend only on $n$ and where $\{\,.\,\}$ denotes the usual Poisson bracket.

\begin{proposition}\label{propperiodic}
Let $H$ be as in~\eqref{HamN}, with $M$ satisfying \eqref{H1} and \eqref{H2} and $v$ $T$-periodic. There exists a constant $A \EP 1$ such that, for $\xi>1$, if
\begin{equation}\label{seuil1}
\nu \leq C(\kappa_\xi)^{-1}s\eta, \quad C^{-1}(s(AT\eta)^{-1}) \leq \kappa_\xi, \quad \kappa_\xi:=\ln \xi(24)^{-1}, 
\end{equation}
then there exists a $(M,s/\xi)$-ultra-differentiable symplectic transformation
\[ \Phi : \T^n \times D_{\delta/\xi} \rightarrow \T^n \times D_{\delta} \]
such that
\[ H\circ\Phi =L_v+S+R+\bar{F}+\hat{F} \]
with $\{\bar{F},L_v\}=0$ and the estimates
\begin{equation}\label{EST1}
|\Phi-\mathrm{Id}|_{s/\xi}\leq 2 T\nu, \quad |\bar{F}|_{s/\xi} \leq 2\nu, \quad |\hat{F}|_{s/\xi} \leq \nu\exp\left(-\kappa_\xi\left(C^{-1}(s(AT\eta)^{-1})\right)^{-1}\right). 
\end{equation}
Moreover, if $I$ is an integrable Hamiltonian such that $\{I,F\}=0$, then $\{I,\bar{F}\}=0$.
\end{proposition}

Proposition~\ref{propperiodic} will be proved by iterating a large number of times an elementary averaging step we now describe. Let us consider the following slightly more general setting:
\begin{equation}\label{HamNN}
\begin{cases}
H(\theta,I)=L_v(I)+S(I)+R(I)+G(\theta,I)+F(\theta,I), \\
H : \T^n \times D_{\delta} \rightarrow \R \\
|\nabla S|_{s} \leq \mu, \quad |\nabla R|_s \leq \rho, \quad |F|_{s} \leq \nu, \quad |G| \leq 2\nu , \quad \{G,L_v\}=0, \\
\eta:=\max\{\mu,\rho\}  
\end{cases}
\end{equation}
which reduces to~\eqref{HamN} when $G=0$. Given a $T$-periodic frequency $v \in \R^n\setminus \{0\}$ and a function $K : \T^n \times D_{\delta} \rightarrow \R$, let us define $[K]_v : \T^n \times D_{\delta} \rightarrow \R$ by
\[ [K]_v(\theta,I)=\int_{0}^{1}K(\theta+tTv,I)dt. \] 
Clearly, $[K]_v$ is the projection of $K$ onto the kernel of the operator $\{\,.\,,L_v\}$, that is $\{K,L_v\}=0$ if and only then $K=[K]_v$.

\begin{lemma}\label{lemmeriodic}
Let $H$ be as in~\eqref{HamNN}, with $M$ satisfying \eqref{H1} and \eqref{H2} and $v$ $T$-periodic. Given $0<\sigma<1$, assume that
\begin{equation}\label{seuil2}
C(\sigma)^2T\nu \PM s^2. 
\end{equation}
Then there exist a $(M,s(1-\sigma)^3)$-ultra-differentiable symplectic transformation
\[ \Psi : \T^n \times D_{\delta(1-\sigma)^3} \rightarrow \T^n \times D_{\delta} \]
such that
\[ H\circ\Psi =L_v+S+R+G+[F]_v+F^+ \]
with the estimates
\begin{equation}\label{EST2}
|\Psi-\mathrm{Id}|_{s(1-\sigma)^3}\leq T\nu, \quad |F^+|_{s(1-\sigma)^3} \MP (T\nu^2 C(\sigma)^2s^{-2}+T\eta\nu C(\sigma)s^{-1}).  
\end{equation}
Moreover, if $I$ is an integrable Hamiltonian such that $\{I,F\}=\{I,G\}=0$, then $\{I,[F]_v\}=\{I,F^+\}=0$.
\end{lemma}

\begin{proof}
Let us define $Y : \T^n \times D_{\delta} \rightarrow \R$ by
\[ Y(\theta,I)=T\int_0^1(F-[F]_v)(\theta+tTv,I)tdt. \]
It is easy to check, using an integration by part, that
\begin{equation}\label{cohoper}
\{Y,N_{v}\}=F-[F]_{v}=F+S+R+G-[F+S+R+G]_{v}
\end{equation}
where the second equality follows from the fact that $[S+R+G]_v=S+R+G$ since $S,R$ are integrable and $\{G,L_v\}=0$ by assumption. Moreover, we have the estimates
\begin{equation}\label{estflotY}
|Y|_s \leq \frac{T}{2}|F-[F]_v|_s \leq T|F|_s \leq T\nu.
\end{equation}
Recalling from~\eqref{Csigma} that $(e\sigma)^{-1} \leq C(\sigma)$, the inequalities~\eqref{seuil2} imply that~\eqref{smallF} is satisfied, and thus Proposition~\ref{flotF} can be applied: for any $t \in [0,1]$, the time-$t$ map $Y^t$ of the Hamiltonian flow of $Y$ belongs to $\mathcal{U}_{s(1-\sigma)^3}(\T^n \times D_{\delta(1-\sigma)^3},\T^n \times D_{\delta(1-\sigma)^2})$ and we have the estimate
\begin{equation}\label{estflotYt}
|Y^t-\mathrm{Id}|_{s(1-\sigma)^3} \leq t|Y|_{s}.
\end{equation} 
We set $\Psi=Y^1$: then~\eqref{estflotYt} and~\eqref{estflotY} gives the first estimate of~\eqref{EST2}. Now a well-known computation using Taylor's formula with integral remainder yields
\[ H \circ \Psi=L_v+S+R+G+[F]_v+F^+ \]
with
\[ F^+:=\int_0^1 \tilde{F}_t \circ Y^tdt, \quad \tilde{F}_t:=\{tF+(1-t)[F]_v+S+R+G,Y\}. \]
To estimate $F^+$, observe that $Y^t$ maps $\T^n \times D_{\delta(1-\sigma)^3}$ into $\T^n \times D_{\delta(1-\sigma)^2}$, so using~\eqref{estflotYt} and~\eqref{seuil2}, we can apply Proposition~\ref{composition} to get
\[ |F^+|_{s(1-\sigma)^3} \leq |\tilde{F}_t|_{s(1-\sigma)^2}.  \]
Now $\tilde{F}_t$, for any $t \in [0,1]$, can be easily estimated using Corollary~\ref{corderivative} and thus we get
\begin{eqnarray*}
|\tilde{F}_t|_{s(1-\sigma)^2} & \MP & (T\nu^2C(\sigma)^2s^{-2}+T\nu^2C(\sigma)^2s^{-2}+T\nu\mu C(\sigma)s^{-1} \\
&  & +T\nu\rho C(\sigma)s^{-1}+T\nu^2C(\sigma)^2s^{-2})  \\
& \MP & (T\nu^2C(\sigma)^2s^{-2}+T\nu\eta C(\sigma)s^{-1})
\end{eqnarray*}
as $\eta=\max\{\mu,\rho\}$ by definition. This gives the second estimate of~\eqref{EST2}.

It remains to prove the second part of the statement, so let $I$ be an integrable Hamiltonian such that $\{I,F\}=\{I,G\}=0$. Obviously, as $I$ is integrable we also have $\{I,L_v\}=0$, which gives $\{I,[F]_v\}=0$: indeed, letting $L_v^{tT}$ the $tT$-time of the Hamiltonian flow associated to $L_v$, as $\{I,L_v\}=0$ we have $I \circ L_v^{tT}=I$ hence 
\begin{eqnarray*}
\{I,[F]_v\} & = & \left\{I, \int_0^1 F \circ L_v^{tT} dt \right \} \\
& = & \int_0^1 \{I,  F \circ L_v^{tT} \}dt \\
& = & \int_0^1 \{I \circ L_v^{tT},F \circ L_v^{tT} \}dt \\
& = & \int_0^1 \{  I ,F  \} \circ L_v^{tT}dt \\
& = & 0.
\end{eqnarray*}
Using this, one obtains then $\{I,Y\}=0$ by a similar argument. Now using $\{I,Y\}=0$ and again the same argument, to show that $\{I,F^+\}=0$ one needs to show that $\{I,\tilde{F}_t\}=0$ for all $t \in [0,1]$: but this in turns follows from Jacobi identity, as each element in the bracket $\tilde{F}_t$ commutes with $I$. This concludes the proof. 
\end{proof}

To prove Proposition~\ref{propperiodic}, it remains to apply inductively Lemma~\ref{lemmeriodic} $m$ times, for some optimized choice of $m$, choosing the same $\sigma \sim m^{-1}$ at each step. Actually, to obtain sharper estimates, we will use a trick due to Neishtadt which consists in applying Lemma~\ref{lemmeriodic} a first time with $\sigma \sim 1$ large, and then $(m-1)$ times with $\sigma \sim m^{-1}$ uniformly small.  

\begin{proof}[Proof of Proposition~\ref{propperiodic}] 
Let $m \geq 1$ be an integer to be chosen below. We set $\sigma_1:=\ln \xi(24)^{-1}$ and for $m \geq 2$, $\sigma_m:=\ln \xi(24(m-1))^{-1}$. For all $0 \leq j \leq m$, we define
\[ \nu_j:=e^{-j}\nu, \quad \gamma_j:=\sum_{i=0}^{j-1}\nu_i,\] 
and
\[ s_j:=s(1-\sigma_1)^3(1-\sigma_m)^{3(j-1)}, \quad \delta_j:=\delta(1-\sigma_1)^3(1-\sigma_m)^{3(j-1)}.  \] 
We claim that if
\begin{equation}\label{seuil3}
\nu \leq C(\sigma_1)^{-1}s \eta, \quad T\eta \PM C(\sigma_m)^{-1} s 
\end{equation}
then for all $1 \leq j \leq m$, there exist a $(M,s_j)$-ultra-differentiable symplectic transformation
\[ \Phi_j : \T^n \times D_{\delta_j} \rightarrow \T^n \times D_{\delta} \]
such that
\[ H\circ\Phi_j =L_v+S+R+\bar{F}_j+\hat{F}_j \]
with $\{\bar{F}_j,L_v\}=0$ and the estimates
\begin{equation*}
|\Phi_j-\mathrm{Id}|_{s_j}\leq T\gamma_j, \quad |\bar{F}_j|_{s_j} \leq \gamma_j,\quad |\hat{F}_j|_{s_j} \MP \nu_{j-1}(T\eta C(\sigma_1)s^{-1}):=\chi_j.
\end{equation*}
Moreover, if $I$ is an integrable Hamiltonian such that $\{I,F\}=0$, then we have $\{I,\bar{F}_j\}=\{I,\hat{F}_j\}=0$. Observe that $\gamma_j < 2\nu$ and as $\sigma_1 \geq \sigma_m$, one can use~\eqref{seuil3} to ensure that $\chi_j \leq \nu_j$.

The statement follows from this claim: indeed, it suffices to choose the integer $m$ as
\[ \ln \xi(24)^{-1}\left(C^{-1}(s(AT\eta)^{-1})\right)^{-1} \leq m \leq 1+ \ln \xi(24)^{-1}\left(C^{-1}(s(AT\eta)^{-1})\right)^{-1}  \]
for a suitable constant $A\EP 1$. Clearly~\eqref{seuil1} gives~\eqref{seuil3} and $m \geq 1$, so we may set
\[ \Phi:=\Phi_m, \quad \bar{F}:=\bar{F}_m, \quad \hat{F}:=\hat{F}_m, \] 
observe that $s_m \geq s/\xi$, $\delta_m \geq \delta/\xi$, and thus
 \[ \Phi : \T^n \times D_{\delta/\xi} \rightarrow \T^n \times D_{\delta} \]
is a $(M,s/\xi)$-ultra-differentiable symplectic transformation such that
\[ H\circ\Phi =L_v+S+R+\bar{F}+\hat{F} \]
with $\{\bar{F},L_v\}=0$ and the estimates
\begin{equation*}
|\Phi-\mathrm{Id}|_{s/\xi}\leq 2T\nu, \quad |\bar{F}|_{s/\xi} \leq 2\nu, 
\end{equation*}
and
\[ |\hat{F}|_{s/\xi} \leq \nu\exp(-m)\leq \nu\exp(-\ln \xi(24)^{-1}\left(C^{-1}(s(AT\eta)^{-1})\right)^{-1}). \]

So it remains to prove the claim, and we will do it by induction on $j$. The second part of the claim is obvious by induction using the second part of Lemma~\ref{lemmeriodic}, so we will only prove the first part of the claim. For $j=1$,~\eqref{seuil3} implies that
\[ C(\sigma_1)^2T\nu \leq C(\sigma_1)T\eta s \PM s^2 \] 
and thus~\eqref{seuil2} is satisfied: Lemma~\ref{lemmeriodic} applies, with $\sigma=\sigma_1$ and $G=0$, and we set $\Phi_1:=\Psi$, $\bar{F}_1:=[F]_v$ and $\hat{F}_1=F^+$. Clearly, $\{\bar{F}_1,L_v\}=0$. As $\gamma_1=\nu_0=\nu$, we do have
\[ |\Phi_1-\mathrm{Id}|_{s_1}\leq T\gamma_1, \quad |\bar{F}_1|_{s_1} \leq \gamma_1. \] 
For $\hat{F}_1$, observe that from the first part of~\eqref{seuil3} one has
\[ T\nu^2 C(\sigma_1)^2s^{-2} \leq T\nu\eta C(\sigma_1)s^{-1}  \]
and thus
\[ |\hat{F}_1|_{s_1} \MP (T\nu^2 C(\sigma_1)^2s^{-2} +T\nu\eta C(\sigma_1)s^{-1}) \MP T\nu\eta C(\sigma_1)s^{-1}=\chi_1.  \]
Now assume that the statement is true for some $1 \leq j \leq m-1$, and let us prove it remains true for $2 \leq j+1 \leq m$. The inductive hypothesis gives a $(M,s_j)$-ultra-differentiable symplectic transformation
\[ \Phi_j : \T^n \times D_{\delta_j} \rightarrow \T^n \times D_{\delta} \]
such that
\[ H\circ\Phi_j =L_v+S+R+\bar{F}_j+\hat{F}_j \]
with $\{\bar{F}_j,L_v\}=0$ and the estimates
\begin{equation*}
|\Phi_j-\mathrm{Id}|_{s_j}\leq T\gamma_j, \quad |\bar{F}_j|_{s_j} \leq \gamma_j,\quad |\hat{F}_j|_{s_j} \MP \nu_{j-1}(T\eta C(\sigma_1)s^{-1}).
\end{equation*}
We wish to apply Lemma~\ref{lemmeriodic} to this Hamiltonian with $\sigma=\sigma_m$, $\bar{F}_j$ instead of $G$ and $\hat{F}_j$ instead of $F$: we have
\[ |\bar{F}_j| \leq \gamma_j \leq 2\nu \]
and, from the first part of~\eqref{seuil3}
\begin{equation}\label{Tchi}
T\chi_j \leq T\nu_{j-1} T\mu C(\sigma_1)s^{-1}\leq T\nu T\eta C(\sigma_1)s^{-1}\leq (T\eta)^2
\end{equation}
so that the second part of~\eqref{seuil3} implies~\eqref{seuil2}. Lemma~\ref{lemmeriodic} applies to give a transformation $\Psi_j$ such that, if define $\Phi_{j+1}:=\Phi_j \circ \Psi_j$, then  
\[ \Phi_{j+1} : \T^n \times D_{\delta_{j+1}} \rightarrow \T^n \times D_{\delta} \]
is an $(M,s_{j+1})$-ultra-differentiable symplectic transformation such that
\[ H\circ\Phi_{j+1}=L_v+S+R+\bar{F}_j+[\hat{F}_j]_{v}+\hat{F}_j^+ \]
and hence, setting $\bar{F}_{j+1}:=\bar{F}_j+[\hat{F}_j]_{v}$ and $\hat{F}_{j+1}:=\hat{F}_j^+$, we have
\[ H\circ\Phi_{j+1} =L_v+S+R+\bar{F}_{j+1}+\hat{F}_{j+1}. \]
Clearly $\{\bar{F}_{j+1},L_v\}=0$ and as $|\hat{F}_j|_{s_j} \leq \nu_j$, we have
\[ |\bar{F}_{j+1}|_{s_{j+1}} \leq |\bar{F}_j|_{s_j} + |[\hat{F}_j]_v|_{s_j} \leq \gamma_j+\nu_j=\gamma_{j+1}  \]
and, using Proposition~\ref{composition},
\[ |\Phi_{j+1}-\mathrm{Id}|_{s_{j+1}}\leq |\Phi_j-\mathrm{Id}|_{s_j}+|\Psi_j-\mathrm{Id}|_{s_{j+1}}\leq  T\gamma_j+T\nu_j=T\gamma_{j+1}. \]
For $\hat{F}_{j+1}$, using~\eqref{Tchi} and the second part of~\eqref{seuil3} we can estimate
\begin{eqnarray*}
|\hat{F}_{j+1}| & \MP & (T\chi_j^2C(\sigma_m)^2s^{-2}+T\eta\chi_j C(\sigma_m)s^{-1}) \\
& \MP & \chi_j(T^2\eta^2 C(\sigma_m)^2s^{-2}+T\eta C(\sigma_m)s^{-1})  \\
& \MP & \chi_j T\eta C(\sigma_m)s^{-1}  \\
& \leq & \chi_{j+1}.
\end{eqnarray*}
This proves the claim, and thus the proposition.
\end{proof}

\subsection{Stability in the linear case: proof of Theorem~\ref{lineaire}}\label{sec:linear}

The goal here is to prove Theorem~\ref{lineaire} and Corollary~\ref{corlineaire}; implicit constants will depend only on $n$, $\omega$ and the function $C$. 

Our proof will use Proposition~\ref{propperiodic} (which can be seen as Theorem~\ref{lineaire} in the special case where $d=1$) which we proved in \S~\ref{sec:periodic}. We will also use a slightly more general version of Proposition~\ref{dio} (\S~\ref{sec:approx}) which we now state and which is still contained in~\cite{BF13}.

\begin{proposition}\label{dio2}
Let $\omega=(\bar{\omega},0) \in \R^d \times \R^{n-d}$ with $\bar{\omega}$ a non-resonant vector. For any $Q\geq n+2$, there exist $d$ periodic vectors $v_1, \dots, v_d \in \R^d \times \{0\}$, of periods $T_1, \dots, T_d$, such that $T_1v_1, \dots, T_dv_d$ form a $\Z$-basis of $\Z^d \times \{0\}$ and for $j\in\{1,\dots,d\}$,
\[ |\omega_0-v_j|\MP(T_j Q)^{-1}, \quad 1 \leq T_j \MP \Psi(Q).\]
\end{proposition}

Given any $\omega \in \R^n$, let us now define
\begin{equation}\label{moy2}
[H]_{\omega}(\theta,I):=\lim_{s \rightarrow +\infty}\frac{1}{s}\int_{0}^s H(\theta+t\omega,I)dt.
\end{equation}
When $\omega=(\bar{\omega},0) \in \R^d \times \R^{n-d}$ with $\bar{\omega}$ a non-resonant vector, clearly $[H]_{\omega}=[H]_d$ where
\[ [H]_d(\theta,I)=\int_{\T^d}H(\theta,I)d\theta_1\dots d\theta_d. \]
We then have the following result, which, as before, is slightly more general than Proposition~\ref{cordio2}, and is a consequence of the fact that $T_1v_1, \dots, T_dv_d$ form a $\Z$-basis of $\Z^d \times \{0\}$.

\begin{proposition}\label{cordio2}
Let $v_1,\dots,v_d \in \R^d \times \{0\}$ be the periodic vectors given by Proposition~\ref{dio2}, and $H$ a function defined on $\T^n \times D_{r}$. Then
\[ [H]_{v_1,\dots,v_d}=[H]_\omega=[H]_d \]
and $\{H,L_\omega\}=0$ if and only if $\{H,L_{v_j}\}=0$ for any $1 \leq j \leq d$.   
\end{proposition} 

Given positive real parameters $s,\rho,\nu$, we shall consider the following Hamiltonian: 
\begin{equation}\label{HamNNN}
\begin{cases}
H(\theta,I)=L_{\omega}(I)+R(I)+F(\theta,I), \\
H : \T^n \times D \rightarrow \R, \quad \omega=(\bar{\omega},0) \in \R^d \times \R^{n-d} \\
|\nabla R|_{s} \leq \rho, \quad |F|_{s} \leq \nu  
\end{cases}
\end{equation}
where $\bar{\omega}$ is non-resonant. In the special case where $R=0$, $\rho=0$ and $\nu=\varepsilon$, this is exactly a Hamiltonian $H$ as in~\eqref{Ham1} with $h=L_\omega$, which is the setting of Theorem~\ref{lineaire}. 

We are now ready to prove the following statement, which will easily imply Theorem~\ref{lineaire}. 

\begin{proposition}\label{proplineaire}
Let $H$ be as in~\eqref{HamNNN}, where $H$ is $(M,s)$-ultra-differentiable with $M$ satisfying \eqref{H1} and \eqref{H2} and $h=L_\omega$ with $\omega=(\bar{\omega},0) \in \R^d \times \{0\}$ and $\bar{\omega}$ non-resonant. For $Q \geq n+2$, assume that
\begin{equation}\label{seuillin}
\rho Q\Psi(Q) \PM 1, \quad \nu Q\Psi(Q)s^{-1} \PM 1, \quad \lambda sQ \SP 1, 
\end{equation}
for some $\lambda \PE 1$. Then there exists a $(M,s/2)$-ultra-differentiable symplectic transformation
\[ \Phi : \T^n \times D_{1/2} \rightarrow \T^n \times D \]
such that
\[ H\circ\Phi =L_\omega+R+\bar{F}+\hat{F} \]
where $\{\bar{F},L_\omega\}=0$ and with the estimates
\begin{equation*}
|\Phi-\mathrm{Id}|_{s/2}\MP \Psi(Q)\nu, \quad |\bar{F}|_{s/2} \MP \nu, \quad |\hat{F}|_{s/2} \MP \nu\exp\left(-\kappa_{2^{1/d}}\left(C^{-1}(\lambda sQ)\right)^{-1}\right). 
\end{equation*}
\end{proposition}

\begin{proof}[Proof of Theorem~\ref{lineaire}]
Here $R=0$, $\rho=0$ and $F=f$ with $\nu=\varepsilon$, it suffices to choose
\[ Q:=\Delta^{-1}\left(c_2s\varepsilon^{-1}\right), \quad c_2 \EP 1 \]
and to verify that choosing $\varepsilon$ sufficiently small allows indeed to apply Proposition~\ref{proplineaire}.
\end{proof}

\begin{proof}[Proof of Proposition~\ref{proplineaire}]
Recall that we are considering
\[ H=L_\omega+R+F. \]
Since $Q \geq n+2$, we can apply Proposition~\ref{dio2}: there exist $d$ periodic vectors $v_1, \dots, v_d$, of periods $T_1, \dots, T_d$, such that $T_1v_1, \dots, T_dv_d$ form a $\Z$-basis of $\Z^d \times \{0\}$ and for $j\in\{1,\dots,d\}$,
\[ |\omega-v_j|\MP(T_j Q)^{-1}, \quad 1 \MP T_j \MP \Psi_\omega(Q) .\] 
For $j\in\{1,\dots,d\}$, let us define 
\[ S_j=L_\omega-L_{v_j}, \quad \mu_j \EP (T_j Q)^{-1} \]
with a suitable implicit constant so that 
\[ |\nabla S_j|_{s}=c|\omega-v_j| \leq \mu_j.\] 
Note that $L_\omega=L_{v_j}+S_j$, and that $T_j\nu_j \EP Q^{-1}$ for any $1 \leq j \leq d$. Let us further choose $\xi:=2^{1/d}$ and define, for $1 \leq j \leq d$ 
\[ s_j:=s\xi^{-j}, \quad \delta_j:=\xi^{-j} \]
so that in particular $s_d=s/2$ and $\delta_d=1/2$. 

We claim that for all $1 \leq j \leq d$, there exists a $(M,s_j)$-ultra-differentiable symplectic map 
$$\Phi_j : \T^n \times D_{\delta_j} \longrightarrow \T^n \times D$$ 
such that 
\[ H\circ\Phi_j=L_\omega+R+\bar{F}_j+\hat{F}_j, \]
with $\{\bar{F}_j,L_{v_i}\}=0$ for any $1 \leq i \leq j$, and with the estimates
\begin{equation}
\begin{cases}
|\Phi_j-\mathrm{Id}|_{s_j} \MP \Psi(Q)\nu, \quad |\bar{F}_j|_{s_j} \MP \nu, \\
|\hat{F}_j|_{s_j} \MP \nu \exp(-\kappa_{2^{1/d}} \left(C^{-1}(\lambda sQ)\right)^{-1}).  
\end{cases}
\end{equation}
The proof of the proposition follows from this claim: it is sufficient to let 
\[ \Phi:=\Phi_d, \quad \bar{F}:=\bar{F}_d,\quad \hat{F}:=\hat{F}_d\] 
since from Corollary~\ref{cordio}, $\{\bar{F},L_{v_i}\}=0$ for $1 \leq i \leq d$ is equivalent to $\{\bar{F},L_{\omega}\}=0$.

Now let us prove the claim by induction on $1 \leq j \leq d$. For $j=1$, our Hamiltonian can be written as 
\[ H=L_{\omega}+R+F=L_{v_1}+S_1+R+F \]
and this is nothing but Proposition~\ref{propperiodic}, which can be applied with $\mu_1$ instead of $\mu$, because~\eqref{seuillin} implies that $\eta=\max\{\mu_1,\rho\}=\mu_1$ and~\eqref{seuil1}: indeed, since $T_1 \MP \Psi(Q)$ and $T_1\mu_1 \EP Q^{-1}$, the inequalities~\eqref{seuillin} imply
\[ \rho \leq \mu_1, \quad \nu \PM s\eta, \quad T\eta \PM s \]
which gives~\eqref{seuil1} for a proper choice of implicit constants. So now assume that the statement holds true for some $1 \leq j \leq d-1$, and let us prove that it is true for $2 \leq j+1 \leq d$. By the inductive assumption, there exists a $(M,s_j)$-ultra-differentiable symplectic map 
$$\Phi_j : \T^n \times D_{\delta_j} \longrightarrow \T^n \times D$$ 
such that 
\[ H\circ\Phi_j=L_\omega+R+\bar{F}_j+\hat{F}_j, \]
with $\{\bar{F}_j,L_{v_i}\}=0$ for any $1 \leq i \leq j$, and with the estimates
\begin{equation}
\begin{cases}
|\Phi_j-\mathrm{Id}|_{s_j} \MP \Psi(Q)\nu, \quad |\bar{F}_j|_{s_j} \MP \nu, \\
|\hat{F}_j|_{s_j} \MP \nu \exp(-\kappa_{2^{1/d}} \left(C^{-1}(\lambda sQ)\right)^{-1}).  
\end{cases}
\end{equation}
Let us consider the Hamiltonian
\[ \bar{H}_j:=H\circ\Phi_j-\hat{F}_j=L_\omega+R+\bar{F}_j  \]
which is $(M,s_j)$-ultra-differentiable on the domain $\T^n \times D_{\delta_j}$. It can also be written as 
\[ \bar{H}_j=L_{v_{j+1}}+S_{j+1}+R+\bar{F}_j  \]
and because of~\eqref{seuillin}, we can still apply Proposition~\ref{propperiodic} to this Hamiltonian with $\mu_{j+1}$ and a constant times $\nu$ instead of respectively $\mu$ and $\nu$ (and as before, $\eta=\mu_{j+1}$). Therefore, since $s_j/\xi=s_{j+1}$ and $\delta_j/\xi=\delta_{j+1}$, we can find a $(M,s_{j+1})$-ultra-differentiable symplectic transformation
\[ \bar{\Phi} : \T^n \times D_{\delta_{j+1}} \rightarrow \T^n \times D_{\delta_j} \]
such that
\[ \bar{H}_j\circ\bar{\Phi} =L_\omega+R+\bar{\bar{F}}_j+\widehat{\bar{F}_j} \]
with $\{\bar{\bar{F}}_j,L_{v_{j+1}}\}=0$ and the estimates
\begin{equation*}
|\bar{\Phi}-\mathrm{Id}|_{s_{j+1}}\MP T_{j+1}\nu, \quad |\bar{\bar{F}}_j|_{s_{j+1}} \MP \nu, \quad |\widehat{\bar{F}_j}|_{s_{j+1}} \MP \nu\exp(-\kappa_{2^{1/d}} \left(C^{-1}(\lambda sQ)\right)^{-1}). 
\end{equation*}
Moreover, for $1 \leq i \leq j$, since $L_{v_i}$ are integrable Hamiltonians for which $\{L_{v_i},\bar{F}_j\}=0$, then $\{L_{v_{i}},\bar{\bar{F}}_j\}=0$ and thus $\{\bar{\bar{F}}_j,L_{v_{i}}\}=0$ for $1 \leq i \leq j+1$. 
We may set
\[ \Phi_{j+1}:=\Phi_j \circ \bar{\Phi}, \quad \bar{F}_{j+1}:=\bar{\bar{F}}_j, \quad \hat{F}_{j+1}:=\widehat{\bar{F}_j}+\hat{F}_j \circ \bar{\Phi} \]
so that
\[ H \circ \Phi_{j+1}=L_\omega+R+\bar{F}_{j+1}+\hat{F}_{j+1}. \]
We already explained that $\{\bar{F}_{j+1},L_{v_{i}}\}=0$ for $1 \leq i \leq j+1$, and
\[ |\bar{F}_{j+1}|_{s_{j+1}} \MP \nu. \]
It is obvious that
\[ \Phi_{j+1} : \T^n \times D_{\delta_{j+1}} \rightarrow \T^n \times D_{\delta} \]
and using Proposition~\ref{composition} and the fact that $T_{j+1} \MP \Psi(Q)$, one has  
\[ |\Phi_{j+1}-\mathrm{Id}|_{s_{j+1}} \leq |\Phi_{j+1}-\mathrm{Id}|_{s_{j}}+|\bar{\Phi}-\mathrm{Id}|_{s_{j+1}}\MP \Psi(Q)\nu. \]
To conclude, using once again Proposition~\ref{composition} we have
\[ |\hat{F}_{j+1}|_{s_{j+1}} \leq |\widehat{\bar{F}_j}|_{s_{j+1}}+ |\hat{F}_{j}|_{s_{j}} \MP \nu \exp(-\kappa_{2^{1/d}} \left(C^{-1}(\lambda sQ)\right)^{-1}).   \] 
This proves that the statement holds true for $j+1$, which finished the induction and the proof.
\end{proof}  

Let us conclude by giving the proof of Corollary~\ref{corlineaire}.

\begin{proof}[Proof of Corollary~\ref{corlineaire}]
Since we are under the assumptions of Theorem~\ref{lineaire}, let us first consider the Hamiltonian in normal form 
\[ H\circ\Phi =h+\bar{f}+\hat{f} \]
where 
\[ \Phi : \T^n \times D_{1/2} \rightarrow \T^n \times D, \] 
$\{\bar{f},L_{\omega}\}=0$ and with the estimates
\begin{equation*}
\begin{cases}
|\Phi-\mathrm{Id}|_{s/2}\leq c_1\Psi(\Delta^{-1}(c_2s\varepsilon^{-1}))\varepsilon, \quad |\bar{f}|_{s/2} \leq c_1\varepsilon, \\
|\hat{f}|_{s/2} \leq c_1\varepsilon\exp(-c_3\left(C^{-1}(c_4s\Delta^{-1}(c_2s\varepsilon^{-1}))\right)^{-1}).
\end{cases}
\end{equation*}
Let $(\tilde{\theta}(t),\tilde{I}(t))$ be a solution of $H \circ \Phi$ with $\tilde{I}(0) \in D_{1/4}$. Since $\{\bar{f},L_{\omega}\}=0$ and $\omega=(\bar{\omega},0) \in \R^d \times \R^{n-d}$ with $\bar{\omega}$ non-resonant, we have $\partial_{\theta_j}\bar{f}=0$ for $1 \leq j \leq d$ and thus 
\[ |\Pi_d (\tilde{I}(t)-\tilde{I}(0))| \leq |t|\sup_{(\tilde{\theta},\tilde{I})\in \T^n \times D_{1/2}}|\partial_\theta \hat{f}(\tilde{\theta},\tilde{I})|  \]
for all $t$ as long as $\tilde{I}(t) \in D_{1/4}$. By Corollary~\ref{corderivative} and the above estimate on $\hat{f}$, we have
\[ \sup_{(\tilde{\theta},\tilde{I})\in \T^n \times D_{1/2}}|\partial_\theta \hat{f}(\tilde{\theta},\tilde{I})| \MP s^{-1}\varepsilon\exp(-c_3\left(C^{-1}(c_4s\Delta^{-1}(c_2s\varepsilon^{-1}))\right)^{-1}) \]
hence if we define
\[ T \PE rs\varepsilon^{-1}\exp(c_3\left(C^{-1}(c_4s\Delta^{-1}(c_2s\varepsilon^{-1}))\right)^{-1}) \]
we obtain
\[ |\Pi_d (\tilde{I}(t)-\tilde{I}(0))| \leq r, \quad |t|\leq T \]
since $r \leq 1/4$ and as long as $\tilde{I}(t) \in D_{1/2}$. Now coming back to the original Hamiltonian $H$, any solution $(\theta(t),I(t))$ with $I(0) \in D_{1/8}$ gives rise to a solution $(\tilde{\theta}(t),\tilde{I}(t))=\Phi^{-1}(\theta(t),I(t))$ of $H \circ \Phi$ with $\tilde{I}(0) \in 1/4$, and therefore, for $|t| \leq T$ and such that $I(t) \in D_{1/4}$, we obtain
\begin{eqnarray*}
|\Pi_d (I(t)-I(0))| & \leq & |\Pi_d (I(t)-\tilde{I}(t))|+ |\Pi_d (\tilde{I}(t)-\tilde{I}(0))| + |\Pi_d (\tilde{I}(0)-I(0))| \\
& \leq & r + 2c_1\Psi(\Delta^{-1}(c_2s\varepsilon^{-1}))\varepsilon \\
& \leq & 2r
\end{eqnarray*}
since we assumed $r \geq 2c_1\Psi(\Delta^{-1}(c_2s\varepsilon^{-1}))\varepsilon$. This concludes the proof.
\end{proof}

\subsection{Diffusion in the linear case: proof of Theorem~\ref{difflineaire}}\label{sec:difflinear}

In this section we shall give the proof of Theorem~\ref{difflineaire}.

\begin{proof}
For convenience, let us write 
\[ \omega=(\bar{\omega},0)=(1,\bar{\omega}_1, \dots,\bar{\omega}_{d-1},0) \in \R^d \times \R^{n-d} \] 
with $|\bar{\omega}_i| \leq 1$ for $1 \leq i \leq d-1$. Let us denote by $(p_j/q_j)_{j\in\N}$ the sequence of the convergents of $\bar{\omega}_1$ for instance. We have the classical inequalities
\begin{equation*}
(q_j+q_{j+1})^{-1}<|q_j\bar{\omega}_1-p_j|<q_{j+1}^{-1}, \quad j\in\N, 
\end{equation*}
and since $q_{j+1}>q_j$, this gives
\begin{equation}\label{est1}
(2q_{j+1})^{-1}<|q_j\bar{\omega}_1-p_j|<q_{j+1}^{-1}, \quad j\in\N. 
\end{equation}
Now by definition of $\Psi_\omega$, we obtain
\begin{equation}\label{est2}
q_{j+1}<\Psi_\omega(q_j)<2q_{j+1}.
\end{equation}
The perturbation $f_j$ will be of the form
\[ f_j(\theta,I)=f_j^1(I)+f_j^2(\theta), \quad (\theta,I)\in \T^n \times \R^n. \]
First, we choose $f_j^1(I)=v_j\cdot I-\omega \cdot I$, where $v_j=(1,p_j/q_j,\bar{\omega}_2,\dots,\bar{\omega}_{d-1},0)$. We set
\[ \varepsilon_j=|\bar{\omega}_1-p_j/q_j| \]
and observe that $\varepsilon_j$ tends to zero as $j$ tends to infinity. From the inequalities~\eqref{est1} and \eqref{est2}, recalling the definitions of $\Delta_\omega$ and $\Delta^*_\omega$, we have 
\begin{equation}\label{est4}
(2\Delta_\omega(q_j))^{-1} \leq \varepsilon_j \leq 2(\Delta_\omega(q_j))^{-1}, \quad \Delta_\omega^*((2\varepsilon_j)^{-1}) \leq q_j \leq \Delta_\omega^*(2\varepsilon_j^{-1}).
\end{equation}
Now by definition, $|\partial_I f_j^1|_{s}= c\varepsilon_j$. Then, if we let $k_j=(p_j,-q_j,0,\dots,0)\in \Z^n$, we define $f_j^2(\theta)=\varepsilon_j\mu_j \sin(2\pi k_j\cdot\theta)$
with $\mu_j$ to be chosen. Note that 
\begin{equation}\label{est5}
q_j \leq |k_j| \leq 2 q_j
\end{equation} 
since $|q_j|\geq|p_j|$ (as $|\bar{\omega}_1|\leq 1$). It is easy to estimate
\[ |\partial_\theta f_j^2|_{s} \leq c \varepsilon_j\mu_j2\pi|k_j|\exp\left(\Omega(8\pi |k_j| s)\right)  \]
and so we choose 
\[ \mu_j=(2\pi |k_j|)^{-1}\exp\left(-\Omega(8\pi |k_j| s)\right) \]
so that $|\partial_\theta f_j^2|_{s} \leq c\varepsilon_j$. Finally, we set $f_j=f^1_j+f^2_j$ and we have 
\[ |X_{f_j}|_s \leq |\partial_I f_j^1|_{s}+|\partial_\theta f_j^2|_{s} \leq 2c\varepsilon_j.  \]
Now we can write the Hamiltonian
\begin{eqnarray*}
H_j(\theta,I) & = & L_\omega(I)+f_j(\theta,I) \\
& = & \omega\cdot I+v_j\cdot I-\omega \cdot I+\varepsilon_j\mu_j \sin(2\pi k_j\cdot\theta) \\
& = & v_j\cdot I+\varepsilon_j\mu_j \sin(2\pi k_j\cdot\theta) 
\end{eqnarray*}
and as $k_j \cdot v_j=0$, the associated system is easily integrated:
\begin{equation*} 
\left\{  
\begin{array}{ccl}
\theta(t) & = & \theta_0+tv_j \quad [\Z^n]\\
 I(t) &= &I_0-t2\pi k_j\varepsilon_j\mu_j\cos(2\pi k_j.\theta_0).
\end{array}
\right. 
\end{equation*}
Choosing any solution $(\theta(t),I(t))$ with an initial condition $(\theta(0),I(0))$ for which $k_j\cdot \theta(0)$ is an integer, using the definition of $\mu_j$, we obtain
\[ |I(t)-I(0)|=|I_1(t)-I_1(0)|+|I_2(t)-I_2(0)|=|t|\varepsilon_j\exp\left(-\Omega(8\pi|k_j| s)\right). \]
Using~\eqref{est5} and~\eqref{est4} this gives
\[ |t|\varepsilon_j\exp\left(-\Omega(16\pi s\Delta_\omega^*(2\varepsilon_j^{-1}))\right) \leq |I(t)-I(0)| \leq |t|\varepsilon_j\exp\left(-\Omega(8\pi s\Delta_\omega^*((2\varepsilon_j)^{-1}))\right)   \]
which is what we wanted to prove, as $|I(t)-I(0)|=|\Pi_d(I(t)-I(0))|$.
\end{proof} 

\subsection{Stability in the non-linear case: proof of Theorem~\ref{nonlineaire}}\label{sec:nonlineaire}

The goal here is to prove Theorem~\ref{nonlineaire} and Corollary~\ref{cornonlineaire}; implicit constants will depend only on $n$, $\omega$, $|h|_s$ and the function $C$.

As before, Theorem~\ref{nonlineaire} will be an easy consequence of the following proposition, where we recall that given $\rho>0$, we let
\[ \sigma=\sigma_\rho : (\theta,I) \longmapsto (\theta,\rho I). \]

\begin{proposition}\label{propnonlineaire}
Let $H$ be as in~\eqref{Ham1}, where $H$ is $(M,s)$-ultra-differentiable with $M$ satisfying \eqref{H1} and \eqref{H2} and $\omega=(\bar{\omega},0) \in \R^d \times \R^{n-d}$ with $\bar{\omega}$ non-resonant. For $Q \geq n+2$ and $\sqrt{\varepsilon} \leq \rho \leq 1$, assume that
\begin{equation}\label{seuilnonlin}
\rho Q\Psi(Q)s^{-1} \PM 1, \quad \lambda s Q \SP 1, 
\end{equation}
for some $\lambda \PE 1$. Then there exists a $(M,s/4)$-ultra-differentiable symplectic transformation
\[ \Phi : \T^n \times D_{\rho/2} \rightarrow \T^n \times D_{\rho} \]
such that
\[ H\circ\Phi =h+\bar{f}+\hat{f} \]
where $\{\bar{f},L_\omega\}=0$ and with the estimates
\begin{equation*}
|\sigma^{-1} \circ \Phi \circ \sigma-\mathrm{Id}|_{s/2}\MP \Psi(Q)\varepsilon, \quad |\bar{f} \circ \sigma|_{s/2} \MP \varepsilon, \quad |\hat{f} \circ \sigma|_{s/2} \MP \varepsilon\exp\left(-\kappa_{2^{1/d}}\left(C^{-1}(\lambda sQ)\right)^{-1}\right). 
\end{equation*}
\end{proposition}

\begin{proof}[Proof of Theorem~\ref{nonlineaire}]
It suffices to choose
\[ Q:=\Delta^{-1}\left(c_2s\rho^{-1}\right), \quad c_2 \EP 1 \]
and to verify that choosing $\rho$ sufficiently small allows indeed to apply Proposition~\ref{proplineaire}.
\end{proof} 

\begin{proof}[Proof of Proposition~\ref{propnonlineaire}]
To analyze our Hamiltonian $H$ in the domain $\T^n \times D_\rho$, which is a neighborhood of size $\rho$ around the origin in action space, we rescale the action variables using the map
\[ \sigma=\sigma_\rho : (\theta,I) \longmapsto (\theta,\rho I) \]
which sends the domain $\T^n \times D$ onto $\T^n \times D_\rho$, the latter being included in $\T^n \times D$ since $\rho \leq 1$. Let 
\[ H'=\rho^{-1}(H\circ\sigma)\] 
be the rescaled Hamiltonian, so $H'$ is defined on $\T^n \times D$ and reads
\[ H'(\theta,I)=\rho^{-1}H(\theta,\rho I)=\rho^{-1}h(\rho I)+\rho^{-1}f(\theta,\rho I), \quad (\theta,I)\in \T^n \times D. \]
Without loss of generality, we assume that $h(0)=0$. Now, using Taylor's formula, we can expand $h$ around the origin to obtain
\begin{equation*}
h(\rho I)  = \rho\omega \cdot I+\rho^2\int_{0}^{1}(1-t)\nabla^2 h(t\rho I)(I,I) dt=\rho\omega\cdot I+\rho^2 h'(I) 
\end{equation*}
and so we can write
\begin{equation*}
\begin{cases}
H'=L_\omega+R+F, \\
R:=\rho h', \quad F:=\rho^{-1}(f\circ\sigma). 
\end{cases}
\end{equation*}
Now we know that $|f|_{s}\leq\varepsilon$, so that 
\[ |\rho^{-1}(f\circ\sigma)|_{s}\leq \rho^{-1}\varepsilon.\] Moreover, applying Corollary~\ref{corderivative} we have the estimate 
\[ |\nabla h'|_{s/2} \MP s^{-1} |h|_{s/2} \MP s^{-1}   \]
so eventually our Hamiltonian $H'$ reads
\begin{equation*}
\begin{cases}
H'=L_\omega+R+F, \\
R:=\rho h', \quad F:=\rho^{-1}(f\circ\sigma) \\
|R|_{s/2} \MP s^{-1}\rho, \quad |F|_{s/2} \leq \rho^{-1}\varepsilon. 
\end{cases}
\end{equation*}
It has exactly the form~\eqref{HamNNN} with $s/2$ instead of $s$, a constant times $s^{-1}\rho$ instead of $\rho$ and $\rho^{-1}\varepsilon$ instead of $\nu$; using the assumption that $\sqrt{\varepsilon} \leq \rho$, we have
\[ \nu=\rho^{-1}\varepsilon \leq \rho \]
and~\eqref{seuilnonlin} implies~\eqref{seuillin}. It follows that Proposition~\ref{proplineaire} can be applied, and it gives the existence of a $(M,s/4)$-ultra-differentiable symplectic transformation
\[ \Phi' : \T^n \times D_{1/2} \rightarrow \T^n \times D \]
such that
\[ H'\circ\Phi' =L_\omega+R+\bar{F}+\hat{F} \]
where $\{\bar{F},L_\omega\}=0$ and with the estimates
\begin{equation}\label{estim}
|\Phi'-\mathrm{Id}|_{s/4}\MP \Psi(Q)\rho^{-1}\varepsilon, \quad |\bar{F}|_{s/4} \MP \rho^{-1}\varepsilon, \quad |\hat{F}|_{s/4} \MP \rho^{-1}\varepsilon\exp(-\kappa_{2^{1/d}}\left(C^{-1}(\lambda sQ)\right)^{-1}). 
\end{equation}
Now, scaling back to our original coordinates, we define 
\[ \Phi=\sigma \circ \Phi' \circ \sigma^{-1} : \T^n \times D_{\rho/2} \rightarrow \T^n \times D_{\rho} \] 
so that 
\begin{eqnarray*}
H\circ\Phi & = & \rho H'\circ\Phi' \circ \sigma^{-1} \\
& = & \rho (L_\omega+ R +\bar{F}+\hat{F}) \circ \sigma^{-1} \\
& = & (\rho L_\omega+\rho^2 h') \circ \sigma^{-1} + \rho \bar{F}\circ \sigma^{-1} +\rho\hat{F}\circ \sigma^{-1}  
\end{eqnarray*}
where we used the definition of $R$. Observe that $(\rho L_\omega+\rho^2 h') \circ \sigma^{-1}=h$, so we may set
\[ \bar{f}=\rho\bar{F}\circ \sigma^{-1}, \quad \hat{f}=\rho\hat{F}\circ \sigma^{-1}, \]
and write
\[ H\circ\Phi=h+\bar{f}+\hat{f}. \]
Since $\{\bar{F},L_\omega\}=0$, $\{\bar{F} \circ \sigma^{-1},L_\omega\}=0$ and thus $\{\bar{f},L_\omega\}=0$. To conclude, observe that the estimates~\eqref{estim} can be written again as
\begin{equation*}
\begin{cases}
|\sigma^{-1} \circ \Phi \circ \sigma-\mathrm{Id}|_{s/4}\MP \Psi(Q)\rho^{-1}\varepsilon, \quad |\bar{f} \circ \sigma|_{s/4} \MP \varepsilon,\\
|\hat{f} \circ \sigma|_{s/4} \MP \varepsilon\exp\left(-\kappa_{2^{1/d}}\left(C^{-1}(\lambda sQ)\right)^{-1}\right) 
\end{cases}
\end{equation*}
which is exactly what we wanted to prove.
\end{proof}

Let us conclude by giving the proof of Corollary~\ref{cornonlineaire}, which is just a slight variation on the proof of Corollary~\ref{corlineaire}.

\begin{proof}[Proof of Corollary~\ref{cornonlineaire}]
Since we are under the assumptions of Theorem~\ref{nonlineaire}, let us first consider the Hamiltonian in normal form 
\[ H\circ\Phi =h+\bar{f}+\hat{f} \]
where 
\[ \Phi : \T^n \times D_{\rho/2} \rightarrow \T^n \times D_{\rho}, \] 
$\{\bar{f},L_{\omega}\}=0$ and with the estimates
\begin{equation*}
\begin{cases}
|\sigma^{-1} \circ \Phi \circ \sigma-\mathrm{Id}|_{s/4}\leq c_1\Psi(\Delta^{-1}(c_2s\varepsilon^{-1}))\rho^{-1}\varepsilon, \quad |\bar{f} \circ \sigma|_{s/4} \leq c_1\varepsilon, \\
|\hat{f} \circ \sigma|_{s/4} \leq c_1\varepsilon\exp(-c_3\left(C^{-1}(c_4s\Delta^{-1}(c_2s\rho^{-1}))\right)^{-1}).
\end{cases}
\end{equation*}
Decomposing $\Phi=(\Phi_\theta,\Phi_I)$ into its angle and action space components, and observing that $\partial_\theta (\hat{f} \circ \sigma)=(\partial_\theta\hat{f}) \circ \sigma$, one obtains from these last estimates and Corollary~\ref{corderivative} that
\[ |\Phi_I-\mathrm{Id}_I|_{s/4}\leq c_1\Psi(\Delta^{-1}(c_2s\varepsilon^{-1}))\varepsilon \]
and
\[ \sup_{(\tilde{\theta},\tilde{I})\in \T^n \times D_{\rho/2}}|\partial_\theta \hat{f}(\tilde{\theta},\tilde{I})| \MP s^{-1}\varepsilon\exp(-c_3\left(C^{-1}(c_4s\Delta^{-1}(c_2s\rho^{-1}))\right)^{-1}). \]
Let $(\tilde{\theta}(t),\tilde{I}(t))$ be a solution of $H \circ \Phi$ with $\tilde{I}(0) \in D_{\rho/4}$. Since $\{\bar{f},L_{\omega}\}=0$ and $\omega=(\bar{\omega},0) \in \R^d \times \R^{n-d}$ with $\bar{\omega}$ non-resonant, we have $\partial_{\theta_j}\bar{f}=0$ for $1 \leq j \leq d$ and thus 
\[ |\Pi_d (\tilde{I}(t)-\tilde{I}(0))| \leq |t|\sup_{(\tilde{\theta},\tilde{I})\in \T^n \times D_{\rho/2}}|\partial_\theta \hat{f}(\tilde{\theta},\tilde{I})|  \]
for all $t$ as long as $\tilde{I}(t) \in D_{\rho/4}$. So if we define
\[ T \PE rs\varepsilon^{-1}\exp(c_3\left(C^{-1}(c_4s\Delta^{-1}(c_2s\varepsilon^{-1}))\right)^{-1}) \]
we obtain
\[ |\Pi_d (\tilde{I}(t)-\tilde{I}(0))| \leq r, \quad |t|\leq T \]
since $r \leq \rho/4$ as long as $\tilde{I}(t) \in D_{\rho/2}$. Now coming back to the original Hamiltonian $H$, any solution $(\theta(t),I(t))$ with $I(0) \in D_{\rho/8}$ gives rise to a solution $(\tilde{\theta}(t),\tilde{I}(t))=\Phi^{-1}(\theta(t),I(t))$ of $H \circ \Phi$ with $\tilde{I}(0) \in D_{\rho/4}$, and therefore, for $|t| \leq T$ and such that $I(t) \in D_{\rho/4}$, we obtain
\begin{eqnarray*}
|\Pi_d (I(t)-I(0))| & \leq & |\Pi_d (I(t)-\tilde{I}(t))|+ |\Pi_d (\tilde{I}(t)-\tilde{I}(0))| + |\Pi_d (\tilde{I}(0)-I(0))| \\
& \leq & r + 2c_1\Psi(\Delta^{-1}(c_2s\varepsilon^{-1}))\varepsilon \\
& \leq & 2r
\end{eqnarray*}
since we assumed $r \geq 2c_1\Psi(\Delta^{-1}(c_2s\varepsilon^{-1}))\varepsilon$. This concludes the proof.
\end{proof}

\subsection{Stability in the quasi-convex case: proof of Theorem~\ref{convex} }\label{sec:convex}

The aim of this section is to prove Theorem~\ref{convex}; implicit constants will depend only on $n$, $|h|_s$ and the function $C$ (except at the very end of the proof where they will also depend on $l$ and $m$).  

To prove Theorem~\ref{convex}, we will have to establish several intermediate statements. The first one is a normal form for an arbitrary non-linear Hamiltonian in a neighborhood of a periodic frequency. In order not to assume that $h$ is a local diffeomorphism, the following definition will be useful in the sequel. For a fixed integrable Hamiltonian $h$ and for $\mu>0$, a point $I_1 \in D$ in action space is said to be $\mu$-close to be $T$-periodic if there exists a $T$-periodic vector $v \in \R^n\setminus\{0\}$ such that
\begin{equation}\label{closeperiodic}
|\nabla h(I_1)-v| \leq \mu.
\end{equation}
We will use the notation $D_r(I_1)$ to denote the ball of radius $r>0$ centered at the point $I_1$. We have the following proposition, where we define 
\[ \tau_1(\theta,I):=(\theta,I-I_1) \]
and we recall that, given $\rho>0$, 
\[ \sigma(\theta,I)=\sigma_\rho(\theta,I)=(\theta,\rho I). \]

\begin{proposition}\label{propconvex1}
Let $H$ be as in~\eqref{Ham1}, where $H$ is $(M,s)$-ultra-differentiable with $M$ satisfying \eqref{H1} and \eqref{H2}. Let $I_1 \in D_{3/4}$ be $\mu$-close to be $T$-periodic, and given $\rho>0$, assume that  
\begin{equation}\label{seuilconvex1}
\mu s \leq \rho, \quad \varepsilon \PM \rho^2, \quad T\rho \PM s^2, \quad \rho <1/4. 
\end{equation}
Then there exists a $(M,s/4)$-ultra-differentiable symplectic transformation
\[ \Phi : \T^n \times D_{\rho/2}(I_1) \rightarrow \T^n \times D_{\rho}(I_1) \]
such that
\[ H\circ\Phi =h+\bar{f}+\hat{f} \]
where $\{\bar{f},L_v\}=0$ and with the estimates
\begin{equation*}
\begin{cases}
|\tau_1^{-1} \circ \sigma^{-1} \circ \Phi \circ \sigma \circ \tau_1-\mathrm{Id}|_{s/4}\leq 2 T\rho^{-1}\varepsilon, \quad |\bar{f} \circ \sigma \circ \tau_1|_{s/4} \leq 2\varepsilon, \\ 
|\hat{f} \circ \sigma \circ \tau_1|_{s/4} \leq \varepsilon\exp\left(-\kappa_2\left(C^{-1}(s^2(AT\rho)^{-1})\right)^{-1}\right). 
\end{cases}
\end{equation*}
\end{proposition}

\begin{proof}
Observe that since $\rho<1/4$ and $I_1 \in D_{3/4}$, then $D_{\rho}(I_1)$ is included in $D$. Without loss of generality, we may assume that $I_1=0$ so that $\tau_1$ is the identity, and we may also assume that $h(0)=0$. Let 
\[ H'=\rho^{-1}(H\circ\sigma)\] 
be the rescaled Hamiltonian, so $H'$ is defined on $\T^n \times D$ and reads
\[ H'(\theta,I)=\rho^{-1}H(\theta,\rho I)=\rho^{-1}h(\rho I)+\rho^{-1}f(\theta,\rho I), \quad (\theta,I)\in \T^n \times D. \]
Using Taylor's formula, we can expand $h$ around the origin to obtain
\begin{equation*}
h(\rho I)  = \rho\omega \cdot I+\rho^2\int_{0}^{1}(1-t)\nabla^2 h(t\rho I)(I,I) dt=\rho\omega\cdot I+\rho^2 h'(I) 
\end{equation*}
which can be written again as
\[ h(\rho I)=\rho v\cdot I+\rho(\omega-v)\cdot I+\rho^2 h'(I)  \]
and therefore
\begin{equation*}
\begin{cases}
H'=L_v+S+R+F, \\
S:=L_\omega-L_v, \quad R:=\rho h', \quad F:=\rho^{-1}(f\circ\sigma). 
\end{cases}
\end{equation*}
From the assumptions that $I_1$ is $\mu$-close to be $T$-periodic we know that 
\[ |\nabla S|_s \leq c\mu. \]
Then as $|f|_{s}\leq\varepsilon$, we have 
\[ |\rho^{-1}(f\circ\sigma)|_{s}\leq \rho^{-1}\varepsilon\] 
and moreover, applying Corollary~\ref{corderivative}, we have the estimate 
\[ |\nabla h'|_{s/2} \MP s^{-1}.   \]
Eventually our Hamiltonian $H'$ reads
\begin{equation*}
\begin{cases}
H'=L_\omega+S+R+F, \\
R:=\rho h', \quad S:=L_\omega-L_v, \quad F:=\rho^{-1}(f\circ\sigma) \\
|\nabla S|_{s/2} \leq c \mu, \quad |R|_{s/2} \MP s^{-1}\rho, \quad |F|_{s/2} \leq \rho^{-1}\varepsilon. 
\end{cases}
\end{equation*}
It has exactly the form~\eqref{HamN} with $s/2$ instead of $s$, $c\mu$ instead of $\mu$, a constant times $s^{-1}\rho$ instead of $\rho$ and $\rho^{-1}\varepsilon$ instead of $\nu$; one easily check that~\eqref{seuilconvex1} implies $\eta \EP \max\{\mu,s^{-1}\rho\} \EP s^{-1}\rho$ and~\eqref{seuil1}. It follows that Proposition~\ref{propperiodic} can be applied, and it gives, choosing $\delta=1$ and $\xi=2$, the existence of a $(M,s/4)$-ultra-differentiable symplectic transformation
\[ \Phi' : \T^n \times D_{1/2} \rightarrow \T^n \times D \]
such that
\[ H\circ\Phi' =L_v+S+\bar{F}+\hat{F} \]
with $\{\bar{F},L_v\}=0$ and the estimates
\begin{equation}\label{estim2}
\begin{cases}
|\Phi'-\mathrm{Id}|_{s/4}\leq 2 T\rho^{-1}\varepsilon, \quad |\bar{F}|_{s/4} \leq 2\rho^{-1}\varepsilon, \\ 
|\hat{F}|_{s/4} \leq \rho^{-1}\varepsilon\exp\left(-\kappa_2\left(C^{-1}(s^2(AT\rho)^{-1})\right)^{-1}\right). 
\end{cases}
\end{equation}
Now, scaling back to our original coordinates, we define 
\[ \Phi=\sigma \circ \Phi' \circ \sigma^{-1} : \T^n \times D_{\rho/2} \rightarrow \T^n \times D_{\rho} \] 
so that 
\begin{eqnarray*}
H\circ\Phi & = & \rho H'\circ\Phi' \circ \sigma^{-1} \\
& = & \rho (L_\omega+ R +\bar{F}+\hat{F}) \circ \sigma^{-1} \\
& = & (\rho L_\omega+\rho^2 h') \circ \sigma^{-1} + \rho \bar{F}\circ \sigma^{-1} +\rho\hat{F}\circ \sigma^{-1}  
\end{eqnarray*}
where we used the definition of $R$. Observe that $(\rho L_\omega+\rho^2 h') \circ \sigma^{-1}=h$, so we may set
\[ \bar{f}=\rho\bar{F}\circ \sigma^{-1}, \quad \hat{f}=\rho\hat{F}\circ \sigma^{-1}, \]
and write
\[ H\circ\Phi=h+\bar{f}+\hat{f}. \]
Since $\{\bar{F},L_\omega\}=0$, $\{\bar{F} \circ \sigma^{-1},L_\omega\}=0$ and thus $\{\bar{f},L_\omega\}=0$. To conclude, observe that the estimates~\eqref{estim2} can be written again as
\begin{equation*}
\begin{cases}
|\sigma^{-1} \circ \Phi \circ \sigma-\mathrm{Id}|_{s/4}\leq 2 T\rho^{-1}\varepsilon, \quad |\bar{f} \circ \sigma|_{s/4} \leq 2 \varepsilon,\\
|\hat{f} \circ \sigma|_{s/4} \leq \varepsilon\exp\left(-\kappa_2\left(C^{-1}(s^2(AT\rho)^{-1})\right)^{-1}\right)
\end{cases}
\end{equation*}
which is exactly what we wanted to prove.
\end{proof}

From the last proposition, it is easy to prove that the action variables of any solution starting close to $I_1$ have small variation for a long interval of time in the direction given by the periodic vector $v$. Using the quasi-convexity assumption on $h$, we will prove that the action variables have also small variation for a long interval of time in the directions transverse to $v$ but tangent to the energy level. Let us recall  that $h : D \rightarrow \R$ is $(l,m)$-quasi-convex, for positive constants $l$ and $m$, if for all $I \in D$, it satisfies 
\begin{equation*}
|\nabla h(I)|\geq l 
\end{equation*}
and at least of the inequalities
\begin{equation*}
|\nabla h(I)\cdot \xi|\geq l|\xi|, \quad |\nabla^2 h(I)\xi\cdot \xi| \geq m|\xi|^2
\end{equation*}
holds true for all $\xi \in \R^n$.

\begin{proposition}\label{propconvex2}
Under the assumptions of Proposition~\ref{propconvex1} and if
\begin{equation}\label{seuilconvex2}
\mu \PM m\rho, \quad \rho \PM l, \quad \varepsilon \PM m\rho^2,
\end{equation}
for any $I_0 \in D_{\mu}(I_1)$ and any solution $(\theta(t),I(t))$ of $H$ with $I(0)=I_0$, we have 
\[ |I(t)-I_0| \leq \rho, \quad |t| \PM  m\rho^2s\varepsilon^{-1}\exp(\kappa_2\left(C^{-1}(s^2(AT\rho)^{-1})\right)^{-1})\]
where $A  \EP 1$.
\end{proposition}

The proof below follows~\cite{Pos93}.

\begin{proof}
Let us define
\[ T_*\PE m\rho^2s\varepsilon^{-1}\exp(\kappa_2\left(C^{-1}(s(AT\rho)^{-1})\right)^{-1}).  \]
For any solution $(\tilde{\theta}(t),\tilde{I}(t))$ of $H \circ \Phi$ with $\tilde{I}(0) \in D_{2\mu}(I_1)$, let $T_e \in ]0,+\infty]$ be the time of first exit of the ball centered at $I_0$ of radius $\rho/4$; observe that $\mu \leq \rho/8$ can be arranged by the first inequality in~\eqref{seuilconvex2} so this is well-defined. We claim it is sufficient to prove that $T_* \leq T_e$. Indeed, in this case one would have
\[ |\tilde{I}(t)-\tilde{I}(0)| \leq \rho/4, \quad 0 \leq t \leq T_* \]
and as the image of $D_{2\mu}(I_1)$ (respectively $D_{\rho/2}(I_1)$) under $\Phi$ contains $D_{\mu}(I_1)$ (respectively is contained in $D_{\rho}(I_1)$), one would obtain that 
\[ |I(t)-I(0)| \leq \rho, \quad 0 \leq t \leq T_* \]
for any solution $(\theta(t),I(t))$ of $H$ with $I(0) \in D_{\mu}(I_1)$.

So it remains to prove the claim, and we will do this by contradiction assuming that $T:=\min\{T_e,T_*\}=T_e$. To simplify the notations, we remove the tilde from the notations and we consider a solution $(\theta(t),I(t))$ of $H \circ \Phi$ with $I(0) \in D_{2\mu}(I_1)$. Let us write 
\[\Delta I:=I(T)-I(0), \quad I(s):=I_0+s\Delta(I), \quad \omega(I(s)):=\nabla h(I(s)) \]
for $0 \leq s \leq 1$. Observe that since we are assuming $T=T_e$, then
\begin{equation}\label{contradiction}
|\Delta I|=\rho/4.
\end{equation}
Recalling that 
\[ |\nabla h(I_1)-v| \leq \mu \]
for some $T$-periodic $v$, we let $\Pi_v$ be the orthogonal projection onto the line generated by $v$ and $\Pi_v^\perp=\mathrm{Id}-\Pi_v$ be the orthogonal projection onto the orthogonal of $v$. Since $I(0) \in D_{2\mu}(I_1)$, we have 
\begin{equation}\label{boubou1}
|\omega(I(0))-\nabla h(I_1)| \MP \mu
\end{equation}
and using~\eqref{contradiction} and the fact that $\mu \leq \rho$, we obtain
\[ |\omega(I(s))-\nabla h(I_1)| \MP \rho.  \]
So from the last two inequalities
\begin{equation}\label{exact}
|\omega(I(s))-v| \MP \rho.
\end{equation}
Since $\Pi_v^\perp v=0$ we have
\[ |\Pi_v^\perp \omega(I(s))|=|\Pi_v^\perp(\omega(I(s)-v))| \MP \rho \]
and so
\begin{equation}\label{bout1}
|\omega(I(s))\cdot \Pi_v^\perp \Delta I| \MP |\Pi_v^\perp(\omega(I(s)-v))||\Pi_v^\perp \Delta I| \MP \rho |\Delta I|.
\end{equation}
Observe that for $s=0$, we have the better estimate
\begin{equation}\label{bout11}
|\omega(I(0))\cdot \Pi_v^\perp \Delta I| \MP \mu |\Delta I|.
\end{equation}
On the other hand, our Hamiltonian is of the form
\[ H \circ \Phi = h + \bar{f} + \hat{f}  \] 
with $\{\bar{f},L_v\}=0$, which implies that $\Pi_v \left(\partial_\theta \bar{f}\right)=0$ and thus
\[ |\omega(I(s))\cdot \Pi_v \Delta I| \leq \int_0^T |\omega(I(s))\cdot \partial_\theta \hat{f}| dt.  \]
Using Corollary~\ref{corderivative} one easily obtains an estimate of $\partial_\theta \hat{f}$ from the known estimate on $\hat{f}$, and, by the definition of $T$ with a proper implicit constant, we arrive at
\begin{equation}\label{bout2}
|\omega(I(s))\cdot \Pi_v \Delta I| < m\rho^2/96 =\left(m\rho/24\right)|\Delta I|
\end{equation}
where the last equality follows from~\eqref{contradiction}. From~\eqref{bout1} and~\eqref{bout2} we get
\[ |\omega(I(s))\cdot \Delta I| \MP \rho |\Delta I| \]
and by a proper choice of the implicit constant in the second condition of~\eqref{seuilconvex2}, we can ensure that
\begin{equation}\label{qconvex}
|\omega(I(s))\cdot \Delta I| \leq l|\Delta I|.
\end{equation}
To reach a contradiction, let $\Delta h:=h(I(T))-h(I(0))$ so that Taylor formula with integral remainder at the point $I(0)$ gives
\[ \Delta h=\omega(I(0))\cdot \Delta I+\int_0^1(1-s)\nabla^2 h(I(s))\Delta I \cdot \Delta I ds. \]
Since $h$ is $(l,m)$-convex, using~\eqref{qconvex} we know that
\[ |\nabla^2 h(I(s))\Delta I \cdot \Delta I | \geq m|\Delta I|^2 \]
and therefore
\begin{equation}\label{contra}
m|\Delta I|^2 \leq 2|\Delta h|+2|\omega(I(0))\cdot \Delta I| \leq 2|\Delta h|+2|\omega(I(0))\cdot \Pi_v^\perp \Delta I|+2|\omega(I(0))\cdot \Pi_v\Delta I|.
\end{equation}
Using the preservation of energy, the last part of~\eqref{seuilconvex2} and~\eqref{contradiction}, we have
\begin{equation}\label{contra1}
2|\Delta h|\MP \varepsilon < \frac{m}{3} \frac{\rho^2}{16}=\frac{m}{3}|\Delta I|^2.
\end{equation}
Next, from~\eqref{bout11} we obtain
\[ 2|\omega(I(0))\cdot \Pi_v^\perp \Delta I| \MP \mu |\Delta I| \leq C\frac{\mu^2}{m}+\frac{m}{6}|\Delta I|^2 \]
for some $C\EP 1$, and, using the first part of~\eqref{seuilconvex2} and~\eqref{contradiction}, we can ensure that
\begin{equation}\label{contra2}
2|\omega(I(0))\cdot \Pi_v^\perp \Delta I|<\frac{m}{6}|\Delta I|^2+\frac{m}{6}|\Delta I|^2 =\frac{m}{3}|\Delta I|^2.
\end{equation}
Finally, recall from~\eqref{bout2} that
\begin{equation}\label{contra3}
2|\omega(I(s))\cdot \Pi_v \Delta I| < m\rho^2/48 =\frac{m}{3}|\Delta I|^2.
\end{equation}
and thus~\eqref{contra}, \eqref{contra1}, \eqref{contra2} and~\eqref{contra3} yields
\[ m|\Delta I|^2 < m|\Delta I|^2 \]
which is absurd.

This proves the proposition for positive times, but for negative times the argument is completely similar so this concludes the proof.
\end{proof}

To finish the proof of Theorem~\ref{convex}, we just need the following lemma which is a direct application of Dirichlet's box principle.

\begin{lemma}\label{dirichlet}
Let $\omega \in \R^n\setminus\{0\}$, and $Q \geq 1$ a real number. Then there exists a $T$-periodic vector $v \in \R^n\setminus\{0\}$ such that
\[ |\omega-v|\leq (n-1)(TQ)^{-1}, \quad |\omega|^{-1} \leq T \leq n|\omega|^{-1}Q^{n-1}.  \]
\end{lemma}

\begin{proof}
Fix $Q \geq 1$. Up to a re-ordering of its component, we can write $\omega=|\omega|_{\infty}(\pm 1, x)$ for some $x \in \R^{n-1}$ and by Dirichlet's approximation theorem, there exists a primitive rational vector $p/q \in \Q^{n-1}$, such that
\[ |qx-p|_{\infty} \leq Q^{-1}, \quad 1 \leq q \leq Q^{n-1}. \]
The vector $v=|\omega|_{\infty}(\pm 1, p/q) \in \R^n$ is then $T$-periodic, for $T=|\omega|_{\infty}^{-1}q$, and we have
\[ |\omega-v|\leq T^{-1}|qx-p|, \quad |\omega|_{\infty}^{-1} \leq T \leq |\omega|_{\infty}^{-1}Q^{n-1}  \]
which implies
\[ |\omega-v|\leq  (n-1)(TQ)^{-1}, \quad |\omega|^{-1} \leq T \leq n|\omega|^{-1}Q^{n-1} \]
and this was the statement to prove.
\end{proof}

\begin{proof}[Proof of Theorem~\ref{convex}]
Let $Q \geq 1$ be a real number to be chosen below. Recall that we are considering $H$ as in~\eqref{Ham1} where $H$ is $(M,s)$-ultra-differentiable with $M$ satisfying \eqref{H1} and \eqref{H2} and $h$ is $(m,l)$-quasi-convex. To further simplify the exposition, implicit constants will now also depend on $l$ and $m$.

Consider any solution $(\theta(t),I(t))$ of $H$ with $I(0) \in D_{1/2}$. We can apply Lemma~\ref{dirichlet} to $\omega:=\nabla h(I(0))$ and find a $T$-periodic vector $\omega \in \R^n\setminus\{0\}$ such that
\[ |\omega-v|\MP (TQ)^{-1}, \quad 1 \PM  T \MP Q^{n-1}.  \] 
We set
\[ \mu \EP (TQ)^{-1} \]
so that $I(0)$ is $\mu$-close to be $T$-periodic. Then we set
\[ \rho \EP s\mu \EP s(TQ)^{-1} \]
and using the fact that $T \MP Q^{n-1}$,~\eqref{seuilconvex1} and~\eqref{seuilconvex2} are satisfied provided that
\begin{equation}\label{seuilconvex3}
Q^{2n}\varepsilon \PM s^2, \quad Qs \PS 1. 
\end{equation} 
Assuming~\eqref{seuilconvex3}, Proposition~\ref{propconvex1} and Proposition~\ref{propconvex2} apply, and we obtain the stability estimates
\[ |I(t)-I(0)| \leq \rho \MP sQ^{-1}, \] 
for
\[ |t| \PM  s\exp\left(\kappa_2\left(C^{-1}(aQs)\right)^{-1}\right), \quad a\PE 1.\]
To conclude, to fulfill~\eqref{seuilconvex3} we may choose
\[ Q  \PE (s^2\varepsilon^{-1})^{\frac{1}{2n}} \]
to finally reach
\[ |I(t)-I(0)| \leq \rho \MP s(s^{-2}\varepsilon)^{\frac{1}{2n}}, \] 
for
\[ |t| \PM  s\exp\left(\kappa_2\left(C^{-1}\left(a's(s^2\varepsilon^{-1})^{\frac{1}{2n}}\right)\right)^{-1}\right), \quad a'\PE 1\]
which was the statement to prove.
\end{proof}

\subsection{Diffusion in the quasi-convex case: proof of Theorem~\ref{diffconvex}}\label{sec:diffconvex}

The goal of this section is to give a proof of Theorem~\ref{diffconvex}, which is the construction of an unstable orbit assuming that we are working in a non-quasi analytic class of ultra-differentiable functions (which are characterized by the condition (H3), see \S~\ref{sec:Nekintro}).

The construction, which is due to Marco-Sauzin, is explained in details in~\cite{MS02} in the case where $M=M_\alpha$ with $\alpha>1$; we are aiming at showing that it still works in any non-quasi-analytic class of ultra-differentiable functions, and it requires only a minor modification (in the same way as Theorem~\ref{destruction} required a minor modification of~\cite{Bes00}). 

The interest is that when the sequence $M$ is matching, the
instability result is very close to the stability result given by
Theorem~\ref{convex}, and thus this shows that the latter result is
quite accurate.

From now on, we write $\A:=\T \times \R$ the infinite annulus. Given a potential function $U : \T \rightarrow \R$, we consider the following family of maps $\psi_q : \A \rightarrow \A$ defined by
\begin{equation}\label{maps1}
\psi_q(\theta,I)=\left(\theta +qI, I-q^{-1}U'(\theta + qI)\right), \quad (\theta,I)\in\A, 
\end{equation}
for $q \in \N^*$. If we require $U'(0)=-1$, for example if we choose 
\[ U(\theta)=-(2\pi)^{-1}\sin (2\pi\theta),\quad U'(\theta)=-\cos(2\pi\theta),\] 
then it is easy to see that $\psi_q(0,0)=(0,q^{-1})$ and by induction 
\begin{equation}\label{driftstand}
\psi_q^k(0,0)=(0,kq^{-1}) 
\end{equation}
for any $k \in \Z$. After $q$ iterations, the point $(0,0)$ drifts from the circle $I=0$ to the circle $I=1$ and it is bi-asymptotic to infinity, in the sense that the sequence $\left(\psi_q^k(0,0)\right)_{k \in \Z}$ is not contained in any semi-infinite annulus of $\A$. Clearly these maps are exact-symplectic, but obviously they have no invariant circles and so they cannot be ``close to integrable". However, we will use the fact that they can be written as a composition of time-one maps,
\begin{equation}\label{driftstand2}
\psi_q=\Phi^{q^{-1}U} \circ \left(\Phi^{\frac{1}{2}I^2} \circ \cdots \circ \Phi^{\frac{1}{2}I^2}\right) =\Phi^{q^{-1}U} \circ \left(\Phi^{\frac{1}{2}I^2}\right)^q, 
\end{equation}
to embed $\psi_q$ in the $q^{th}$-iterate of a near-integrable map of $\A^n$, for $n \geq 2$. To do so, we will use the following ``coupling lemma", in which the product of functions acting on separate variables is denoted by $\otimes$, \textit{i.e.}
\[ f \otimes g(x,x')=f(x)g(x') \quad  x \in \A^m , \; x' \in \A^{m'}. \] 

\begin{lemma}[Marco-Sauzin] \label{coupling}
Let $m,m' \geq 1$, $F : \A^m \rightarrow \A^m$ and $G : \A^{m'} \rightarrow \A^{m'}$ two maps, and $f : \A^m \rightarrow \R$ and $g : \A^{m'} \rightarrow \R$ two Hamiltonian functions generating complete vector fields. Suppose there is a point $a \in \A^{m'}$ which is $q$-periodic for $G$ and such that the following ``synchronisation" conditions hold:
\begin{equation}\label{sync}
g(a)=1, \quad dg(a)=0, \quad g(G^k(a))=0, \quad dg(G^k(a))=0, \tag{S}
\end{equation} 
for $1 \leq k \leq q-1$. Then the mapping
\[ \Psi=\Phi^{f \otimes g} \circ (F \times G) : \A^{m+m'} \longrightarrow \A^{m+m'} \]
is well-defined and for all $x \in \A^m$, 
\[ \Psi^q(x,a)=(\Phi^f \circ F^q(x),a). \]
\end{lemma}  

Therefore, if we set $m=1$, $F=\Phi^{\frac{1}{2}I_1^2}$ and $f=q^{-1}U$ in the coupling lemma, the $q^{th}$-iterate $\Psi^q$ will leave the submanifold $\A \times \{a\}$ invariant, and its restriction to this annulus will coincide with our ``unstable map" $\psi_q$. Hence, after $q^2$ iterations of $\Psi$, the $I_1$-component of the point $((0,0),a) \in \A^2$ will move from $0$ to $1$. 

The difficult part is then to find what kind of dynamics we can put on the second factor to apply this coupling lemma. Following \cite{MS02}, we introduce a new sequence of ``large" parameters $A_j \in \N^*$ for $j \in \N$ and in the second factor we consider a family of suitably rescaled penduli on $\A$ given by
\[ P_j(\theta_2,I_2)=\demi I_2^2 +A_j^{-2}V(\theta_2), \]
where $V(\theta)=-\cos 2\pi\theta$. We eventually define
\[ G_j=\Phi^{\frac{1}{2}(I_2^2+I_3^2+\cdots+I_{n}^2)+N_{j}^{-2}V(\theta_2)}. \]
In this case, the map $\Psi$ given by the coupling lemma is a perturbation of 
\[ \Phi^{h}=\Phi^{\frac{1}{2}(I_1^2+I_2^2+\cdots+I_{n}^2)} \] 
of the form $\Psi=\Phi^u \circ \Psi^{\tilde{h}+v}$, with 
\[ u=f \otimes g, \quad  v=A_j^{-2}V.\] 
But for the map $G_j$, due to the presence of the pendulum factor, it is possible to find a periodic orbit for $G_j$ with an irregular distribution: more precisely, a $q_j$-periodic point $a_j$ such that its distance to the rest of its orbit is of order $N_{j}^{-1}$, no matter how large $q_j$ is. 

Let us denote by $(p_j)_{j\geq 0}$ the ordered sequence of prime numbers and let us choose $A_j=p_jA_j'$ where $A_j'=1$ if $n=2$ and
\[ A_j'=p_{j-(n-3)}p_{j-(n-4)}\cdots p_{j}\in\N^*, \quad n \geq 3.\]
Our goal is to prove the following proposition.

\begin{proposition}\label{pertu} 
Let $n\geq 2$ and $s>0$. Then for all $j \in \N$, there exist a sequence of functions $g_j\in \mathcal{U}_s(\A^{n-1})$, a sequence of points $a_j\in\T^{n-1}$ and positive constants $c_1$ and $s'$ depending only on $s$ and $M$ such that if
\[ B_j=2\left[c_1c^{2(n-1)}A_j\exp(2\Omega(s'p_j))+1\right], \quad q_j=A_jB_j, \]
then $a_j$ is $q_j$-periodic for $G_j$ and $(g_j, G_j, a_j, q_j)$ satisfy the synchronization conditions (\ref{sync}): 
\begin{equation*}
g_j(a_j)=1, \quad dg_j(a_j)=0, \quad g_j(G_j^k(a_j))=0, \quad dg_j(G_j^k(a_j))=0, 
\end{equation*} 
for $1 \leq k \leq q_j-1$. Moreover, the estimate
\begin{equation}\label{estimgn} 
q_j^{-1}|g_j|_{s}\leq A_{j}^{-2}, 
\end{equation}
holds true.
\end{proposition}  

We claim that to prove Theorem~\ref{diffconvex}, it is enough to justify Proposition~\ref{pertu}. Indeed, it is rather easy to obtain a discrete version of Theorem~\ref{diffconvex} for maps on $\A^{n}$, using the prime number Theorem which ensures that
\[ A_j^{\frac{1}{n-1}} \leq p_j \leq 2A_j^{\frac{1}{n-1}}   \]
for $j$ large enough, and hence, setting $\varepsilon_j=A_j^{-2}$
\[ \varepsilon_j^{-\frac{1}{2(n-1)}} \leq p_j \leq 2\varepsilon_j^{-\frac{1}{2(n-1)}}. \]
Then by an elementary suspension argument (using the existence of bump functions for non-quasi-analytic ultra-differentiable functions) one arrives at Theorem~\ref{diffconvex} with a quasi-convex integrable Hamiltonian in $n+1$ degrees of freedom. This is described in details in~\cite{MS02} for the Gevrey case and works exactly in the same way in our context, so we shall not give more details and concentrate on the proof of Proposition~\ref{pertu}. 

We first consider the simple pendulum
\[ P(\theta,I)=\demi I^2 + V(\theta), \quad (\theta,I)\in\A. \]
With our convention, the stable equilibrium is at $(0,0)$ and the unstable one is at $(0,1/2)$. Given any $B\in\N^*$, there is a unique point $b^B=(0,I_B)$ which is $B$-periodic for $\Phi^P$ (this is just the intersection between the vertical line $\{0\} \times \R$ and the closed orbit for the pendulum of period $B$). One can check that $I_B \in\, ]2,3[$ and as $B$ goes to infinity, $(0,I_B)$ tends to the point $(0,2)$ which belongs to the upper separatrix. Since $P_j(\theta,I)=\demi I^2 +A_j^{-2}V(\theta)$, then one can see that
\[ \Phi^{P_{j}}=\sigma_j^{-1} \circ \Phi^{A_{j}^{-1}P} \circ \sigma_j, \] 
where $\sigma_{j}(\theta,I)=(\theta,A_j I)$ is the rescaling by $A_j$ in the action components. Therefore the point $b_j^B=(0,A_{j}^{-1}I_{B})$ is $q_j$-periodic for $\Phi^{P_{j}}$, for $q_j=A_jB$. Let $(\Phi_t^{P})_{t \in \R}$ be the flow of the pendulum, and 
\[ \Phi_t^{P}(0,I_{B})=(\theta_{B}(t),I_{B}(t)). \]
The function $\theta_{B}(t)$ is analytic. The crucial observation is the following simple property of the pendulum 

\begin{lemma}
Let $\vartheta=-\demi +\frac{2}{\pi}\arctan e^{\pi} < \frac{1}{2}$. For any $B \in \N^*$,  
\[ \theta_{B}(t) \notin [-\vartheta,\vartheta],  \]
for $t\in[1/2,B-1/2]$.
\end{lemma} 

Hence no matter how large $B$ is, most of the points of the orbit of $b^B\in\A$ will be outside the set $\{-\vartheta \leq \theta \leq \vartheta\}\times \R$. The construction of a function that vanishes, as well as its first derivative, at these points, will be easily arranged by means of a function, depending only on the angle variables, with support in $\{-\vartheta \leq \theta \leq \vartheta\}$. 

As for the other points, it is convenient to introduce the function 
\[ \tau_{B} : [-\vartheta,\vartheta] \longrightarrow \, ]-1/2,1/2[ \] 
which is the analytic inverse of $\theta_{B}$. One can give an explicit formula for this map: 
\[ \tau_{B}(\theta)=\int_{0}^{\theta}\frac{d\varphi}{\sqrt{I_B^2-4\sin^2 \pi\varphi}}. \]
In particular, it is analytic and therefore it belongs to $\mathcal{U}_{s}([-\vartheta,\vartheta])$ for any sequence $M:=(M_l)_{l \in \N}$ and any $s>0$, and one can obtain the following estimate (see Lemma 2.3 in \cite{MS02} for a proof).

\begin{lemma}\label{lemmelambda}
For any sequence $M:=(M_l)_{l \in \N}$ and any $s>0$, 
\[ \Lambda=\Lambda(M,s)=\sup_{B\in\N^*}|\tau_B|_{s}<+\infty. \]
\end{lemma}

Under the action of $\tau_{B}$, the points of the orbit of $b^B$ whose projection onto $\T$ belongs to $\{-\vartheta \leq \theta \leq \vartheta\}$ get equi-distributed, and we can use the following elementary lemma, which is the only place where the regularity of the function actually come into play.

\begin{lemma}\label{funct}
For $p \in \N^*$, the analytic function $\eta_p : \T \rightarrow \R$ defined by
\[ \eta_p(\theta)=\left(\frac{1}{p}\sum_{l=0}^{p-1}\cos 2\pi l\theta \right)^2 \]
satisfies
\[ \eta_p(0)=1, \quad \eta_p'(0)=0, \quad \eta_p(k/p)=\eta_p'(k/p)=0, \]
for $1 \leq k \leq p-1$, and 
\[ |\eta_p|_{s}\leq c^2\exp(2\Omega(8\pi p s)). \]
\end{lemma}

\begin{proof}
Let
\[ \xi_p(\theta)=\frac{1}{p}\sum_{l=0}^{p-1}\cos 2\pi l\theta.  \]
For any $l \in \N$, we have 
\[ \sup_{\theta \in \T}|\xi_p^{(l)}| \leq (2\pi p)^l \]
hence
\begin{equation*}
|\xi_p|_{s}=c\sup_{l \in \N}\frac{(l+1)^2s^l (2\pi p)^l }{M_l} \leq c \sup_{l \in \N}\frac{(8\pi p s)^l}{M_l}= c\exp(\Omega(8\pi p s))
\end{equation*}
and therefore
\[ |\eta_p|_{s}\leq |\xi_p|_{s}^2 \leq c^2\exp(2\Omega(8\pi p s)). \]
\end{proof}

We can now conclude the proof of Proposition~\ref{pertu}.

\begin{proof}[Proof of Proposition~\ref{pertu}]
  Given $s>0$, consider a bump function
  $\varphi_{s,\vartheta}\in \mathcal{U}_s(\T)$ which vanishes
  identically outside $]-\vartheta,\vartheta[$ and let
  $c_1=|\varphi_{s,\vartheta}|$. We choose our function
  $g_j\in \mathcal{U}_{s}(\A^{n-1})$, depending only on the angle
  variables, of the form
\[ g_j=g_j^{(2)}\otimes \cdots \otimes g_j^{(n)}, \]
where
\[ g_j^{(2)}(\theta_2)=\eta_{p_{j}}(\tau_{B_j}(\theta_2))\varphi_{s,\vartheta}(\theta_2), \]
and
\[ g_j^{(i)}(\theta_i)=\eta_{p_{j-(n-i)}}(\theta_i), \quad 3\leq i \leq n.  \]
Now we choose our point $a_j=(a_j^{(2)},\dots,a_j^{(n)})\in\A^{n-1}$. We set
\[ a_j^{(2)}=b_j^{B_j}=(0,A_{j}^{-1}I_{B_j}),  \]
where
\[ B_j=2\left[c_1c^{2(n-1)}A_j\exp(2\Omega(s'p_j))+1\right] \]
and
\[ a_j^{(i)}=(0,p_{j-(n-i)}^{-1}), \quad 3\leq i \leq n. \]
Using the definition of $A_j$, it is not hard to check that $a_j$ is $q_j$-periodic for the map $G_j$.   

Now let us show that the synchronization conditions~(\ref{sync}) hold true, that is
\[ g_j(a_j)=1, \quad dg_j(a_j)=0, \quad g_j(G_j^k(a_j))=0, \quad dg_j(G_j^k(a_j))=0, \]
for $1\leq k\leq q_j-1$. Since $\varphi_{s,\vartheta}(0)=1$, then
\[ g_j(a_j)=g_j^{(2)}(0)\cdots g_j^{(n)}(0)=1 \]
and as $\varphi_{s,\vartheta}'(0)=0$, then 
\[ dg_j(a_j)=0. \]
To prove the other conditions, let us write $G_j^k(a_j)=(\theta_k,I_k)\in\A^{n-1}$, for $1\leq k\leq q_j-1$. 

If $\theta_k^{(2)}$ does not belong to $]-\vartheta,\vartheta[$, then $g_j^{(2)}$ and its first derivative vanish at $\theta_k^{(2)}$ because it is the case for $\varphi_{s,\vartheta}$, so 
\[ g_j(\theta_k)=dg_j(\theta_k)=0. \]
Otherwise, if $-\vartheta<\theta_k^{(2)}<\vartheta$, one can easily check that
\[ - \frac{A_j-1}{2}\leq k \leq \frac{A_j-1}{2} \]
and therefore
\[ \tau_{B_j}(\theta_k^{(2)})=\frac{k}{A_j},  \]
while
\[ \theta_k^{(i)}=\frac{k}{p_{j-(n-i)}}, \quad 3\leq i \leq n. \]
If $A_j'$ divides $k$, that is $k=k'A_j'$ for some $k'\in\Z$, then
\[ \tau_{B_j}(\theta_k^{(2)})=\frac{k}{A_j}=\frac{k'}{p_{j}} \]
and therefore, by Lemma~\ref{funct}, $\eta_{p_{j}}$ vanishes with its differential at $\theta_k^{(2)}$, and so does $g_j^{(2)}$. Otherwise, $A_j'$ does not divide $k$ and then, for $3\leq i \leq n$, at least one of the functions $\eta_{p_{j-(n-i)}}$ vanishes with its differential at $\theta_k^{(2)}$, and so does $g_j^{(i)}$. Hence in any case
\[ g_j(\theta_k)=dg_j(\theta_k)=0, \quad 1\leq k\leq q_j-1, \]
and the synchronization conditions~(\ref{sync}) are satisfied.

Now it remains to estimate the norm of the function $g_j$. First, by definition of $c_1$ and using Proposition~\ref{produit}, one finds 
\[ |g_j|_{s} \leq c_1 |\eta_{p_{j}}\circ\tau_{B_j}|_{s}|\eta_{p_{j-(n-3)}}|_{s} \cdots |\eta_{p_{j}}|_{s}.  \]
Then, using Lemma~\ref{lemmelambda}, Proposition~\ref{composition} and Lemma~\ref{funct} we obtain
\[ |\eta_{p_{j}}\circ\tau_{B_j}|_{s} \leq c^2\exp(2\Omega(s' p_j)) \]
for some $s'$ which depends only on $s$ and $\Lambda$. By Lemma~\ref{funct} for $3 \leq i \leq n$, since $p_{j-(n-i)} \leq p_j$, we have
\[ |\eta_{p_{j-(n-i)}}|_{s} \leq c^2\exp(2\Omega(s' p_j )) \]
and therefore
\[ |g_j|_{s} \leq c_1 c^{2(n-1)}\exp(2(n-1)\Omega(s' p_j )) \]
Finally, by definition of $B_j$ we obtain
\begin{equation*}
|g_j|_{s}\leq B_jA_j^{-1},
\end{equation*}
and as $q_j=A_jB_j$, we end up with
\[ q_j^{-1}|g_j|_{s}\leq A_j^{-2}. \]
This concludes the proof.
\end{proof}

\subsection{Stability in the steep case: proof of Theorem~\ref{steep}}\label{sec:steep}

The goal of this section is to prove Theorem~\ref{steep}; implicit constants will depend only on $n$, $|h|_s$, $l$, $L$, $\delta$, $p$ and the function $C$.

Given positive real parameters $s,\delta,\mu,\nu$, consider a Hamiltonian of the form 
\begin{equation}\label{HAMN}
\begin{cases}\tag{$\#$}
H(\theta,I)=L_v(I)+g(I)+f(\theta,I), \\
H : \T^n \times D_{\delta} \rightarrow \R \\
|\nabla g|_{s} \leq \rho, \quad |f|_{s} \leq \nu, \\
\end{cases}
\end{equation}
where $\{\,.\,\}$ denotes the usual Poisson bracket, $D_{\delta} \subset \R^n$ is the ball of radius $\delta$ around the origin. Observe that this is nothing but a special case of~\eqref{HamN}, up to a change of notations, and thus the following proposition is just a special case of Proposition~\ref{propperiodic}, written in a slightly different way using a real parameter $K \geq 1$.

\begin{proposition}\label{propMS}
Let $H$ be as in~\eqref{HAMN}, with $M$ satisfying \eqref{H1} and \eqref{H2} and $v$ a $T$-periodic vector. Given $K \geq 1$, assume that
\begin{equation}\label{seuil1MS}
\nu \PM s\rho, \quad T\rho \PM s, \quad KT\rho \PM 1.
\end{equation}
Then there exists a $(M,s/2)$-ultra-differentiable symplectic transformation
\[ \Psi : \T^n \times D_{\delta/2} \rightarrow \T^n \times D_{\delta} \]
such that
\[ H\circ\Psi =L_v+g+\bar{f}+\hat{f} \]
with $\{\bar{f},L_v\}=0$ and the estimates
\begin{equation}\label{EST1MS}
|\Psi-\mathrm{Id}|_{s/2}\leq 2 T\nu, \quad |\bar{f}|_{s/2} \leq 2\nu, \quad |\hat{f}|_{s/2} \leq \nu\exp\left(-\kappa_2\left(C^{-1}(Ks)\right)^{-1}\right). 
\end{equation}
Moreover, if $I$ is an integrable Hamiltonian such that $\{I,F\}=0$, then $\{I,\bar{F}\}=0$.
\end{proposition}

Next consider $T_i$-periodic vectors $v_i$ for every $1 \leq i \leq n$, and assume that they are linearly independent. For simplicity, let us just write $L_{i}(I)=v_i\cdot I$. Now consider the following Hamiltonian: 
\begin{equation}\label{HAMNN}
\begin{cases}\tag{$\#_j$}
H_j(\theta,I)=L_{j}(I)+g_j(I)+f(\theta,I), \\
H_j : \T^n \times D_{\delta_j} \rightarrow \R \\
|\nabla g_j|_{s_j} \leq \rho_j, \quad |f|_{s_j} \leq \nu
\end{cases}
\end{equation}
where we set $\delta_j=2^{j-1}$ and $s_j=s\delta_j$. 

\begin{proposition}\label{propMS2}
Let $1 \leq j \leq n$, $H_j$ be as in~\eqref{HAMNN} with $M$ satisfying \eqref{H1} and \eqref{H2}. Given $K \geq 1$, assume that
\begin{equation}\label{seuilMS2}
\begin{cases}
\nu \PM s\rho_i, \quad T_i\rho_i \PM s, \quad KT_i\rho_i \PM 1 \quad 1 \leq i \leq j, \\
|v_j-v_{j-1}| \MP \rho_{j-1}, \quad \rho_j \PM \rho_{j-1}, \quad j \geq 2. 
\end{cases} 
\end{equation}
Then there exist a $(M,s/2)$-ultra-differentiable symplectic transformation
\[ \Psi_j : \T^n \times D_{1/2} \rightarrow \T^n \times D_{\delta_j} \]
such that
\[ H_j\circ\Psi_j =L_j+g+\tilde{f}+f^+ \]
with $\{\tilde{f},L_i\}=0$ for $1 \leq i \leq j$ and the estimates
\begin{equation}\label{EST2MS}
|\Psi_j-\mathrm{Id}|_{s/2}\MP  \max_{1 \leq i \leq j}\{T_i\nu\}, \quad |\tilde{f}|_{s/2} \MP \nu, \quad |f^+|_{s/2} \MP \nu\exp\left(-\kappa_2\left(C^{-1}(Ks)\right)^{-1}\right). 
\end{equation}
Moreover, if $I$ is an integrable Hamiltonian such that $\{I,f\}=0$, then $\{I,\tilde{f}\}=0$.
\end{proposition}

\begin{proof}
For $j=1$, this is nothing but Proposition~\ref{propMS} (setting $\tilde{f}=\bar{f}$ and $f^+=\hat{f}$) so we may argue by induction. Thus we assume that the statement is true for some $1 \leq j-1 \leq n-1$, and we need to show that it remains true for $2 \leq j \leq n$.

To the Hamiltonian $H_j=L_j+g_j+f$ we can apply Proposition~\ref{propMS} (with $L_j,g_j,s_j,\delta_j,\rho_j$ instead of $L_{v},g,s,\delta,\rho$) as~\eqref{EST2MS} implies~\eqref{EST1MS}: since $s_j/2=s_{j-1}$ and $\delta_j/2=\delta_{j-1}$, we find a $(M,s_{j-1})$-ultra-differentiable symplectic transformation $\Psi^{j} : \T^n \times D_{\delta_{j-1}} \rightarrow \T^n \times D_{\delta_j}$ such that
\[ H_j\circ\Psi^j=L_j+g_j+\bar{f}+\hat{f} \]
with $\{\bar{f},L_j\}=0$ and with the estimates
\[ |\Psi^j-\mathrm{Id}|_{s_{j-1}}\MP T_j\nu, \quad |\bar{f}|_{s_{j-1}} \MP \nu, \quad |\hat{f}|_{s_{j-1}} \MP \nu\exp\left(-\kappa_2\left(C^{-1}(Ks)\right)^{-1}\right).  \] 
Moreover, if $I$ is an integrable Hamiltonian such that $\{f,I\}=0$, then $\{\bar{f},I\}=0$. 

Then observe that we can write
\[ H_j\circ\Psi^j-\hat{f}=L_j+g_j+\bar{f}=L_{j-1}+g_{j-1}+\bar{f}, \quad g_{j-1}:=L_j-L_{j-1}+g_j,\]
where 
\[ |\bar{f}|_{s_{j-1}}\MP \nu\] 
and, using~\eqref{seuilMS2}, 
\[ |\nabla g_{j-1}|_{s_{j-1}}\leq |\nabla (L_j-L_{j-1})|_{s_{j-1}}+|\nabla g_{j}|_{s_{j-1}} \MP \rho_{j-1}+\rho_j \MP \rho_{j-1}.\]
Thus we can apply our inductive assumption (with $\bar{f}$ instead of $f$), and there exists a $(M,s/2)$-ultra-differentiable symplectic transformation 
\[ \Psi_{j-1} : \T^n \times D_{1/2} \rightarrow \T^n \times D_{\delta_{j-1}}\] 
such that
\[ (H\circ\Psi^j-\hat{f})\circ\Psi_{j-1}=L_{j-1}+g_{j-1}+\tilde{\bar{f}}+\bar{f}^+=L_j+g_j+\tilde{\bar{f}}+\bar{f}^+\]
with $\{\tilde{\bar{f}},L_i\}=0$ for $1 \leq i \leq j-1$ and the estimates
\[ |\Psi_{j-1}-\mathrm{Id}|_{s/2} \MP \max_{1 \leq i \leq j-1}\{T_i\nu\} , \quad |\tilde{\bar{f}}|_{s/2} \MP \nu, \quad |\bar{f}^+|_{s/2} \MP \nu\exp\left(-\kappa_2\left(C^{-1}(Ks)\right)^{-1}\right).  \]
Moreover, if $I$ is an integrable Hamiltonian such that $\{\bar{f},I\}=0$, then $\{\tilde{\bar{f}},I\}=0$. We may then set
\[ \Psi_j:=\Psi^j \circ \Psi_{j-1}, \quad \tilde{f}:=\tilde{\bar{f}}, \quad f^+:=\bar{f}^++\hat{f}\circ \Psi_{j-1}. \]
Obviously we have 
\[ |\tilde{f}|_{s/2} \MP \nu. \]
Using Proposition~\ref{composition}, we get
\[ |\Psi_{j}-\mathrm{Id}|_{s/2} \leq |\Psi^{j}-\mathrm{Id}|_{s_{j-1}} + |\Psi_{j-1}-\mathrm{Id}|_{s/2} \MP \max_{1 \leq i \leq j}\{T_i\nu\} \]
but also
\[ |f^+|_{s/2} \leq |\bar{f}^+|_{s/2}+|\hat{f}\circ \Psi_{j-1}|_{s/2} \leq |\bar{f}^+|_{s/2}+|\hat{f}|_{s_{j-1}} \MP \nu\exp\left(-\kappa_2\left(C^{-1}(Ks)\right)^{-1}\right).  \]
From the first part of the inductive assumption, we know that $\{\tilde{f},L_i\}=0$ for $1 \leq i \leq j-1$, but since $\{\bar{f},L_j\}=0$, from the second part of the inductive assumption (applied to $S=L_j$) we obtain $\{\tilde{f},L_j\}=0$ and hence $\{\tilde{f},L_i\}=0$ for $1 \leq i \leq j$. This proves the first part of the statement, but the second part is obvious, as $\{f,I\}=0$ implies $\{\bar{f},I\}=0$ which implies $\{\tilde{f},I\}=0$.            
\end{proof}

Now let us come back to a Hamiltonian as in~\eqref{Ham1}. For any $1 \leq j \leq n$, let $I_j \in D_{1/2}$, and as before, we still consider a $T_j$-periodic vector $v_j$ and a positive parameter $\rho_j$.

\begin{proposition}\label{propMS3}
Let $1 \leq j \leq n$ and $H$ be as in~\eqref{Ham1}, and given $K \geq 1$, assume that
\begin{equation}\label{seuilMS3}
\begin{cases}
|\nabla h(I_j)-v_j|\leq \rho_j, \quad \rho_j \PM 1, \\
\varepsilon \PM \rho_i^2, \quad T_i\rho_i \PM s^2, \quad KT_i\rho_i \PM 1, \quad 1 \leq i \leq j, \\
|\omega_j-\omega_{j-1}| \PM \rho_{j-1}, \quad \rho_j \PM \rho_{j-1}, \quad 2\leq j. 
\end{cases} 
\end{equation}
Then there exists a smooth symplectic transformation
\[ \Phi_j : \T^n \times D_{\rho_j/2}(I_j) \rightarrow \T^n \times D_{\delta_j\rho_j}(I_j) \]
such that
\[ H\circ\Phi_j=h+f'+f^* \]  
with $\{f',L_i\}=0$ for $1 \leq i \leq j$ and with the estimate 
\begin{equation}\label{estMS3}
\begin{cases}
|\Pi_I\Phi_j-\mathrm{Id}_I|_{C^0(\T^n \times D_{\rho_j/2}(I_j))} \MP \max_{1 \leq i \leq j}\{T_i\} \varepsilon, \\ 
|\partial_\theta f^*|_{C^0(\T^n \times D_{\rho_j/2}(I_j))} \MP \varepsilon s^{-1}\exp\left(-\kappa_2\left(C^{-1}(Ks)\right)^{-1}\right).
\end{cases}
\end{equation}
\end{proposition} 

Let us point out that one can obtain better (anisotropic) estimates such as those obtained in Proposition~\ref{propnonlineaire} or Proposition~\ref{propconvex1}; yet only the estimates~\eqref{estMS3} will be useful in the sequel. 

\begin{proof}
Observe that the assumptions $\rho_j \PM 1$ ensures that $D_{\delta_j\rho_j}(I_j)$ is contained in $D$.  Without loss of generality, we may assume that $I_j=0$ and $h(0)=0$, and we set $\omega:=\nabla h(0)$. Let $\sigma_j(\theta,I):=(\theta,\rho_j I)$: expanding $h \circ \sigma_j$ at order $2$ around $I=0$ we have
\[ h \circ \sigma_j(I)=\rho_j\omega\cdot I+\rho_j^2h_j(I), \quad I \in D_{\delta_j}, \]
with $|\nabla h_j|_{s/2} \MP |\nabla h|_{s/2} \MP s^{-1}$. Now consider $H_j:=\rho_j^{-1}H \circ \sigma_j$; it can be written as
\[ H_j(\theta,I)=\omega\cdot I+\rho_j h_j(I)+\rho_j^{-1}f\circ \sigma_j(\theta,I)=\omega_j\cdot I+g_j(I)+f_j \]
with 
\[ g_j(I):=(\omega-v_j)\cdot I+\rho_jh_j(I), \quad f_j(I):\rho_j^{-1}f\circ \sigma_j(\theta,I)\]
satisfying the estimates
\[ |g_j|_{s/2} \leq s^{-1}\rho_j, \quad |g_j|_{s/2} \leq \rho_j^{-1}\varepsilon.  \] 
To this Hamiltonian we can then apply Proposition~\ref{propMS2}, with $s/2$ instead of $2^{j-1}s$, $s^{-1}\rho_j$ instead of $\rho_j$ and $\nu=\rho_j^{-1}\varepsilon$, as~\eqref{seuilMS3} implies~\eqref{seuilMS2}: for $s'\PE s$, we obtain a $(M,s')$-ultra differentiable symplectic transformation 
\[ \Psi_j : \T^n \times D_{1/2} \rightarrow \T^n \times B_{\delta_j}\] 
such that
\[ H_j\circ\Psi_j=L_j+g_j+\tilde{f}+f^+ \]
with $\{\tilde{f},L_i\}=0$ for $1 \leq i \leq j$ and $\Psi_j$ and $f^+$ satisfying the estimates~\eqref{EST2MS} with $s^{-1}\rho_j$ instead of $\rho_j$ and $\nu=\rho_j^{-1}\varepsilon$. We can eventually set
\[ \Phi_j:=\sigma_j\circ \Psi_j \circ \sigma_j^{-1}, \quad f':=\rho_j\tilde{f}\circ \sigma_j^{-1}, \quad f^*:=\rho_jf^+\circ \sigma_j^{-1}. \]
One easily check that $H\circ\Phi_j=h+f'+f^*$, that $\{f',L_i\}=0$ for $1 \leq i \leq j$ and the wanted estimates~\eqref{estMS3} on $\Phi_j$ and $f^*$ follows easily from the estimates on $\Psi_j$ and $f^+$ given by~\eqref{EST2MS}.   
\end{proof}

For $1 \leq j \leq n$, let $\Lambda_j \subseteq \R^n$ be the real subspace orthogonal to $v_1,\dots, v_j$: since these vectors are linearly independent, $\Lambda_j$ has dimension $n-j$. Let also $\Pi_{\Lambda_j}$ (respectively $\Pi_{\Lambda_j^\perp}$) be the orthogonal projection onto $\Lambda_j$ (respectively $\Lambda_j^\perp$).

\begin{proposition}\label{dicho}
Let $H_j:=H \circ \Phi_j$ be the Hamiltonian given by Proposition~\ref{propMS3}, and let $(\theta(t),I(t))$ be a solution of the system associated to $H_j$ with $I(0) \in D_{\rho_j/4}(I_j)$. Then for $1 \leq j \leq n-1$, we have the following dichotomy:
\begin{itemize}
\item[$(1)$] either $I(t) \in D_{\rho_j/2}(I(0))$ for $0 \leq t \leq s\exp\left(\kappa_2\left(C^{-1}(Ks)\right)^{-1}\right)$,
\item[$(2)$] or there exists a time $0<t^+ \leq s\exp\left(\kappa_2\left(C^{-1}(Ks)\right)^{-1}\right)$ for which $|I(t^+)-I(0)|=\rho_j/2$ and 
\[ |I(t)-I(0)|\leq \rho_j/2, \quad |\Pi_{\Lambda_j^\perp}(I(t)-I(0))| \MP \varepsilon, \quad 0 \leq t \leq t^+.  \]  
\end{itemize}
For $j=n$, we have $I(t) \in D_{\rho_n/2}(I(0))$ for $0 \leq t \leq s\exp\left(-\kappa_2\left(C^{-1}(Ks)\right)^{-1}\right)$.
\end{proposition}

\begin{proof}
  First assume $1 \leq j \leq n-1$. Let $t^+>0$ be the time of first
  exit of $I(t)$ from $D_{\rho_j/2}(I(0))$. Then either
  $t^+ > s\exp\left(\kappa_2\left(C^{-1}(Ks)\right)^{-1}\right)$, in
  which case $(1)$ clearly holds true, or
  $t^+ \leq s\exp\left(\kappa_2\left(C^{-1}(Ks)\right)^{-1}\right)$,
  in which case we will show that $(2)$ holds true. The facts that
  $|I(t^+)-I(0)|=\rho_j/2$ and $|I(t)-I(0)|\leq \rho_j/2$ for
  $0 \leq t \leq t^+$ are obvious. Then since $H\circ\Phi_j=h+f'+f^*$
  with $\{f',L_i\}=0$ for $1 \leq i \leq j$, we have
  $\partial_\theta f' \in \Lambda_j$, and thus for
  $0 \leq t \leq t^+$,
  \[ |\Pi_{\Lambda_j^\perp}(I(t)-I(0))| \leq |t| |\partial_\theta
  f^*|_{C^0(\T^n \times B_{\rho_j}(I_j))} \MP t^+ \varepsilon
  s^{-1}\exp\left(-\kappa_2\left(C^{-1}(Ks)\right)^{-1}\right) \MP
  \mu.  \]
  For $j=n$, since $\{f',L_j\}=0$ for $1 \leq j \leq n$, the
  Hamiltonian $f'$ is integrable and the last estimate shows that
  $t^+ \geq s\exp\left(\kappa_2\left(C^{-1}(Ks)\right)^{-1}\right)$;
  then the conclusion follows from $\mu \PM \rho_n$.
\end{proof}

Proposition~\ref{propMS3} and Proposition~\ref{dicho}, which consist
respectively in the construction of normal forms and their dynamical
consequences, constitute the analytical part of the proof. The
geometrical part will be based on two propositions. The first one is
Proposition~\ref{dirichlet}, which we already used
in \S~\ref{sec:convex} and is a mere consequence of Dirichlet's box
principle. The second one is a consequence of the steepness property
of $h$; the statement below is a special case of the lemma on ``almost
plane curves" of Nekhoroshev, stated in \cite{Nek77} and proved in
\cite{Nek79} (our case below corresponds to ``plane curves").

\begin{proposition}\label{nekho}
  Let $h : G \rightarrow \R$ be a smooth function which is
  $(l,L,\delta,p)$-steep, and such that
  $\sup_{I \in G}|\nabla^2 h (I)| \MP 1$. Let
  $\gamma : [0,t^*] \rightarrow \R^n$ be a continuous curve, $\lambda$
  a proper affine subspace of $\R^n$ and $\varrho>0$. Assume that
\begin{itemize}
\item[$(i)$] for all $t \in [0,t^*]$, $\gamma(t) \in \lambda
$;
\item[$(ii)$] for all $t \in [0,t^*]$, $|\gamma(0)- \gamma(t)|\leq \varrho$ and $|\gamma(0)-\gamma(t^*)|=\varrho$;
\item[$(iii)$] the ball $\{I \in \R^n \; | \; |I-\gamma(0)|\leq \varrho\}$ is contained in $G$;
\item[$(iv)$] $\varrho \PM 1$,
\end{itemize}
then there exists a time $\tilde{t} \in [0,t^*]$ such that 
\[ ||\Pi_\Lambda \nabla h(\gamma(\tilde{t}))|| \PS \varrho^p, \]
where $\Lambda$ is the vector space associated to $\lambda$.
\end{proposition}

We can now conclude the proof of Theorem~\ref{steep}.

\begin{proof}[Proof of Theorem~\ref{steep}]
For $1 \leq j \leq n$, let us set
\[ a_j:=(np)^{n-j}, \quad Q_j \PE \varepsilon^{-\frac{1}{2na_j}} \]
and
\[ K \PE Q_1 \PE \varepsilon^{-\frac{1}{2na_1}}. \]
Given any solution $(\theta(t),I(t))$ of the Hamiltonian system associated to $H$ with $I(0) \in D_{1/2}$, we claim that
\[ |I(t)-I(0)|\MP Q_1^{-1}, \quad |t|\leq s\exp\left(\kappa_2\left(C^{-1}(Ks)\right)^{-1}\right)\]
provided that 
\begin{equation}\label{SE}
\varepsilon \PM 1, \quad Q_1^{-1} \PM s^2. 
\end{equation}
which holds true if $\varepsilon$ is sufficiently small.

In view of the definitions of $Q_1$, $K$ and the fact that $a_1=a$, this claim clearly implies the statement we want to prove for positive times; for negative times the argument is entirely similar. So it remains to prove the claim, and we will do this by an algorithm that stops after at most $n$ steps. Let us write
\[ \tau_K:=s\exp\left(\kappa_2\left(C^{-1}(Ks)\right)^{-1}\right). \]

For the first step, we set $I_1:=I(0)$ and we apply Proposition~\ref{dirichlet} to $\omega:=\nabla h(I_1)$ with $Q=Q_1$: we find a $T_1$-periodic vector $v_1 \in \R^n\setminus\{0\}$ such that
\[ |\nabla h(I_1)-v_1|\MP (T_1Q_1)^{-1}, \quad 1 \MP T_1 \MP Q_1^{n-1}.  \]
Let us set $\rho_1 \EP (T_1Q_1)^{-1}$. The conditions~\eqref{SE} imply~\eqref{seuilMS3} for $j=1$ and hence Proposition~\ref{propMS3} can be applied: let $H_1:=H \circ \Phi_1$ be the new Hamiltonian, and let us denote by $(\theta^1(t),I^1(t))$ its solutions. Consider the solution for which $\Phi_1((\theta^1(0),I^1(0)))=(\theta(0),I(0))$: in view of the estimate on $\Phi_1$ given in Proposition~\ref{propMS3}, $I^1(0) \in D_{\rho_1/4}(I_1)$ and Proposition~\ref{dicho} gives a dichotomy:
\begin{itemize}
\item[$(1)$] either $I^1(t) \in D_{\rho_1/2}(I^1(0))$ for $0 \leq t \leq \tau_K$,
\item[$(2)$] or there exists a time $0<t_1^+\leq \tau_K$ for which $|I^1(t_1^+)-I^1(0)|=\rho_1/2$ and 
\[ |I^1(t)-I^1(0)|\leq \rho_1/2, \quad |\Pi_{\Lambda_1^\perp}(I^1(t)-I^1(0))| \MP \varepsilon, \quad 0 \leq t \leq t_1^+.  \]
\end{itemize}  
If the first alternative holds true, the claim is proved since this implies
\[ |I(t)-I(0)|\leq 2\rho_1 \MP Q_1^{-1}, \quad 0 \leq t \leq \tau_K. \]
If the second alternative holds true, we will show how the algorithm moves to the second step. Consider the curve
\[ \gamma_1(t):=I^1(0)+\Pi_{\Lambda_1}\left(I^1(t)-I^1(0)\right). \]
Since $\varepsilon \PM \rho_1$, we can ensure that
\[ |\gamma_1(t_1^+)-\gamma_1(0)|=|\Pi_{\Lambda_1}\left(I^1(t_1^+)-I^1(0)\right)| \geq |I^1(t_1^+)-I^1(0)| - |\Pi_{\Lambda_1^\perp}(I^1(t_1^+)-I^1(0))|\geq\rho_1/4  \]
and therefore we can find a time $0<t_1^*<t_1^+$ for which
\begin{equation*}
\begin{cases}
|\gamma_1(t)-\gamma_1(0)|\leq \rho_1/4, \quad 0 \leq t \leq t_1^*, \\ 
|\gamma_1(t_1^*)-\gamma_1(0)|=\rho_1/4.
\end{cases}
\end{equation*}
This curve $\gamma_1$, on the interval $[0,t_1^*]$, satisfies the assumption of Proposition~\ref{nekho} with $\lambda:=I^1(0)+\Lambda_1$, $\varrho=\rho_1/4$ and therefore we can find a time $\tilde{t}_1 \in [0,t_1^*]$ such that 
\[ |\Pi_{\Lambda_1} \nabla h(\gamma_1(\tilde{t}_1))| \PS \rho_1^p. \]
Moreover,
\[ |\gamma_1(\tilde{t}_1)-I^1(\tilde{t}_1)|=|\Pi_{\Lambda_1^\perp}\left(I^1(\tilde{t}_1)-I^1(0)\right)|\MP \varepsilon \]
and since $\varepsilon \PM \rho_1^{p}$ as one easily check, we find that
\[ |\Pi_{\Lambda_1} \nabla h(I^1(\tilde{t}_1))| \PS \rho_1^p. \]
Then, using the estimate on $\Phi_1$ given in Proposition~\ref{propMS3} and the fact that $T_1\varepsilon \PM \rho_1^{p}$, one also has
\[ |I(\tilde{t}_1)-I^1(\tilde{t}_1)| \PM \rho_1^{p}  \]
and thus
\[ |\Pi_{\Lambda_1} \nabla h(I(\tilde{t}_1))| \PS \rho_1^p. \]
Now we set $I_2:=I(\tilde{t}_1)$, and we apply again Proposition~\ref{dirichlet} to $\omega:=\nabla h(I_2)$ with $Q=Q_2$: we find a $T_2$-periodic vector $v_2 \in \R^n\setminus\{0\}$ such that
\[ |\nabla h(I_2)-v_2|\MP (T_2Q_2)^{-1}, \quad 1 \MP T_2 \MP Q_2^{n-1}.  \]
Let us set $\rho_2 \EP (T_2Q_2)^{-1}$. To check that $v_2$ is linearly independent from $v_1$, observe that
\[ \rho_2 \MP Q_2^{-1} \PM Q_1^{-np} \PM \rho_1^p \]
and therefore, by a proper choice of implicit constants, one can ensure that
\[ |\Pi_{\Lambda_1} v_2|\geq |\Pi_{\Lambda_1} \nabla h(I_2))|-|\Pi_{\Lambda_1}\left(\nabla h(I_2)-v_2\right)| \geq |\Pi_{\Lambda_1} \nabla h(I_2))|-|\nabla h(I_2)-v_2|>0. \]
Clearly, we have $\rho_2 \PM \rho_1$ and it is straightforward to check that $|v_1-v_2| \MP \rho_1$. Then conditions~\eqref{SE} imply~\eqref{seuilMS3} for $j=2$, and once again, we can apply Proposition~\ref{propMS3} and Proposition~\ref{dicho} leading to a new dichotomy: if the first alternative holds true, the claim is proved while if the second alternative holds true, we move to the next step. If the algorithm lasts $n-1$ steps, then only possibility (1) in Proposition~\ref{dicho} can hold and this concludes the proof of the claim.  
\end{proof} 

\appendix

\section{On the moderate growth condition}
\label{sec:MG}

Recall from \S~\ref{sec:gintro} that the sequence $M$ has
\emph{moderate growth} if
\begin{equation}
  \label{MG}%
  \sup_{l,j \in \N}
  \left(\frac{M_{l+j}}{M_lM_j}\right)^{\frac{1}{l+j}} < +\infty \tag{MG}
\end{equation}
Here we compare this condition to \eqref{H2}, which requires the sequence $\left(\frac{\ln \mu_l}{l}\right)_{l\geq 1}$ to converge to zero as $l$ goes to infinity; in particular, \eqref{H2} is satisfied if $\left(\frac{\ln \mu_l}{\ln l}\right)_{l\geq 2}$ is bounded.  

\begin{lemma}
  \label{lm:MG}%
  If conditions \eqref{H1} and \eqref{MG} hold, positive sequences
  $\left(\frac{\ln M_l}{l \ln l}\right)_{l\geq 2}$ and
  $\left(\frac{\ln \mu_l}{\ln l}\right)_{l\geq 2}$ are bounded away
  from $0$ and $+\infty$ and, in particular, \eqref{H2} is satisfied.
\end{lemma}

\begin{proof}
  As noticed in \S~\ref{sec:udiffintro}, \eqref{H1} implies 
  \[M_l \geq l!, \quad \mu_l \geq l+1 \quad (\forall l \in \N),\]
  which shows that both sequences of the statement are lower bounded.

  Now, assume that $(M_l)$ has moderate growth. In restriction to
  $j=l$, \eqref{eq:MG} reduces to
  \[M_{2l}\leq A^{l}M_l^2\]
  for some $A \in [1,+\infty[$ independant of $l\in \N$. Using that
  $M_1=1$, by induction we see that 
  \begin{equation}
    \label{eq:2k}
    M_{2^k} \leq A^{k2^{k-1}} \quad (k \in \N).
  \end{equation}
  As already noticed, \eqref{H1} implies that $M$ itself is
  log-convex, hence, if $2^k \leq l < 2^{k+1}$,
  \[M_l \leq M_{2^k}^{\frac{2^{k+1}-l}{2^k}}
  M_{2^{k+1}}^{\frac{l-2^k}{2^k}} \leq \left( M_{2^k} M_{2^{k+1}}
  \right)^{\frac{l}{2^k}} .\] 
  Using~\eqref{eq:2k}, if $k\geq 1$,
  \[M_l \leq A^{lk/2 + l(k+1)} \leq A^{5lk/2},\]
  whence
  \[M_l \leq \left(A^{\frac{5}{2\ln 2}}\right)^{l \ln l} \quad (\forall
  l \geq 2),\]
  and the sequence $\frac{\ln M_l}{l \ln l}$ is bounded. 

  Let us now see why $\rho_l := \frac{\ln \mu_l}{\ln l}$ is upper
  bounded. For any given $l \in \N$, let $k = k_l \in \N$ be such that
  $2^{k-1} \leq l < 2^k$. By the definition~\eqref{sequences} of
  $\mu$,
  \[M_{2^{k+1}} = M_{2^k} \mu_{2^k} \cdots \mu_{2^{k+1}-1}.\]
  Because of \eqref{H1}, the sequence $(M_{2^k})$ is larger than $1$
  and the sequence $\mu$ is increasing, so
  \[M_{2^{k+1}} \geq (\mu_{2^k})^{2^k} \geq (\mu_{l})^{2^k} = e^{2^k\rho_l \ln l}.\] 
  Using~\eqref{eq:2k}, we get
  \[\rho_l \ln l \leq (k+1)\ln A,\]
and as
\[ (k+1) \leq \frac{\ln l}{\ln 2}+2 \]
we have
\[ \rho_l \leq \ln A\left(\frac{1}{\ln 2}+\frac{2}{\ln l}\right) \quad (\forall l \geq 1),\]
which shows that $(\rho_l)$ must be bounded. 
\end{proof}

\textit{Acknowledgements. } The authors have benefited
from partial funding from the ANR project BeKAM (Beyond KAM Theory). 

\addcontentsline{toc}{section}{References}
\bibliographystyle{amsalpha}
\bibliography{Udiff6}

\end{document}